\documentclass{amsart}
\usepackage[utf8]{inputenc}

\usepackage[foot]{amsaddr}

\usepackage{biblatex}
\usepackage{mathtools}
\usepackage{amsmath}
\usepackage{amssymb}
\usepackage{amsthm}
\usepackage{stmaryrd}
\usepackage{hyperref}
\usepackage{cleveref}
\usepackage{listings}

\usepackage[hyphenbreaks]{breakurl}

\usepackage{xcolor}
\usepackage{dutchcal}
\usepackage{tikz-cd}
\usetikzlibrary{calc,decorations.markings,decorations.pathmorphing}
\usepackage{ifthen}
\usepackage{bm}

\newcommand{\precdot}{\prec\mathrel{\mkern-3mu}\mathrel{\cdot}}
\newcommand{\succdot}{\mathrel{\cdot}\mathrel{\mkern-3mu}\succ}

\newcommand{\ULozenge}[1]{
\draw ($#1 - 0.8*(Z)$) -- ($#1 - 0.2*(Z)$); 
\draw ($#1 + (X) + 0.2*(Y)$) -- ($#1 - (Z) - 0.2*(Y)$); 
\draw ($#1 - 0.8*(Y)$) -- ($#1 - 0.2*(Y)$); 
\draw ($#1 + (X) + 0.2*(Z)$) -- ($#1 - (Y) - 0.2*(Z)$); 
}
\newcommand{\VLozenge}[1]{
\draw ($#1 + 0.2*(X)$) -- ($#1 + 0.8*(X)$);
\draw ($#1 + (X) + 0.2*(Z)$) -- ($#1 + (X) + 0.8*(Z)$);
\draw ($#1 + 0.2*(Z)$) -- ($#1 + 0.8*(Z)$);
\draw ($#1 + (Z) + 0.2*(X)$) -- ($#1 + (Z) + 0.8*(X)$);
}
\newcommand{\WLozenge}[1]{
\draw ($#1 + 0.2*(X)$)  -- ($#1 + 0.8*(X)$);
\draw ($#1 - 0.8*(Y)$)  -- ($#1 - 0.2*(Y)$);
\draw ($#1 - (Y) + 0.2*(X)$) -- ($#1 - (Y) + 0.8*(X)$);
\draw ($#1 + (X) - 0.8*(Y)$) -- ($#1 + (X) - 0.2*(Y)$);
}

\theoremstyle{definition}
\newtheorem{Def}[]{Definition}[section]
\newtheorem{Pro}[Def]{Proposition}
\newtheorem{Rem}[Def]{Remark}
\newtheorem{Lem}[Def]{Lemma}
\newtheorem{Exp}[Def]{Example}
\newtheorem{Theo}[Def]{Theorem}
\newtheorem{Cor}[Def]{Corollary}
\newtheorem{Ques}[Def]{Question}
\newtheorem{Con}[Def]{Construction}
\newtheorem*{Ques*}{Question}
\newtheorem*{Theo*}{Theorem}
\newtheorem*{Cor*}{Corollary}
\newtheorem*{Pro*}{Proposition}

\DeclareMathOperator{\Ext}{Ext}
\DeclareMathOperator{\Hom}{Hom}
\DeclareMathOperator{\RHom}{\mathbb{R}\strut\kern-.2em\operatorname{Hom}}

\DeclareMathOperator{\Ker}{Ker}

\DeclareMathOperator{\SL}{SL}
\DeclareMathOperator{\GL}{GL}

\DeclareMathOperator{\Db}{\mathrm{D}^{\mathrm{b}}}
\DeclareMathOperator{\spec}{Spec}
\DeclareMathOperator{\Irr}{Irr}

\DeclareMathOperator{\diag}{diag}

\newcommand{\bbone}{\text{\usefont{U}{bbold}{m}{n}1}}
\MakeRobust{\bbone}

\renewcommand{\mod}{\operatorname{mod}}
\renewcommand{\Bbbk}{k}
\renewcommand{\Im}{\operatorname{Im}}

\title{A classification of $n$-representation infinite algebras of type $\Tilde{A}$}

\author[Darius Dramburg]{Darius Dramburg}
\address{\emph{Corresponding author:} Darius Dramburg, Department of Mathematics, Uppsala University, Box 480, 75106 Uppsala, Sweden}

\author[Oleksandra Gasanova]{Oleksandra Gasanova}
\address{Oleksandra Gasanova, Faculty of Mathematics, University of Duisburg-Essen, 45127 Essen, Germany}

\email{darius.dramburg@math.uu.se, oleksandra.gasanova@uni-due.de}

\subjclass{16G20, 16S35, 16E65, 14E16}

\date{\today}

\begin{document}

\begin{abstract}
    We classify $n$-representation infinite algebras $\Lambda$ of type $\Tilde{A}$. This type is defined by requiring that $\Lambda$ has higher preprojective algebra $\Pi_{n+1}(\Lambda) \simeq k[x_1, \ldots, x_{n+1}] \ast G$, where $G \leq \SL_{n+1}(k)$ is finite abelian. For the classification, we group these algebras according to a more refined type, and give a combinatorial characterisation of these types. This is based on so-called height functions, which generalise the height function of a perfect matching in a dimer model. In terms of toric geometry and McKay correspondence, the types form the lattice simplex of junior elements of $G$. We show that all algebras of the same type are related by iterated $n$-APR tilting, and hence are derived equivalent. By disallowing certain tilts, we turn this set into a finite distributive lattice, and we construct its maximal and minimal elements. 
\end{abstract}

\maketitle

\section*{Introduction}
The famous McKay correspondence, named after John McKay who first discovered it \cite{MR0604577}, gives a bijection between finite subgroups of $\SL_2(\mathbb{C})$ and simply laced Dynkin diagrams. It can be realised geometrically by studying the associated quotient singularities $\mathbb{C}^2/G$ and their resolution graphs, and relating them to the McKay quivers $Q(G)$ of the involved groups. The McKay quiver $Q(G)$ turns out to be a doubled version of the corresponding \emph{extended} Dynkin diagram. 

Another interpretation involves the skew-group algebra $\mathbb{C}[x,y] \ast G$, which is Morita equivalent to a quotient of the path algebra $kQ(G)$ by commutativity relations $I$. This quotient $kQ(G)/I$ is the preprojective algebra of an acyclic orientation of the underlying extended Dynkin diagram. In fact, different acyclic orientations yield the same preprojective algebra, so via the preprojective algebra every basic connected tame hereditary algebra appears in the correspondence.  

It is of great interest to generalise the correspondence, or at least parts of it, to higher dimensions. Geometrically, one can consider finite subgroups of $\SL_{n+1}(\mathbb{C})$. On the representation-theoretic side, a generalisation of the situation is provided by Iyama's higher Auslander-Reiten theory (AR-theory). 

The higher versions of hereditary representation infinite algebras are called $n$-representation infinite algebras, which were introduced by Herschend, Iyama and Oppermann in \cite{HIO}. In the same paper, the authors also introduced $n$-hereditary algebras, and proved in a dichotomy-style theorem that this class of algebras naturally splits into $n$-representation infinite algebras and $n$-representation finite algebras. 

The $n$-representation finite algebras and the related $n$-cluster tilting subcategories have been studied extensively since their introduction \cite{Herschend_Iyama_2011, IyamaOppermann, IyamaOppermannStable, DarpöIyama}. The infinite side has received less attention, and fewer examples are known, partly due to the difficulty of checking directly whether a given algebra satisfies the definition of being $n$-representation infinite. There are some general constructions like tensor products \cite{HIO} and $n$-Auslander-Platzeck-Reiten ($n$-APR) tilting \cite{MizunoYamaura} which produce $n$-representation infinite algebras directly, and they may be obtained by finding suitable tilting objects in derived categories of varieties \cite{BuchweitzHille}. 

The main source, however, is the class of $(n+1)$-Calabi-Yau algebras, taking a detour via the higher preprojective algebras: In detail, to an $n$-hereditary algebra $\Lambda$, Iyama and Oppermann \cite{IyamaOppermannStable} associate the higher preprojective algebra as the tensor algebra
\[ \Pi_{n+1}(\Lambda) = T_\Lambda \Ext^{n}(D(\Lambda), \Lambda),  \]
which naturally comes with a grading induced from tensor degrees. Then one obtains $\Lambda$ as the degree $0$ part of $\Pi_{n+1}(\Lambda)$. It was shown \cite{KellerVandenBergh, HIO} that if $\Lambda$ is $n$-representation infinite, then $\Pi_{n+1}(\Lambda)$ is locally finite dimensional and \emph{bimodule $(n+1)$-Calabi-Yau of Gorenstein parameter $1$}. Furthermore, these properties characterise all $n$-representation infinite algebras. Thus, one may take an ungraded $(n+1)$-Calabi-Yau algebra and produce a grading of Gorenstein parameter $1$ to construct an $n$-representation infinite algebra. This leads to the following general question \cite{Thibault}. 

\begin{Ques*}
    Let $\Gamma$ be bimodule $(n+1)$-Calabi-Yau. When does there exist a grading making $\Gamma$ into an $(n+1)$-preprojective algebra? 
\end{Ques*}

When the Calabi-Yau algebra $\Gamma$ is basic, the gradings we refer to are called \emph{cuts}, see \Cref{Def: Cuts}. Performing an $n$-APR tilt then amounts to a change of grading called \emph{cut mutation} \cite{IyamaOppermann}. We follow this strategy of computing preprojective algebras, and use it to give a description of all $n$-representation infinite algebras of type $\Tilde{A}$. This type is defined by analogy with the classical McKay correspondence. 
A $1$-representation infinite algebra $\Lambda$ is nothing but a hereditary representation infinite algebra, and type $\Tilde{A}$ means that $\Lambda \simeq kQ$ where $Q$ is an acyclic orientation of an extended Dynkin diagram of type $\Tilde{A}$. In this case, the preprojective algebra $\Pi_2(kQ)$ is isomorphic to a skew-group algebra $k[x_1, x_2] \ast G$, where $G \leq \SL_2(k)$ is a finite abelian group acting on the variables of $k[x_1, x_2]$. The higher case is defined by requiring that $\Lambda$ is $n$-representation infinite and that $\Pi_{n+1}(\Lambda) \simeq R \ast G$, where $G \leq \SL_{n+1}(k)$ is a finite abelian group acting on $R = k[x_1, \ldots, x_{n+1}]$. 

The first subtlety that arises with this definition is that not every such skew-group algebra $R \ast G$ will appear as a higher preprojective algebra \cite{Thibault}. Additionally, it is a priori not obvious how to decide whether a given $R \ast G$ does appear, and if so how to find the correct grading that will make it into a higher preprojective algebra. Furthermore, there may be several different gradings that turn $R \ast G$ into a higher preprojective algebra. We investigated this question for the case $n=2$ in \cite{DramburgGasanova} and gave a classification for type $\Tilde{A}$. The case $n=2$ has also been studied from a more geometric perspective by Nakajima using dimer models in \cite{Nakajima}. 

We exploit the same ideas we used for $n=2$ in \cite{DramburgGasanova}, and show that they generalise to arbitrary $n$. The algebra $R \ast G $ can be equivalently described via an embedding of lattices $L_1 \leq L_0 = \mathbb{Z}^{n}$ which gives a presentation of the dual group $\hat{G} \simeq L_0/L_1$. To every cut of interest we assign a so-called \emph{type vector}
and show that these types naturally are certain $(|G| \cdot L_1^\vee)$-lattice vectors.

\begin{Theo*}[= \Cref{Theo: Types are dual lattice points}]
    Let $L_1 \xhookrightarrow[]{} L_0$ be a cofinite embedding of lattices with $|L_0/L_1| = m$. Then a vector $\gamma \in \mathbb{N}^{1 \times (n+1)}$ is a type vector of a cut if and only if $\sum_{i =1}^{n+1} \gamma_i = m $ and 
    \[ (\gamma_1, \ldots, \gamma_n) \in m \cdot L_1^\vee. \]
    Furthermore, $\gamma$ is a type of a higher preprojective cut if and only if in addition, $\gamma_i >0$ for all $i \in \{1, \ldots, n+1 \}$.
\end{Theo*}

The conditions of \Cref{Theo: Types are dual lattice points} show that the types of higher preprojective cuts are the \emph{internal} $N_G$-lattice points of a certain $n$-dimensional simplex $P$ embedded in $\mathbb{R}^{n+1}$, with the lattice $N_G$ determined by $G$, see \Cref{Cor: Cut exists iff internal point exists}. Deciding whether a given algebra $R \ast G$ has a preprojective cut is then equivalent to deciding whether $P$ has an internal lattice point, and the negative case splits naturally into some lower dimensional families and finitely many exceptions, see \Cref{Cor: No higher preprojective grading if exceptional or projecting}. 

The analogy with McKay correspondence then continues as follows. Since $G$ is abelian, the variety $X = \spec(R^G)$ associated to the invariant ring is toric, hence defined by a lattice cone $\sigma$. The simplex $P$, which we obtain in a purely combinatorial way, turns out to be isomorphic to the simplex one obtains by intersecting $\sigma$ with a suitable hyperplane at ``height $1$''. In standard toric terms, this can be seen as a consequence of a cut, that is a grading, on $R \ast G$ defining a certain $1$-parameter subgroup of $X$.

\begin{Theo*}[= \Cref{Pro: Polytopes agree}]
     Let $P$ be the $N_G$-lattice simplex whose lattice points correspond to the types of cuts on $R\ast G \simeq kQ/I$, and $N$ the lattice of $1$-parameter subgroups of the toric variety $X = \spec(R^G)$, and $\sigma \subseteq N_\mathbb{R} =  N \otimes_\mathbb{Z} \mathbb{R}$ the defining cone of $X$. Then the lattice simplex in $N_\mathbb{R}$ given by intersecting the defining cone $\sigma \subseteq N_\mathbb{R}$ with the hyperplane corresponding to the degree condition $ |\prod_{i = 1}^{n+1} x_i | = 1$ on $R^G$, is isomorphic to $P$.  
\end{Theo*}

Using the language of Ito and Reid from \cite{ItoReid}, we can see that our types correspond to junior elements in $G$, which are known to form the same simplex we have been considering, see \Cref{Rem: Ito-Reid interpretation}. This completes the pictures, since our initial construction, following Amiot, Iyama and Reiten \cite{AIR}, is based on finding certain junior elements. The geometric perspective shows that the decision problem of whether $R \ast G$ has such a higher preprojective grading is equivalent to deciding whether $G$ contains certain nontrivial junior elements, which in turn correspond to exceptional crepant divisors in a relative minimal model of $X$. While this is of course computable for a given $G$, we are not aware of a simple criterion similar to \cite[Theorem 5.15]{Thibault} that allows us to decide the existence of such a junior element without computing all elements of $G$, or the simplex $P$, or some geometric data. 
We furthermore note that the quivers, and in particular the degree $0$ parts of our cuts, have appeared in work by Yamagishi \cite{YamagishiIAbelian}, based on the work of Logvinenko on geometrically natural families of $G$-constellations \cite{LogvinenkoGNAToverQuot, logvinenko2008natural}. There, Yamagishi works with $G$-constellations and their moduli spaces for a finite abelian group $G \leq \SL_{n+1}(k)$. The motivation is purely geometric, and we are not aware of an a priori reason for why quivers of $n$-representation infinite algebras should appear. 

After fixing an algebra $R \ast G$ and a type, we investigate how to obtain all the cuts of this type. We show that cut mutation is transitive on the set of cuts of a given positive type. 

\begin{Cor*}[= \Cref{Cor: Mutation is transitive}]
    Any pair of $L_1$-periodic cuts $C_1$ and $C_2$ of a positive type $\gamma$ can be obtained one from another via a sequence of (nonzero) mutations.    
\end{Cor*}

This was already known for $n=2$ from the dimer perspective \cite{Nakajima} or from combinatorial considerations \cite{DramburgGasanova}. In this way, we answer a question of Amiot, Iyama and Reiten \cite{AIR} by showing that all cuts are mutation equivalent to cuts arising from cyclic groups. 

This also leads to an algorithmic classification of $n$-representation infinite algebras of type $\Tilde{A}$: For each abelian group $G$ of order $m$, we can go through all embeddings of $G$ into $\SL_{n+1}(k)$ by considering generating sets of $\hat{G}$ of size $n$. In the next step, for a fixed embedding we can calculate the simplex $P$ of types of cuts. Then for each type, we can construct a cut. Finally, we can apply mutations until we found all cuts of the same type. In this way, we can list all type $\Tilde{A}$ algebras with $m$ simple modules.

Our transitivity result is part of a broader observation. The set of cuts of a fixed type can be given the structure of a finite distributive lattice, whose cover relations are given by certain mutations. More precisely, we fix a vertex in the quiver $Q$ of $R \ast G$ at which we do not allow mutations, and the remaining mutations, called the nonzero mutations, give the cover relations.

\begin{Theo*}[= \Cref{Pro: Min and max of cuts} \& \Cref{Cor: Cuts are lattice} \& \Cref{Theo: mutations are covers}]
    For a fixed positive type vector $\gamma$, the set $ \{C \mid  C \text{ is } L_1\text{-periodic and } \theta(C) = \gamma \}$ becomes a finite distributive lattice for the order induced by $C_1 \precdot C_2$ if and only if $C_1$ can be obtained from $C_2$ via a nonzero sink mutation.
\end{Theo*}

One way to view this lattice is by considering the height functions on $Q$, which we define in \Cref{Sec: Height functions}. These functions generalise the $n=2$ theory of dimer models, and they turn out to capture our combinatorics in any dimension. Since $Q$ is finite, the height functions are determined by finitely many values, so they can be seen as vectors in some $\mathbb{Z}^m$, and the lattice of cuts of a fixed type translates via height functions to a sublattice of $(\mathbb{Z}^m, \leq)$. A more thorough investigation of these lattices seems warranted, and we initiate this by constructing their minimal and maximal elements in \Cref{SSec: Constructing min and max}. We also observe that the lattices are always self-dual.  

We may view $Q$ as being embedded on the $n$-torus, and the height functions as a replacement for perfect matchings of the dimer model. This idea was previously considered from a statistical point of view in \cite{LammersDimer} for bounded regions. It would be interesting to generalise our observations to quivers of other $(n+1)$-Calabi-Yau algebras. 

\subsection*{Outline}
We begin by recalling some facts about $n$-representation infinite algebras, their preprojective algebras, quiver descriptions, higher preprojective gradings and skew-group algebras in \Cref{Sec: Setup}. The first technical tool we need is called the height function, which is the main focus in \Cref{Sec: Height functions}, and it is used to prove one implication of our first main theorem \Cref{Theo: Divisibility conditions} in \Cref{Pro: Divisibility conditions are necessary}. In \Cref{Sec: Construction of cuts}, we then prove the other implication via an explicit construction in \Cref{Pro: Cut constr}, and comment on the problem of deciding the existence of cuts in \Cref{SSec: Hollow lattice polytopes}. After that, we reinterpret \Cref{Theo: Divisibility conditions} and give explicit connections to toric geometry in \Cref{Sec: Toric geometry}. Finally, in \Cref{Sec: Mutation lattices}, we consider cut mutation. In \Cref{Cor: Cuts are lattice}, we show that cuts of the same type form a finite distributive lattice, and obtain in \Cref{Cor: Mutation is transitive} that they are related by a sequence of mutations. We construct the maximal element of such a lattice in \Cref{Theo: constructing max}. 

\subsection*{Notation}
We use the following notation and conventions throughout. 
\begin{enumerate}
    \item Unless specified otherwise, we denote by $\Bbbk$ an algebraically closed field of characteristic $0$.
    \item All undecorated tensors $\otimes $ are taken over the field $\Bbbk$.
    \item We denote by $\diag(x_1, \ldots, x_n) $ the diagonal $(n \times n)$-matrix with diagonal entries $x_1, \ldots, x_n$. 
    \item We write $\frac{1}{m}(e_1, \ldots, e_n)$ for the diagonal matrix $\diag(\xi^{e_1}, \ldots, \xi^{e_n})$ of order $m$, where $\xi$ is a primitive $m$-th root of unity which will be specified when necessary. 
    \item A quiver $Q$ consists of vertices $Q_0$, arrows $Q_1$ and the source and target maps $s, t \colon Q_1 \to Q_0$.
    \item For $ a \in \mathbb{Z}$ and $n \in \mathbb{N}$, we write $a \bmod n$ for the smallest non-negative representative of the residue class $a + n\mathbb{Z}$. 
    \item For $ a, b \in \mathbb{Z}$ and $n \in \mathbb{N}$ we write $ a \equiv_n b$ if $a+ n \mathbb{Z} = b + n \mathbb{Z}$. We drop the subscript $n$ when no confusion is possible. 
\end{enumerate}

\section{Setup}\label{Sec: Setup}
In this section, we collect the necessary definitions and tools we will use throughout. We work over an algebraically closed field $k$ of characteristic $0$. 

\subsection{Higher representation-infinite algebras and their preprojective algebras}
Let $\Lambda$ be a finite dimensional $k$-algebra of finite global dimension $n \geq 1$. In Iyama's higher Auslander-Reiten theory, the following functors play an important role:
\[ \nu = D\RHom_\Lambda(\--,\Lambda) \colon \Db( \mod \Lambda) \to \Db( \mod \Lambda)  \] 
is the \emph{derived Nakayama functor} on the bounded derived category of finitely generated $\Lambda$-modules, and its quasi-inverse is
\[\nu^{-1} = \RHom_{\Lambda^{\mathrm{op}}}(D(\--), \Lambda) \colon \Db( \mod \Lambda) \to \Db( \mod \Lambda).\]
The \emph{derived higher Auslander-Reiten translation} is the autoequivalence 
\[\nu_n := \nu \circ [-n] \colon \Db(\mod \Lambda ) \to \Db( \mod \Lambda).\]

\begin{Def}\cite[Definition 2.7]{HIO} 
The algebra $\Lambda$ is called \emph{$n$-representation infinite} if for any projective module $P$ in $ \mod \Lambda$ and integer $i \geq 0$ we have 
\[ \nu_n^{-i} P \in   \mod \Lambda . \]
\end{Def}

For a given $\Lambda$, this condition may be difficult to check since it requires knowledge of the infinitely many translates $\nu_n^{-i}(\Lambda)$. However, the \emph{higher preprojective algebra} allows a different characterisation. 

\begin{Def}\cite[Definition 2.11]{IyamaOppermannStable}
Let $\Lambda$ be of global dimension at most $n$. The $(n+1)$-preprojective algebra of $\Lambda$ is 
\[ \Pi_{n+1}(\Lambda) = T_\Lambda \Ext^{n}_\Lambda(D(\Lambda), \Lambda). \]
\end{Def}

We will refer to such an algebra $\Pi_{n+1}(\Lambda)$ simply as preprojective. It was established by Keller in \cite{KellerVandenBergh} that such a preprojective algebra is Calabi-Yau in the sense of Ginzburg \cite{Ginzburg}. The preprojective algebra comes naturally equipped with a grading, given by tensor degree, and we recover the original algebra as the degree $0$ part, i.e. we have $\Lambda = \Pi_{n+1}(\Lambda)_0$. The properties of this grading, together with being Calabi-Yau, essentially determine such preprojective algebras. The following definition is a convenient summary of the relevant properties.

\begin{Def}\cite[Definition 3.1]{AIR} \label{Def: n-CY GP a}
Let $\Gamma=\bigoplus_{i \geq 0}\Gamma_i$ be a positively graded $k$-algebra. We call $\Gamma$ a \emph{bimodule $(n+1)$-Calabi-Yau algebra of Gorenstein parameter $a$} if there exists a bounded graded projective $\Gamma$-bimodule resolution
$P_\bullet$ of $\Gamma$ and an isomorphism of complexes of graded $\Gamma$-bimodules
\[ P_\bullet \simeq \Hom_{\Gamma^{\mathrm{e}}}(P_\bullet, \Gamma^{\mathrm{e}})[n+1](-a).  \]
\end{Def}

Note that if $\Gamma$ is bimodule $(n+1)$-Calabi-Yau of Gorenstein parameter $a$, then the ungraded version of $\Gamma$ is simply $(n+1)$-Calabi-Yau. Conversely, different gradings may be introduced on $\Gamma$, so that the projective bimodule resolution becomes graded, and we can study the resulting Gorenstein parameter. It was shown by Minamoto and Mori in \cite{MinamotoMori} that the degree $0$ part of a bimodule $(n+1)$-Calabi-Yau algebra of Gorenstein parameter $1$ is $n$-representation infinite. The theorems from \cite{MinamotoMori} and \cite{KellerVandenBergh} are summarised in the following bijection. 

\begin{Theo}\cite[Theorem 4.35]{HIO}\label{Theo: HPG is f.d. GP1}
There is a bijection between $n$-representation infinite algebras $\Lambda$ and bimodule $(n+1)$-Calabi-Yau algebras $\Gamma$ of Gorenstein parameter 1 with $\dim_\Bbbk \Gamma_i < \infty $ for all $i \in \mathbb{N}$, both sides taken up to isomorphism. The bijection is given by 
\[ \Lambda \mapsto \Pi_{n+1}(\Lambda) \quad \mathrm{ and } \quad \Gamma \mapsto \Gamma_0.  \]
\end{Theo}

In particular, this means that being higher preprojective is the additional structure of a specific grading on a Calabi-Yau algebra. Just like the term \emph{graded algebra} means an algebra $\Gamma$ with a chosen grading, we therefore mean by \emph{preprojective algebra} an algebra $\Gamma$ with a chosen grading such that $\Gamma = \Pi_{n+1}(\Gamma_0)$. We can therefore phrase the general motivating question as follows. 

\begin{Ques}
    Let $\Gamma$ be bimodule $(n+1)$-Calabi-Yau. Is $\Gamma$ higher preprojective? How many different gradings exist that make $\Gamma$ higher preprojective? 
\end{Ques}

We restrict our interest to the so-called type $\Tilde{A}$ in this article. The type $\Tilde{A}$ is also defined in terms of the preprojective algebra. For this definition, we briefly recall the construction of a skew-group algebra. Given a $k$-algebra $A$ and a group $G$ acting on $A$ via automorphisms, the skew-group algebra is the tensor product 
\[ A \otimes  kG \]
with multiplication given by 
\[ (a \otimes g) \cdot (b \otimes h) = (ag(b) \otimes gh).  \]
In particular, given a group $G \leq \SL_{n+1}(k)$, we have that $G$ acts by left-multiplication on $V = \langle x_1, x_2, \ldots, x_{n+1} \rangle$, and hence on the polynomial ring $R = \Bbbk[V]$. 

\begin{Def}
    An $n$-representation infinite algebra $\Lambda$ has type $\Tilde{A}$ if $\Pi_{n+1}(\Lambda) \simeq k[x_1, \ldots, x_{n+1}] \ast G$ for a finite abelian group $G \leq \SL_{n+1}(k)$. 
\end{Def}

For the purposes of classifying the $n$-representation infinite algebras we may therefore first compute $R \ast G$, and then compute all possible higher preprojective gradings. However, there are many examples of finite abelian groups $G \leq \SL_{n+1}(k) $ such that $R \ast G$ can not be endowed with this additional structure, and we are not aware of an easy criterion to decide when this is possible. We note an important result in this direction, proven by Thibault also for non-abelian groups.

\begin{Theo}\cite[Theorem 5.15]{Thibault} \label{Theo: Thibaults Criterion}
Suppose that the embedding of a finite subgroup $G \leq \SL_{n+1}(k)$ factors through some embedding of $\SL_{n_1}(k) \times \SL_{n_2}(k)$ into $\SL_{n+1}(k)$ for $n_1 + n_2 = n+1$. Then $R \ast G$ can not be endowed with the grading of a $(n+1)$-preprojective algebra.
\end{Theo}

However even for abelian groups, there are examples showing that this criterion is not sufficient, see \cite[Example 3.15, Example 6.5]{DramburgGasanova}. To give some constructions of preprojective structures, we need an explicit quiver description of $R \ast G$, and we need to describe how it interacts with preprojective gradings. Before we do so, we need the following simplification of the setup.  

\subsection{Cuts}\label{SSec: Cuts}
Note that for an abelian group $G$, the algebra $R \ast G$ is basic and can therefore be written as the quotient of a quiver algebra, i.e. we have $R \ast G \simeq kQ/I$. To construct a grading, it then often suffices to grade the quiver $Q$, i.e. to place certain arrows in $Q$ in degree $1$ so that $kQ/I$ becomes graded. This method has been employed e.g.\ in \cite{Giovannini, DramburgGasanova, HIO}. However, it is a priori not clear that every possible higher preprojective grading will arise in this way. We therefore make the following definition. 

\begin{Def}\label{Def: Cuts}
    Let $R \ast G \simeq kQ/I$. A higher preprojective grading on $R \ast G$ is called a \emph{cut} if every arrow in $Q$ is homogeneous of degree $0$ or $1$. For such a cut, we let $C = \{ a \in Q_1 \mid \deg(a) = 1 \} $ and refer to the set $C$ as a cut as well.
\end{Def}

Taking the preprojective degree $0$ part of an algebra with a cut amounts to removing, or cutting, all arrows in $C$ from the quiver $Q$ and removing all relations from $I$ which contain arrows in $C$. In order to refer to this procedure, the following notation will be used throughout. 

\begin{Def}\label{Def: Cut quiver}
    Let $C \subseteq Q_1$ be a cut. Then we refer to the cut quiver $Q_C = (Q_0 , Q_1 -C)$ and the cut quiver algebra $(kQ/I)_C = \langle k(Q_C) \rangle \leq kQ/I$. 
\end{Def}

Note that not every higher preprojective grading is a cut. For a detailed discussion of this problem, we refer the reader to \cite[Section 3.1]{DramburgGasanova}. However, in many cases it is possible to map a higher preprojective grading to a cut via some necessarily ungraded automorphism of $R \ast G$. This has been investigated in \cite{DramburgSandoy}. In particular, if every higher preprojective grading can be mapped to a cut, it suffices to study cuts to classify all $n$-representation infinite algebras that give rise to the same higher preprojective algebra.

\begin{Rem}
    The original definition of higher type $\Tilde{A}$ in \cite{HIO} only mentions $n$-representation infinite algebras that arise from cuts. While we can produce other higher preprojective gradings by applying ungraded automorphisms to cuts, the resulting $n$-representation infinite algebras will be isomorphic to those constructed via cuts. We therefore only consider cuts in this article.  
\end{Rem}

\subsection{Quivers and universal covers}
In this subsection, we give a detailed summary of the quiver description of $R \ast G$ for abelian $G$. This is based on \cite{BSW} and \cite{HIO}. See also \cite[Section 7.1]{DramburgGasanova} for a similar summary. We recall that $G$ acts on $V = \langle x_1, \ldots, x_{n+1} \rangle$ via the embedding of $G$ in $\SL_{n+1}(k)$. In particular, we view $V$ as a representation of $G$. Since $G$ is abelian, we can diagonalise this representation. That means decompose $V = \rho_1 \oplus \cdots \oplus \rho_{n+1}$ into 1-dimensional representations so that the subspace $\langle x_i \rangle$ affords the representation $\rho_i$.  
As proven in \cite{BSW}, the quiver for $R \ast G$ is the \emph{McKay quiver} of $G$. That means $R \ast G \simeq kQ/I$, where $Q = (Q_0, Q_1)$ has vertices $Q_0 = \Irr(G)$ given by the irreducible representations of $G$. For $\chi_1, \chi_2 \in \Irr(G)$, the arrows $\chi_1 \to \chi_2$ are a basis of $\Hom_G(\chi_1, \chi_2 \otimes V)$.  Since $V$ is a faithful representation, it follows that $\{ \rho_1, \ldots, \rho_n \}$ is a generating set for the dual group $\hat{G} = \Hom(G, \mathbb{C}^\ast)$ \cite[Lemma 4.2]{DramburgGasanova}. We arrive at the following observation.

\begin{Rem}
    Let $G$ be abelian. The quiver $Q$ for $R \ast G$ is the directed Cayley graph of $\hat{G}$ with respect to the generating set determined by $V$: The representation $V$ decomposes into $n+1$ irreducible representations $\rho_1, \ldots, \rho_{n+1}$ of $G$, which can be seen as elements of the dual group $\hat{G}$. Since $V$ is a faithful, the set $\{ \rho_1, \ldots, \rho_{n+1} \}$ generates $\hat{G}$. The quiver is therefore given by the elements of $\hat{G}$ together with arrows from every element $\chi$ to each product $\chi \rho_i = \chi \otimes \rho_i$. In particular, every vertex of $Q$ has $n+1$ outgoing and $n+1$ incoming arrows. 
\end{Rem}

The fact that $Q$ is a Cayley graph also suggests that we associate a type to each arrow, corresponding to the $\rho_i$ it is associated to. 

\begin{Def}
    Let $Q$ be the quiver of $R \ast G$, where $G$ is abelian. Every arrow $a \colon \chi_1 \to \chi_2$ corresponds to an isomorphism $\chi_1 \simeq \chi_2 \otimes \rho_i$ for a unique $1 \leq i \leq n+1$, and we call $\theta(a) = i$ the \emph{type} of the arrow. 
\end{Def}

We recall from \cite{BSW} that the relations $I$ for $R \ast G \simeq kQ/I$ are generated by commutator relations. 

\begin{Rem}\label{Rem: Quiver description of R G}
    The isomorphism $R \ast G \simeq kQ/I$ takes the element $x_i \otimes 1 \in R \ast G$ to the sum of all arrows of type $i$. The commutativity relations $x_i x_j = x_j x_i$ in $R$ give rise to the defining relations of $kQ/I$ as follows. 
    Let $a_i \colon x \to y$ and $a_j \colon y \to z$ be of types $\theta(a_i) = i \neq j = \theta(a_j)$. Then there exist arrows $b_j \colon x \to y'$ and $b_i \colon y' \to z$ of types $\theta(b_j) = j$ and $\theta(b_i) = i$. In $kQ/I$, we have 
    \[ a_i a_j = b_j b_i, \]
    and $I$ is generated by all relations of this form. 
\end{Rem}

Next, we describe a universal cover of $Q$ as detailed in \cite[Section 5]{HIO}. We denote the standard basis of $\mathbb{R}^{n}$ by $\{ \alpha_1, \ldots, \alpha_n \}$ and fix $\alpha_{n+1} = - \sum_{i = 1 }^n \alpha_i$. We consider the $n$-dimensional lattice 
\[ L_0 = \langle \alpha_i \mid 1 \leq i \leq n \rangle \]
and the corresponding infinite quiver 
\[ \hat{Q} = (L_0 , \{ x \to x + \alpha_i \mid 1 \leq i \leq {n+1} \} ).    \]
To cover $Q$, we consider the map 
\[ L_0 \to \hat{G}, \alpha_i \mapsto \rho_i,   \]
where $V = \rho_1 \oplus \cdots \oplus \rho_{n+1}$ is the decomposition of the representation $V$. This is a group homomorphism, which on the level of quivers induces a surjection 
\[ \hat{Q} \to Q, \]
such that the arrows $x \to x + \alpha_i$ in $\hat{Q}$ map to the arrows of type $i$ in $Q$. Completing this into a short exact sequence of abelian groups 
\[ 0 \to L_1 \to L_0 \to \hat{G} \to 0 ,\]
we see that the kernel $L_1$ is itself an $n$-dimensional lattice, so we may think of the quiver $Q$ being embedded on an $n$-torus, while $\hat{Q}$ is the universal cover of this quiver in $\mathbb{R}^n$. The precise relationship between finite subgroups of $\SL_{n+1}(\Bbbk)$ and sublattices $L_1 \leq L_0$, as well as toric quotient singularities, will be made precise in \Cref{SSec: Reps and lattices}. 

We now use the cover $\hat{Q}$ to fix important terminology. 

\begin{Def}
    We keep the setup for $\hat{Q}$. 
    \begin{enumerate}
        \item We say that an arrow $x \to x + \alpha_i$ has \emph{type} $i$. 
        \item A cycle of length $n+1$ in $\hat{Q}$ consisting of arrows of $n+1$ distinct types is called an \emph{elementary cycle}. 
        \item We call a set of arrows $\hat{C} \subseteq \hat{Q}_1$ a \emph{cut} if every elementary cycle passes through exactly one arrow in $\hat{C}$.
        \item The same terminology applies to arrows and cycles in $Q$.
        \item We call a cut in $\hat{Q}$ \emph{periodic} with respect to $L_1$ if it is invariant under $L_1$-translation, that is, $x \to x + \alpha_i$ is in $\hat{C}$ if and only if $(x+y) \to (x+y) + \alpha_i$ is in $\hat{C}$ for all $y \in L_1$. 
    \end{enumerate}
\end{Def}

\begin{Rem}
    The purpose of periodicity is that the periodic cuts are precisely the ones that descend to $Q$. More precisely, let $\hat{C}$ be a cut on $\hat{Q}$, then for every arrow $a \in \hat{C}$, its $L_1$-orbit is contained in $\hat{C}$, and this orbit is naturally identified with a single arrow in $Q$. This clearly gives a bijection between $L_1$-periodic cuts of $\hat{Q}$ and cuts of $Q$, and to alleviate notation we identify these without explicit mention when no confusion is possible. Hence we also refer to the infinite cut quiver $\hat{Q}_C$ in analogy with \Cref{Def: Cut quiver}. 
\end{Rem}

Let us now characterise a higher preprojective cut purely in quiver terms. This has been proven explicitly for $n=2$ in \cite{DramburgGasanova}. See also \cite[Theorem 5.6]{HIO} and \cite[Corollary 4.7]{Giovannini} for parts of this statement. The part of the proof showing finite-dimensionality uses a statement that is most conveniently proven using height functions, which we introduce in \Cref{Sec: Height functions}. The implication using height functions is in fact slightly stronger. We obtain that the cut $Q_C$ quiver is acyclic. In \cite[Question 5.9]{HIO}, the authors ask whether the quiver of an $n$-hereditary algebra is acyclic, so our result confirms this for all algebras of type $\Tilde{A}$, in any global dimension.  

\begin{Pro}\label{Pro: Cut is higher pp iff positive type}
    Let $C \subseteq Q_1$ be a cut. Then declaring $\deg(a) = 1$ for all arrows $a \in C$ defines a higher preprojective structure on $kQ/I$ if and only if $C$ contains arrows of all types. In this case, the cut quiver $Q_C$ is acyclic. 
\end{Pro}

\begin{proof}
    We need to check the conditions from \Cref{Theo: HPG is f.d. GP1}. To do this, we use the bimodule resolution of $R \ast G \simeq kQ/I$ given by the superpotential $\omega$ as in \cite{BSW}. In detail, $\omega$ is a linear combination of all elementary cycles in $Q$, and its $i$-th partial derivatives generate the $(n+1-i)$-th term in a projective bimodule resolution of $kQ/I$. Since we assumed that $C$ is a cut, it follows that with the grading induced from $C$, the superpotential $\omega$ becomes homogeneous of degree $1$. Now consider the last term in the projective bimodule resolution, which is generated by the $0$-th derivatives of $\omega$. That means it is generated by $\omega$ itself, and hence in degree $1$ with respect to the induced grading. Therefore, every cut induces a grading of Gorenstein parameter $1$ on $kQ/I$. It remains to consider when the grading is locally finite dimensional. Since $kQ/I$ is generated by $Q$, and $C \subseteq Q_1$ is finite, it clearly suffices to check that the induced degree $0$ part $(kQ/I)_0$ is finite dimensional. If $C$ does not contain arrows of all types, we assume without loss of generality that $C$ contains no arrows of type $1$. Then it follows from \Cref{Rem: Quiver description of R G} that the sum of these arrows corresponds under the isomorphism $kQ/I \simeq R \ast G$ to the element $x_1 \otimes 1$. Clearly, this element generates an infinite dimensional subalgebra isomorphic to $k[x_1] $ in $(kQ/I)_0$, so the cut is not higher preprojective. Now suppose that $C$ does contain arrows of all types. Then it follows from \Cref{Lem: Positive type is acyclic} that the quiver $Q_C$, which generates $(kQ/I)_0$, is acyclic and hence $(kQ/I)_0$ is finite dimensional. 
\end{proof}

Note that the above proposition clearly applies to $L_1$-periodic cuts as well. We therefore now fix a cofinite sublattice $L_1 \leq L_0$ and want to study $L_1$-periodic cuts of $\hat{Q}$. Once $L_1$ is fixed, we also fix the size of the quotient $|L_0/L_1| = |\hat{G}| = m$. 

\begin{Con}\label{Con: B' and B}
We recall the short exact sequence 
\[ 0 \to L_1 \xrightarrow[]{B'} L_0 \to \hat{G} \to 0. \]
We choose $B'$ as a matrix for the embedding, with respect to the fixed basis $\{ \alpha_1, \ldots, \alpha_n  \} $ for $L_0$. Since $B'$ is then defined up to $\GL_n(\mathbb{Z})$-right action, that means up to base change of $L_1$, we furthermore assume that $\det(B') = |G|$, and consider $B'$ up to $\SL_{n}(\mathbb{Z})$-right action. We will also refer to the matrix 
\[  B = \left(\begin{array}{ c | c }
    B' & \begin{matrix}
           1 \\
           \vdots \\
         \end{matrix} \\
    \hline
    0 \cdots  & 1
  \end{array}\right).
\]
\end{Con} 

\begin{Rem}\label{Rem: Positive matrix}
    It will be convenient at times to consider a specific $B'$ which has only non-negative entries. This is possible by computing the Hermite normal form of $B'$. 
\end{Rem}

A crucial tool to study cuts are their types, which we define only for $Q$ but can be defined analogously for $\hat{Q}$ by choosing a fundamental domain for the $L_1$-action.

\begin{Def}
    For a cut $C$ on $Q$, we call 
    \[ \theta(C) = (\# \{ a \in C \mid \theta(a) = i  \})_{1 \leq i \leq n+1} \]
    the \emph{type} of $C$.
\end{Def}

We make some observations on these types.

\begin{Pro}\label{Pro: Types are on a cone slice}
    Let $C \subseteq Q_1$ be a cut of type $\theta(C) = (\gamma_1, \ldots, \gamma_{n+1})$. Then $\sum_{i = 1}^{n+1} \gamma_i = |L_0/L_1| = |G| = m $. In particular, $0 \leq \gamma_i \leq m$ for all $i$.
\end{Pro}
       
\begin{proof}
 First note that $Q$ contains $m \cdot n!$ elementary cycles, counted up to cyclic permutation. Indeed, there are $m$ arrows of type $1$, and every arrow of type $1$ gives rise to $n!$ different elementary cycles. For every arrow $a\in C$ let $S_a$ be the set of elementary cycles passing through $a$, and note that $|S_a|=n!$ for every $a\in C$. Since every elementary cycle passes through \emph{exactly} one arrow in $C$, the sets $S_a$ form a partition of the set of all elementary cycles, and hence $|C|=m$.
\end{proof}       

\begin{Rem}
It follows from \Cref{Pro: Types are on a cone slice} that the vectors $\theta(C)$ are certain integer points on the hyperplane $\sum_{i= 1}^{n+1} x_i = m$ intersected with the cone spanned by the vectors $(0, \ldots, 0, m, 0, \ldots, 0)$. These are precisely the type vectors of the \emph{trivial cuts}, which are obtained by cutting all arrows of type $i$. This connection will be made more precise in \Cref{SSec: Hollow lattice polytopes} and \Cref{Sec: Toric geometry}.
\end{Rem}    

\subsection{From representations to lattices and back}\label{SSec: Reps and lattices}
Before we continue, let us briefly but thoroughly discuss the relationship between the choice of a finite abelian subgroup $G \leq \SL_{n+1}(\Bbbk)$, the associated representation, and the constructed cofinite sublattice $L_1 \leq L_0$, as well as the choice of the matrix $B'$, and even the toric variety $X = \spec(R^G)$. 

We first note that we may always assume that $G \leq \SL_{n+1}(\Bbbk)$ is diagonal.
\begin{Rem}
    Let $G \leq \SL_{n+1}(\Bbbk)$ be a finite group, and let $M \in \GL_{n+1}(\Bbbk)$ be a matrix. Then the group $G' = M^{-1}GM$ consisting of the conjugates of $G$ under $M$ is isomorphic to $G$, and we have $R \ast G \simeq R \ast G'$. Furthermore, if $G$ is abelian, the elements of $G$ can be simultaneously diagonalised by conjugation. When viewing the embedding $G \hookrightarrow \SL_{n+1}(\Bbbk)$ as a representation of $G$, this is simply the decomposition of the representation into its irreducible summands. 
\end{Rem}

Next, we note that $R \ast G$ only depends on the chosen subgroup, while addressing $G \leq \SL_{n+1}(\Bbbk)$ via a representation leaves room for ambiguity. This ambiguity is precisely the choice of an automorphism of $G$. 

\begin{Rem}
    Let $\rho_1, \rho_2 \colon G \to \SL_{n+1}(\Bbbk)$ be faithful representations of the finite group $G$ such that the images agree, i.e.\ $\rho_1(G) = \rho_2(G)$. Then clearly $R \ast (\rho_1(G)) \simeq R \ast (\rho_2(G))$. However, $\rho_1$ and $\rho_2$ need not be isomorphic representations. Indeed, let $\varphi$ be a non-trivial automorphism of $G$, then for a faithful representation $\rho_1$, the homomorphism $\rho_2 = \rho_1 \circ \varphi^{-1}$ in general gives rise to a different representation, but still the images clearly agree. Furthermore, it is easy to see that if $\rho_1$ and $\rho_2$ are faithful representations with equal images, then $ \rho_2^{-1} \circ \rho_1$ gives an automorphism of $G$ by which the two representations differ. 
\end{Rem}

A particular consequence of this observation is that we do not need to specify a primitive root of unity. 

\begin{Rem}
    If $G$ is abelian and the integer $d$ is coprime to the exponent $e = \exp(G) = \min \{ n \in \mathbb{N}_{> 0} \mid g^n = 1 \text{ for all } g \in G \}$, then the morphism $G \to G, g \mapsto g^d$ is an automorphism. 
    This justifies the use of the notation $\frac{1}{m}(e_1, \ldots, e_n)$ for a diagonal matrix with diagonal entries $\xi^{e_i}$ where $\xi$ is an unspecified primitive $m$-th root of unity. Changing the root of unity amounts to applying an automorphism to the group $\langle \frac{1}{m}(e_1, \ldots, e_n) \rangle $ which yields a different representation with the same image.  
\end{Rem}

Thus, we now may simply think of diagonal abelian subgroups $G \leq \SL_{n+1}(\Bbbk)$, which can be produced by different representations. From this, we construct $L_1 \leq L_0$, and note that the construction does not depend on the chosen representation, only on the image. However, since we view $L_0$ as generated by the ordered basis vectors $\alpha_1, \ldots, \alpha_n$, the ordering of the diagonal of $G$ needs to be taken into account. 

\begin{Rem}
    Let $\rho \colon G \hookrightarrow \SL_{n+1}(\Bbbk)$ be the embedding of a diagonal abelian subgroup, and decompose it as $\rho_1 \oplus \cdots \oplus \rho_{n+1}$. Next, let $\varphi$ be an automorphism of $G$, which gives rise to an automorphism of the dual group $\hat{G}$ via $\chi \mapsto \chi \circ \varphi^{-1}$. This means that when we construct $L_1$ as the kernel of 
    \[ L_0 \to \hat{G}, \alpha_i \mapsto \rho_i, \]
    it is clear that $L_1$ is also the kernel of 
    \[ L_0 \to \hat{G}, \alpha_i \mapsto \rho_i \circ \varphi^{-1}. \]
\end{Rem}

We have thus established the following correspondence.

\begin{Rem}
    There is a bijective correspondence between diagonal finite subgroups $G \leq \SL_{n+1}(\Bbbk)$ and cofinite sublattices $L_1 \leq L_0$. Taking abelian finite subgroups $G \leq SL_{n+1}(\Bbbk)$ up to conjugation, which means that we additionally allow ourselves to permute the diagonal entries, we thus obtain cofinite sublattices $L_1 \leq L_0$ up to the permutation of the basis vectors of $L_0$. 
\end{Rem}  

Next, we consider the matrix $B'$ we use when dealing with $L_1 \leq L_0$. 

\begin{Rem}
    The matrix $B'$ is chosen as a matrix for the embedding $L_1 \hookrightarrow L_0$ with respect to the chosen basis $\alpha_1, \ldots, \alpha_n$ of $L_0$ and an arbitrary basis of $L_1$. Therefore, the matrix $B'$ is unique only up to $\GL_n(\mathbb{Z})$-right action. Thus, we obtain a further bijective correspondence between diagonal finite abelian subgroups $G \leq \SL_{n+1}$ and full rank $n \times n$ integer matrices up to $\GL_n(\mathbb{Z})$-right action.
\end{Rem}

Finally, as a preparation for \Cref{Sec: Toric geometry}, let us add the toric variety $X = \spec(R^G)$ into the picture. We refer the reader to \cite[Chapter 11.4]{CLS} for details. 

\begin{Rem}
    We associate to the lattice $L_0$ the torus $T_0 = \Bbbk^n/L_0$, and to the lattice $L_1$ the torus $T_1 = \Bbbk^n/L_1$. The fact that $L_1 \leq L_0$ gives rise to a morphism $T_1 \twoheadrightarrow T_0$ with kernel $L_0/L_1$. If we choose a basis $\{b_1, \ldots, b_n\}$ of $L_1$ and view the cone spanned by these vectors as an $L_1$-lattice cone, the associated toric variety is simply $X_1 = \Bbbk^n$ with torus $T_1$. When viewing the cone instead as an $L_0$-lattice cone, we obtain a variety $X_0$ with torus $T_0$. The map $T_1 \twoheadrightarrow T_0$ extends to a map $X_1 \twoheadrightarrow X_0$, which shows that $X_0 = X_1/(L_0/L_1)$ is the quotient of affine space by the action of a finite abelian group. Furthermore, the action is precisely determined by the data of the embedding $L_1 \leq L_0$, and knowing the action we can recover $L_1 \leq L_0$. The involved matrix $B'$ is simply given by putting the chosen vectors $b_1, \ldots, b_n$ as columns for the matrix $B'$. 
    
    To recover $X$, we simply recall that the action of $G$ in the missing $(n+1)$-th coordinate direction is recovered from the actions in the known ones, that is we add one extra dimension to $X_1$ on which $G$ acts by scalars which are completely determined by $X_0$. This is mirrored by passing from $L_1 \xhookrightarrow{B'} L_0$ to the corresponding matrix $B$ which embeds a lattice of rank $n+1$ into $\mathbb{Z}^{n+1}$. 
\end{Rem}

Let us illustrate the above correspondences in an example. 

\begin{Exp}\label{Exp: C2 x C6}
    Let $g=\frac{1}{2}(1,1,0) $ and $h = \frac{1}{6}(1,4,1) $, and define $G = \langle g,h \rangle \leq \SL_3(\Bbbk)$. As we have seen above, it is not necessary to specify the roots of unity, we simply fix one choice. It is easy to compute that $G \simeq C_2 \times C_6$ is non-cyclic of order $12$. To describe the lattice $L_1 \leq L_0$, we need to find the minimal relations between the diagonal entries, viewed as linear characters of $G$. More precisely, we consider the linear characters $(1,1), (1,4), (0,1)$ of $G$, where the linear character $(a,b)$ sends $g$ to $\epsilon_2^a$ and $h$ to $\epsilon_6^b$, where the $\epsilon$ are the chosen primitive roots of unity. Clearly, as elements in $\hat{G}$, the linear character $(1,1)$ has order $6 = \exp(G)$. Furthermore, it is easy to see that a minimal relation between the first two characters is $4 \cdot (1,1) + 2 \cdot (1,4) = (0,0)$. Hence, a matrix $B'$ for this group is 
    \[ B' = \left( \begin{matrix}
        6 & 4 \\ 0 & 2
    \end{matrix}  \right). \]
    Computing the Smith normal form of this matrix also shows that we are indeed working with a group of isomorphism type $C_2 \times C_6$. The lattice $L_1$ is given by $\mathbb{Z}$-column span of $B'$, read in coordinates with respect to the fixed basis of $L_0$. Had we instead chosen a different way to present the minimal relations, we could have obtained for example the matrix 
    \[ \left( \begin{matrix}
        0 & -2 \\ 6 & 2
    \end{matrix}  \right).\]
    Lastly, since $G \leq \SL_3(\Bbbk)$, we know that $(1,1) + (1,4) + (0,1) = (0,0)$. This can be checked by hand, of course, and leads to the matrix 
    \[ B = \left( \begin{matrix}
        6 & 4 & 1 \\ 0 & 2 & 1 \\ 0 & 0& 1
     \end{matrix}  \right). \]
\end{Exp}

\section{Height functions}\label{Sec: Height functions}
In this section, we introduce the main tool to investigate cuts, called the height function. This is inspired by the case $n=2$, where one can associate a height function to a perfect matching on a dimer model \cite{IUDimerSpecial, kenyon2009lecturesdimers}. The same generalisation also appeared in \cite{LammersDimer}. Since there is no dual bipartite graph in the case $n>2$, we instead define the height function directly on the lattice $L_0$. When $n=2$, this recovers a function which is equivalent to the height function for the hexagonal dimer model. As in \Cref{Con: B' and B}, we fix a cofinite $L_1 \xhookrightarrow{B'} L_0$ and we write $|L_0/L_1| = m$. We first introduce height functions abstractly. Then we will show that ($L_1$-equivariant) height functions are equivalent to ($L_1$-periodic) cuts. 

\begin{Def}
    A map $h \colon L_0 \to \mathbb{Z}$ is called a \emph{height function} if it satisfies the following two conditions: 
    \begin{enumerate}
        \item We have $h(0) = 0$,
        \item For any $x \in L_0$ and $i \in \{1 , \ldots, {n+1} \} $ we have 
        \[h(x+ \alpha_i) \in \{ h(x) +1, \; h(x) -n \}.\] 
    \end{enumerate}
    If in addition $h(x + y) = h(x) + h(y)$ for all $x \in L_0$ and $y \in L_1$, we call $h$ an $L_1$-\emph{equivariant} height function. 
\end{Def}

Clearly, every height function gives rise to a cut. 

\begin{Pro}\label{Pro: height function to cut}
    Let $h \colon L_0 \to \mathbb{Z}$ be a height function. Then the set 
    \[ C_h = \{ (x \to x + \alpha_i) \in \hat{Q}_1 \mid h(x + \alpha_i) = h(x) - n \}  \]
    is a cut. If furthermore $h$ is $L_1$-equivariant, then $C_h$ is $L_1$-periodic.  
\end{Pro}

\begin{proof}
    Let $x \to x + \alpha_{i_1} \to \cdots \to x $ be an elementary cycle in $\hat{Q}$. Since the cycle starts and ends at the same vertex $x$, and $h$ changes by $1$ or $-n$ along each arrow, it follows that exactly $1$ arrow in this cycle leads from a vertex with greater height to one with smaller height. Hence, every elementary cycles has exactly one arrow in $C_h$. 
    
    Now assume in addition that $h$ is $L_1$-equivariant. By construction, an arrow $x \to x + \alpha_i $ is in $C_h$ if and only if $h(x + \alpha_i) = h(x) - n$. Since $h$ is $L_1$-equivariant, we then have $h(x + y + \alpha_i) = h(y) + h(x + \alpha_i) = h(x+y) - n$, hence $C_h$ is $L_1$-periodic. 
\end{proof}

Now, we want to construct a height function $h_C \colon L_0 \to \mathbb{Z}$ from a (not necessarily $L_1$-periodic) cut $C$. To do this, we first consider the height increment $\Delta h_C(p)$ of a path $p$. The height function $h_C$ at a point $x$ will simply be the height increment of any path from $0$ to $x$. 

\begin{Def}
    Let $C$ be a cut and let $a \in \hat{Q}_1$ be an arrow. We define 
    \[ \Delta h_C(a) = \begin{cases}
        1 & a \not \in C, \\
        - n & a \in C.
    \end{cases} \]
    Let $p = a_1 a_2 \cdots a_l $ be a path in $\hat{Q}$. Then we define the height increment along $p$ as
    \[ \Delta h_C(p) =  \sum_{i = 1}^l \Delta h_C(a_i).  \]
\end{Def}

\begin{Rem}
\label{Rem: strongly connected}
For any pair of points $x,y\in L_0$ there is a path from $x$ to $y$ in $\hat{Q}$. In other words, $\hat{Q}$ is strongly connected. This follows from the more general observation that, as an undirected graph, $\hat{Q}$ is connected, and that any arrow can be completed into a cycle. Indeed, choose an undirected path from $x$ to $y$. If this undirected path contains a reversed arrow, we can simply replace this arrow by a directed path that completes the arrow into a cycle.    
\end{Rem}

In the next proposition, we show that the height increment of any cycle is $0$. We do this by showing that we may deform a path by commuting arrows of different types or contracting elementary cycles without affecting its height increment. As a consequence, we will obtain that the height increments of any two parallel paths are the same.

\begin{Pro}
\label{Pro: Closed paths}
Let $c \colon  x\to \ldots \to x$ be a cycle in $\hat Q$. Then $\Delta h_C(c)=0$.    
\end{Pro}

\begin{proof}
First, note that any elementary cycle $c$ has length $n+1$, and exactly one of its arrows is in $C$, so we have $\Delta h_C(c) = 0$. Next, let $p_1=a_ia_j$ and $p_2=a_ja_i$ be two paths from $x$ to $x+\alpha_i+\alpha_j$. Then $\Delta h_C(p_1)=\Delta h_C(p_2)$. Indeed, this is obvious for $i=j$, and otherwise $p_1$ and $p_2$ can be completed into elementary cycles by the same path $p_3$, that is, there exists a path $p_3$ such that $p_1p_3$ and $p_2p_3$ are elementary cycles. But that means 
    \[ \Delta h_C(p_1) + \Delta h_C(p_3) =  \Delta h_C(p_1p_3) = 0 = \Delta h_C(p_2p_3) = \Delta h_C(p_2) + \Delta h_C(p_3),   \]
and hence $\Delta h_C(p_1) = \Delta h_C(p_2)$. This means that deforming a path using commutativity relations $a_ia_j=a_ja_i$ does not change the height increment.
Finally, let $c$ be any cycle $x\to \ldots \to x$ and let $t_i$ denote the number of arrows of type $i$ in $c$ for all $i \in \{1,\ldots, n+1 \}$. Then 
 $$
 0=\sum_{i=1}^{n+1}t_i\alpha_i=\sum_{i=1}^{n}t_i\alpha_i+t_{n+1}\left(-\sum_{i=1}^{n}\alpha_i\right)=\sum_{i=1}^{n}(t_i-t_{n+1})\alpha_i.
 $$
 from which we conclude that $t_1=t_2=\ldots=t_{n+1}$. Using commutativity relations, this path can therefore be deformed into the path $(c')^{t_1}$, where $c'$ is an elementary cycle. Since $c'$ gives height increment $0$, so does $(c')^{t_1}$ and hence the proof is complete.   
\end{proof}

\begin{Cor}\label{Pro: Parallel paths have same height increment}
    Let $p$, $q$ be paths from $x$ to $ y$ in $\hat{Q}$. Then 
    \[ \Delta h_C(p) = \Delta h_C(q). \]
\end{Cor}

\begin{proof}
Let $r \colon  y\to \ldots \to x$ be any path in $\hat{Q}$, which we know exists by \Cref{Rem: strongly connected}. Then both $pr$ and $qr$ are cycles $x\to \ldots \to x$. Then
 $$\Delta h_C(p)+h_C(r)=\Delta h_C(pr)=0= \Delta h_C(qr)=\Delta h_C(q)+\Delta h_C(r),$$
 which implies $\Delta h_C(p)=\Delta h_C(q)$.
\end{proof}

This allows us to make the following construction, which is the converse of \Cref{Pro: height function to cut}.

\begin{Pro}\label{Pro: cut to height function}
    Let $C$ be a cut, and for $x \in L_0$ denote by $p_x$ any path from $0$ to $x$ in $\hat{Q}$. Then the map 
    \[ h_C \colon L_0 \to \mathbb{Z},\  h_C(x) = \Delta h_C(p_x)   \]
    is a height function. If furthermore $C$ is $L_1$-periodic, then $h_C$ is $L_1$-equivariant.  
\end{Pro}

\begin{proof}
    Let $x \in L_0$. It follows from \Cref{Pro: Parallel paths have same height increment} that $h_C$ is well defined, since any two paths $p$, $q$ from $0$ to $x$ have the same height increment. Next, recall that for every arrow $a \in \hat{Q}$, we have $\Delta h_C(a) \in \{ 1, -n \}$, and hence $\Delta h_C(x + \alpha_i) \in \{ h_C(x) +1, h_C(x) -n \} $ for any $x \in L_0$ and any $i \in \{ 1,\ldots, n+1 \}$. Since any path from $0$ to $0$ is a cycle, it follows that $\Delta h_C(0) = 0$, so $h_C$ is indeed a height function. 

    Now assume in addition that $C$ is $L_1$-periodic, and let $y \in L_1$. It follows from periodicity that there exist paths $0 \to \cdots \to x$ and $y \to \cdots \to x+y$ of the same height increment. Therefore, we have $h_C(x + y) = \Delta h_C(p_{x + y}) = \Delta h_C(p_y)  + \Delta h(p_x) = h_C(x) + h_C(y) $, and hence $h_C$ is $L_1$-equivariant.  
\end{proof}

Thus, we have established that height functions and cuts carry the same information. 

\begin{Cor}\label{Cor: Bijection height functs and cuts}
    The maps $h \to C_h$ and $C \to h_C$ are mutually inverse bijections between the set of height functions and the set of cuts. They restrict to bijections between the set of $L_1$-equivariant height functions and $L_1$-periodic cuts. 
\end{Cor}

\begin{proof}
    Follows from combining \Cref{Pro: cut to height function} and \Cref{Pro: height function to cut}. 
\end{proof}

Just as with cuts, we will often drop the adjective $L_1$-equivariant for a height function, since we will only consider those which are $L_1$-equivariant. 

\begin{Rem}
\label{Rem: equivariance gives homeo}
    Note that a height function $h$ does not descend to a function on $L_0/L_1$, but equivariance implies that a height function $h$ restricts to a homomorphism $h \colon L_1 \to \mathbb{Z}$. 
\end{Rem}

Next, we aim to compute the values of the homomorphism $h_C$ on $L_1$ explicitly, based on the type of $C$. For the computation, we fix the type $\theta(C) = \gamma=(\gamma_i)_{1 \leq i \leq n+1}$. Due to the discrepancy between the number of entries in $\gamma$ and the dimension of $L_0$, the following notation will be convenient throughout.

\begin{Def} \leavevmode
    \begin{enumerate}
        \item We denote by $\mathbf{1} = (1, \ldots , 1)$ the vector consisting of all $1$. Its length will be clear from context. 
        \item For row or column vectors $x = (x_1, \ldots, x_{l_1} )$ and $y = (y_1, \ldots, y_{l_2})$, where $l_1, l_2 \in \{ n, n+1 \}$, we write 
    \[  \langle x, y \rangle_{n} = \sum_{i = 1}^{n} x_i y_i. \]
    If $l_1 = l_2 = n+1$, we also write 
    \[\langle x, y \rangle_{n+1} = \sum_{i = 1}^{n+1} x_i y_i. \]
    \end{enumerate}
\end{Def}

\begin{Pro} \label{Pro: Height function values on L1}
    For any $y \in L_1$ and $L_1$-periodic cut $C$ of type $\gamma=(\gamma_1,\ldots, \gamma_{n+1})$ we have
    \[ h_C(y) = \left\langle y,\textbf{1}-\frac{n+1}{m}\gamma\right\rangle_{n}.\]
\end{Pro}

\begin{proof}
    Let $o_i$ denote the order of $\alpha_i$ in $L_0/L_1$. First we prove the claim for the points $o_i \alpha_i \in L_1$.
    To compute $h_C(o_i \alpha_i)$, consider the path $0 \to \alpha_i \to 2\alpha_i \to \cdots \to o_i \alpha_i$, where each arrow is of type $i$. Then $h_C(o_i \alpha_i) = o_i - (n+1) \theta_i'$, where $\theta_i' = \# \{ j \in [0, o_i) \mid ( j \alpha_i \to (j+1) \alpha_i ) \in C \}  $ is the number of arrows in this path which lie in $C$. Next, consider any path $x \to x+ \alpha_i \to \ldots \to x + o_i \alpha_i$. Since $o_i\alpha_i\in L_1$, equivariance implies 
    \[h_C(x + o_i \alpha_i) - h_C(x) = h_C(o_i\alpha_i)=o_i - (n+1) \theta_i'. \]
    Equivalently, this means that the number of arrows in $C$ in this path is also 
    \[ \theta_i' =  \frac{1}{n+1}(o_i - h_C(o_i \alpha_i)).   \]
    To cover all arrows of type $i$ in $Q$ exactly once by these paths, we let $\{ x_1, \ldots, x_{\frac{m}{o_i}} \} $ be a system of representatives of the cosets of $\langle \alpha_i \rangle$ in $L_0/L_1$. This means that each arrow of type $i$ in $Q$ lies in exactly one cycle 
    \[ x_j + L_1 \to x_j + \alpha_i + L_1 \to \cdots \to x_j + o_i\alpha_i + L_1,  \]
    and hence we have that 
    \[ \gamma_i = \frac{m}{o_i} \cdot \theta_i' = \frac{m}{o_i} \frac{1}{n+1}(o_i - h_C(o_i \alpha_i)),  \]
    which gives 
    \[ h_C(o_i \alpha_i) = o_i - \frac{n+1}{m} \gamma_i o_i=o_i\left(1-\frac{n+1}{m}\gamma_i\right),   \]
    which is the desired statement, given that $o_i\alpha_i=(0,\ldots,0, o_i,0,\ldots,0)^\top$. Now for an arbitrary $y=(y_1,\ldots,y_n)^\top \in L_1$ we compute $h_C(my)$ in two different ways. On the one hand we have 
    $$
    h_C(my)=h_C\left(\sum_{i=1}^{n}{my_i\alpha_i}\right)=h_C\left(\sum_{i=1}^{n}{y_i\cdot\frac{m}{o_i}\cdot(o_i\alpha_i)}\right)=\sum_{i=1}^{n}{y_i\cdot\frac{m}{o_i}\cdot h_C(o_i\alpha_i)},
    $$
    since $o_i\alpha_i\in L_1$ and since $\frac{m}{o_i}\in \mathbb{Z}$ for all $i$. Since we already established the height values of $o_ia_i$, the above equality becomes
    \begin{align*}
       h_C(my)&=\sum_{i=1}^{n}{y_i\cdot\frac{m}{o_i}\cdot o_i\left(1-\frac{n+1}{m}\gamma_i\right)}\\&=m\sum_{i=1}^{n}{y_i\cdot\left(1-\frac{n+1}{m}\gamma_i\right)}=m\left\langle y, \textbf{1}-\frac{n+1}{m}\gamma\right\rangle_n 
    \end{align*}

    On the other hand, $h_C(my)=mh_C(y)$ since $y\in L_1$. This completes the proof.
\end{proof}

The homomorphism from the above proposition will be useful later on, so we fix notation for it here. 
 
\begin{Def}
    For a type $\gamma$ we write 
    \[ h_\gamma \colon L_1 \to \mathbb{Z}, y \mapsto  \left\langle y,\textbf{1}-\frac{n+1}{m}\gamma\right\rangle_{n}. \]
\end{Def}

The formula we give above for $h_\gamma$ is invertible, so we can use $h_\gamma$ or the type $\gamma$ interchangeably, as the following proposition shows. 

\begin{Pro} \label{Pro: Homs equal iff types equal}
    Let $\gamma$ and $\delta$ be types of $L_1$-periodic cuts. Then $h_\gamma = h_\delta$ if and only if $\gamma = \delta$.
\end{Pro}

\begin{proof}
    One implication is obvious. For the other, assume that $h_\delta(y) = h_\gamma(y)$ for all $y\in L_1$. Therefore, from \Cref{Pro: Height function values on L1} we get 
    \[ \langle y,\textbf{1}\rangle_n-(n+1)\frac{\langle y, \gamma\rangle_n}{m}=h_\gamma(y)=h_\delta(y)=\langle y,\textbf{1}\rangle_n-(n+1)\frac{\langle y, \delta\rangle_n}{m}, \]
 which implies $\langle y, \gamma-\delta\rangle_n=0$ for all $y\in L_1$. Setting $y=m\alpha_i\in L_1$ for all $i \in  \{ 1,\ldots, n \}$, we obtain $\gamma_i=\delta_i$ for all $i \in \{1,\ldots, n \}$. Since
 $\langle \gamma,\textbf{1}\rangle_{n+1}=m=\langle\delta,\textbf{1}\rangle_{n+1}$, we also get $\gamma_{n+1}=\delta_{n+1}$.
\end{proof}

The following reformulation of \Cref{Pro: Height function values on L1} and \Cref{Pro: Homs equal iff types equal} will be useful in \Cref{Sec: Mutation lattices}

\begin{Cor}\label{Cor: type from L1 height}
    Let $h$ be an $L_1$-equivariant height function and let $\gamma\in\mathbb{Z}^{n+1}$. Then $h(y)=\left\langle y,\textbf{1}-\frac{n+1}{m}\gamma\right\rangle_{n}$ for all $y \in L_1$ if and only if $C_h$ has type $\gamma$.
\end{Cor}

We are ready to show the necessity of the condition from \Cref{Theo: Divisibility conditions}.

\begin{Pro} \label{Pro: Divisibility conditions are necessary}
    Let $C$ be an $L_1$-periodic cut on $\hat{Q}$ of type $\gamma$ and $B$ as in \Cref{Con: B' and B}. Then we have 
    \[ \gamma \cdot B \in m \mathbb{Z}^{1 \times (n+1)} \]
\end{Pro}

\begin{proof}
    Write $\gamma =(\gamma_1,\ldots,\gamma_{n+1})$. It suffices to show that 
    \[ (\gamma_1, \ldots, \gamma_n) y \in m \mathbb{Z}^{1 \times n} \]
    for all columns $y$ of $B'$. Recall from \Cref{Rem: Positive matrix} that we may choose $B'$ with non-negative entries. It suffices to show the claim for this choice of $B'$ since clearly the $\SL_n(\mathbb{Z})$-action leaves invariant the sublattice $m\mathbb{Z}^{1\times n}$, more precisely we have $ m \mathbb{Z}^{1 \times n} \cdot \SL_n(\mathbb{Z}) = m \mathbb{Z}^{1\times n}$. 
    Thus, let $y  = (y_1, \ldots, y_n)^\top$ be a column, where $y_i \geq 0 $. We follow the path 
    \[ 0 \to \alpha_1 \to 2 \alpha_1 \to \cdots \to y_1 \alpha_1 \to y_1 \alpha_1 + \alpha_2 \to \cdots \to y_1 \alpha_1 + y_2 \alpha_2 \to \cdots \to y.  \]
    This path has $\langle y,\textbf{1}\rangle_n$ arrows. Denote by $\varepsilon$ the number of arrows in this path that belong to $C$. Then, $h_C(y)=\langle y,\textbf{1}\rangle_n-(n+1)\varepsilon$.
    On the other hand, it follows from \Cref{Pro: Height function values on L1} that 
    \[ h_C(y) = \left\langle y,\textbf{1}-\frac{n+1}{m}\gamma\right\rangle_n=\langle y,\textbf{1}\rangle_n-\frac{n+1}{m} \langle y,\gamma\rangle_n. \]
    Comparing the two expressions for $h_C(y)$, we obtain $\langle y,\gamma\rangle_n=m\varepsilon\in m\mathbb{Z}$,
    which proves our claim. 
\end{proof}

We conclude this section by using height functions to prove that cuts of positive type give acyclic cut quivers. In other words, a cut of positive type is higher preprojective, see \Cref{Pro: Cut is higher pp iff positive type}.

\begin{Lem}\label{Lem: Positive type is acyclic}
Let $C$ be an $L_1$-periodic cut of type $\theta(C) = \gamma=(\gamma_1, \gamma_2, \ldots, \gamma_{n+1})$ such that $\gamma_i > 0$ for all $i$. Then there is no infinite path in $\hat{Q}_C$. Equivalently, the cut quiver $Q_C$ for $L_0 / L_1$ is acyclic.  
\end{Lem}

\begin{proof}
 Assume there is an infinite path in $\hat{Q}_C$. Since $L_0/L_1$ is a finite group, there exists a nontrivial subpath $c \colon x\to\ldots\to x+y$ in $\hat{Q}_C$ such that $y\in L_1$. Let $y=(y_1,\ldots,y_n)^\top$. 
 We know that $h_C(x+y)-h_C(x)=h_C(y)$, therefore, the height increment of $c$ is $h_C(c)=h_C(y)=h_{\gamma}(y)$. Now, the height increases exactly by $1$ each step and hence $c$ has exactly $h_{\gamma}(y)$ steps. Let $t=(t_1,\ldots, t_{n+1})$, where $t_i$ denotes the number of arrows of type $i$ in $c$. Then we have
  $t_{n+1}=h_{\gamma}(y)-\langle t, \textbf{1}\rangle_n$ and 
 \begin{align*}
     \sum_{i=1}^{n}{y_i\alpha_i}&=\sum_{i=1}^{n+1}{t_i\alpha_i}\\&=\sum_{i=1}^{n}{t_i\alpha_i}+(h_{\gamma}(y)-\langle t, \textbf{1}\rangle_n)\alpha_{n+1} \\
     &=\sum_{i=1}^{n}{t_i\alpha_i}+(h_{\gamma}(y)-\langle t, \textbf{1}\rangle_n)\left(-\sum_{i=1}^{n}{\alpha_i}\right)\\
     &=\sum_{i=1}^{n}{t_i\alpha_i}+(\langle t, \textbf{1}\rangle_n-h_{\gamma}(y))\left(\sum_{i=1}^{n}{\alpha_i}\right)\\
     &=\sum_{i=1}^{n}{\alpha_i(\langle t, \textbf{1}\rangle_n+t_i-h_{\gamma}(y))}.
 \end{align*}
 Comparing the $\alpha_i$-coefficients, we obtain for each $i \in \{ 1,\ldots, n \}$ that
 $$
 y_i=\langle t, \textbf{1}\rangle_n+t_i-h_{\gamma}(y).
 $$
 Summing over all $i$, we obtain
$$
 \langle y,\textbf{1}\rangle_n =n\langle t, \textbf{1}\rangle_n+\langle t, \textbf{1}\rangle_n-nh_{\gamma}(y)=(n+1)\langle t, \textbf{1}\rangle_n-nh_{\gamma}(y),
 $$
in other words,
$$
\langle t, \textbf{1}\rangle_n=\frac{\langle y,\textbf{1}\rangle_n+nh_{\gamma}(y)}{n+1}.
$$
This, combined with the equation for the $\alpha_i$-coefficient, gives 
\[ t_i=y_i+h_{\gamma}(y)-\langle t, \textbf{1}\rangle_n=y_i+h_{\gamma}(y)-\frac{\langle y,\textbf{1}\rangle_n+nh_{\gamma}(y)}{n+1}=y_i+\frac{h_{\gamma}(y)-\langle y,\textbf{1}\rangle_n}{n+1} \]
for all $i \in \{ 1,\ldots,n \}$. We also conclude that 
$$
t_{n+1}=h_{\gamma}(y)-\langle t, \textbf{1}\rangle_n=h_{\gamma}(y)-\frac{\langle y,\textbf{1}\rangle_n+nh_{\gamma}(y)}{n+1}=\frac{h_{\gamma}(y)-\langle y,\textbf{1}\rangle_n}{n+1}.
$$ 
Now, 
\begin{align*}
\langle t,\gamma\rangle_{n+1}&=\langle y,\gamma\rangle_n+\frac{h_{\gamma}(y)-\langle y,\textbf{1}\rangle_n}{n+1}\langle \gamma, \textbf{1}\rangle_{n+1}\\&=\langle y,\gamma\rangle_n+\frac{h_{\gamma}(y)-\langle y,\textbf{1}\rangle_n}{n+1}\cdot m\\&=\frac{m}{n+1}\left(\left\langle y,\frac{n+1}{m}\gamma\right\rangle_n+h_{\gamma}(y)-\langle y,\textbf{1}\rangle_n\right)\\&= \frac{m}{n+1}\left(h_{\gamma}(y)-\left\langle y,\textbf{1}-\frac{n+1}{m}\gamma\right\rangle_n\right)=0,
\end{align*}
where the last equality follows from \Cref{Pro: Height function values on L1}.
Since $\gamma_1,\ldots,\gamma_{n+1}>0$ and $t_1,\ldots,t_{n+1}\ge 0$, we conclude $t_1=\ldots=t_{n+1}=0$, which implies that the path $c$ is empty, which contradicts our assumption.
\end{proof}

\section{Construction of cuts and simplices of types}\label{Sec: Construction of cuts}
\subsection{Construction of cuts}
In this subsection, we construct a cut for a given type, following the ideas outlined in \cite{DramburgGasanova}. We fix as in \Cref{Con: B' and B} $L_1 \xhookrightarrow[]{B'} L_0$ and $m = |L_0/L_1|$. We begin by recalling the construction by Amiot-Iyama-Reiten from \cite[Section 5]{AIR}, which produces a cut for the quiver $Q$ of a cyclic group generated by a junior element. In the following, recall that we refer by $x \bmod m$ to the smallest non-negative representative of the equivalence class $x + m \mathbb{Z}$. 

\begin{Pro}\cite[Theorem 5.6]{AIR} \label{Pro: Decreasing arrows cut}
    Suppose $L_0/ L_1$ is cyclic. If there exist $e_1, \ldots, e_{n+1} \in \{ 0, \ldots, m-1 \}$ with $\sum_{i = 1}^{n+1} e_i = m$ and an isomorphism 
    \[ \xi \colon L_0/L_1 \to  \mathbb{Z}/m\mathbb{Z} \]
    such that $\xi(\alpha_i + L_1) = e_i + m \mathbb{Z}$ for all $i \in \{ 1, \ldots, n+1\} $, then there exists a cut of type $(e_1, \ldots, e_{n+1})$ on $Q$. 
\end{Pro}

\begin{proof}
    We use the isomorphism $\xi$ to label the vertices of $Q$ by $\{ 0, \ldots , m-1\}$, by identifying $\xi(x)$ with its smallest non-negative representative. This way, every arrow $x \to x + \alpha_i$ of type $i$ corresponds to $j \to (j + e_i) \bmod m$ for $\xi(x) = j \in \{ 0, \ldots, m-1 \}$. We then define 
    \[ C = \{ j \to (j + e_i) \bmod m \mid j > (j + e_i) \bmod m \}. \]
    We claim that $C$ is a cut. To see this, consider an elementary cycle $c$. Let $D$ denote the set of decreasing arrows of $c$, that is, $D = C\cap c$, and let $N=c\setminus D$ be the set of non-decreasing arrows in $c$. Each decreasing arrow acts by adding $e_i-m$, and any non-decreasing arrow acts by adding $e_i$. Then we get
    \begin{align*}
    0&=\sum_{a\in N } e_{\theta(a)}+\sum_{a\in D}{(e_{\theta(a)}-m)}\\&=\sum_{a\in c}{e_{\theta(a)}}-\sum_{a\in D}{m}=m-\sum_{a\in D}{m},
    \end{align*}
    which means there is exactly one decreasing arrow in $c$. To see that $C$ is of the correct type, simply note that an arrow of type $i$ is in $C$ if and only if it ends at one of the vertices $0, 1, \ldots, e_i -1$. 
\end{proof}

\begin{Rem}
    Note that the condition on the $e_i$ in \Cref{Pro: Decreasing arrows cut} means that we are considering the group $\mathbb{Z}/m\mathbb{Z}$ as generated inside of $\SL_{n+1}(k)$ by the element $\frac{1}{m}(e_1, \ldots, e_{n+1})$. In the terminology of Ito-Reid \cite{ItoReid}, this is a junior element. 
\end{Rem}

Now, our construction proceeds in two steps. We construct an epimorphism $\xi$ as above for each potential type $\gamma$, and then lift a cut from this to our quiver. Equivalently, we are showing that we can find a cyclic subgroup of $G$ generated by the correct junior to realise $\gamma$ as a type, and extend to all of $G$. 

\begin{Lem} \label{Lem: xi is a hom}
    Let $\gamma = (\gamma_1, \ldots, \gamma_{n+1}) \in \mathbb{N}^{1 \times (n+1)}$ such that $\langle \gamma, \mathbf{1} \rangle_{n+1} = m$ and $\gamma B \in m \mathbb{Z}^{1 \times (n+1)}$. Then 
    \[ \xi_\gamma \colon L_0 \to \mathbb{Z}/m\mathbb{Z}, x \to \langle x, \gamma \rangle_n + m\mathbb{Z} \]
    is a group homomorphism with $\Ker(\xi_\gamma) \supseteq L_1$ and cyclic image of order $\frac{m}{\gcd(\gamma_1, \ldots, \gamma_{n+1})}$. 
\end{Lem}

\begin{proof}
    It is clear that $\xi_\gamma$ is a group homomorphism. We first investigate 
    \[ \Ker(\xi_\gamma) =\{x\in L_0\mid  \langle x,\gamma\rangle_n\equiv_m 0 \}. \]
    Recall that any $y \in L_1$ can be written as $y = B'v$ for some $v \in \mathbb{Z}^n$. Using our assumption that $\gamma B \in m \mathbb{Z}^{1 \times (n+1)}$, we find   
    \[ \xi_\gamma(y) = \langle \gamma, B'v\rangle_n=\langle \gamma B,v\rangle_n \equiv_m 0,   \]
and therefore $\Ker(\xi_\gamma) \supseteq L_1$. To compute the order of $\Im(\xi_\gamma)$, note that we can write $d = \gcd(\gamma_1, \ldots, \gamma_{n+1}) = \gcd(\gamma_1, \ldots, \gamma_n, m)$ as an integer combination of $\gamma_1\ldots,\gamma_n$ and $m$. That means there exist numbers $x_1, \ldots, x_n \in \mathbb{Z}$ such that $\xi_\gamma(x_1, \ldots,x_n) = d+ m \mathbb{Z} $. Since $d \mid \gamma_1,\ldots, \gamma_n$, it follows that the subgroup $\Im(\xi_\gamma) \subseteq  \mathbb{Z}/m\mathbb{Z}$ is generated by the coset of $d$, and therefore has order $\frac{m}{\gcd(\gamma_1, \ldots,\gamma_{n+1})}$.
\end{proof}

\begin{Pro} \label{Pro: Cut constr}
Let $\gamma$ be as in \Cref{Lem: xi is a hom} and let $L_2 = \Ker(\xi_\gamma)$. Then there exists an $L_2$-periodic cut $C$ of type $(\frac{\gamma_1}{d},\ldots,\frac{\gamma_{n+1}}{d})$, where $d=\gcd(\gamma_1,\ldots,\gamma_{n+1})$. In particular, $C$ is $L_1$-periodic of type $(\gamma_1,\ldots \gamma_{n+1})$. 
\end{Pro}

\begin{proof}
    We write $m' = \frac{m}{d}$. If $m' = 1$, one of the $\gamma_i$ is $m$ and $L_2 = L_0$, hence we take $C$ to be the trivial cut of the corresponding type as defined in \Cref{Pro: Types are on a cone slice}. Otherwise, by \Cref{Lem: xi is a hom}, the morphism $\xi_\gamma$ has a cyclic image $\Im(\xi_\gamma) \simeq \mathbb{Z}/m'\mathbb{Z}$. Hence, we have that $L_0/L_2$ is cyclic of order $m'$, and we will use \Cref{Pro: Decreasing arrows cut} to construct an $L_2$-periodic cut. To do this, we factor $\xi_\gamma$ as the projection $q \colon L_0 \to L_0/L_2$, followed by the induced isomorphism $\overline{\xi_\gamma}$ to $\mathbb{Z}/m'\mathbb{Z}$, and the inclusion of $\mathbb{Z}/m'\mathbb{Z}$ as $\Im(\xi_\gamma) \leq \mathbb{Z}/m\mathbb{Z}$. 
    \[\begin{tikzcd}
	{L_0} & {\mathbb{Z}/m\mathbb{Z}} \\
	{L_0/L_2} & {\mathbb{Z}/m'\mathbb{Z}}
	\arrow["\xi_\gamma", from=1-1, to=1-2]
	\arrow["q"', from=1-1, to=2-1]
	\arrow["{\overline{\xi_\gamma}}"', "\sim", from=2-1, to=2-2]
	\arrow["{\cdot d}"', from=2-2, to=1-2]
    \end{tikzcd}\]
    We claim that $\overline{\xi_\gamma} \circ q$ satisfies the conditions from \Cref{Pro: Decreasing arrows cut}. Since $\xi_\gamma(\alpha_i) = \gamma_i + m\mathbb{Z}$, we have that $(\overline{\xi_\gamma} \circ q )(\alpha_i) = \frac{\gamma_i}{d} + m \mathbb{Z}$. Furthermore, since $\langle \gamma, \mathbf{1} \rangle_{n+1} = m$, it follows that $\langle \frac{1}{d} \gamma , \mathbf{1} \rangle_{n+1} = \frac{m}{d} = m'$. Therefore, the values $e_i = \frac{\gamma_i}{d}$ together with $\overline{\xi_\gamma} \circ q$ satisfy the conditions from \Cref{Pro: Decreasing arrows cut}. Thus, there exists an $L_2$-periodic cut $C$ of type $\frac{1}{d} \gamma$, and since $L_1$ has index $d$ in $L_2$, it follows that the type of $C$ as an $L_1$-periodic cut is $\gamma$.
\end{proof}

This leads to our first main theorem. 

\begin{Theo} \label{Theo: Divisibility conditions}
    Let $L_1 \xhookrightarrow[]{B'} L_0$ be a cofinite embedding of lattices and $m = |L_0/L_1|$, with $B'$ and $B$ as in \Cref{Con: B' and B}. Then a vector $\gamma \in \mathbb{N}^{1 \times (n+1)}$ is a type of a cut if and only if $\langle \gamma, \mathbf{1}\rangle_{n+1} = m $ and 
    \[ \gamma B \in m \mathbb{Z}^{1 \times (n+1)}. \]
    Furthermore, $\gamma$ is a type of a higher preprojective cut if and only if in addition, $\gamma_i >0$ for all $i \in \{1, \ldots, n+1 \}$.
\end{Theo}

\begin{proof}
    The first implication follows from \Cref{Pro: Divisibility conditions are necessary}. The other implication follows from \Cref{Pro: Cut constr}. The statement about higher preprojectivity is \Cref{Pro: Cut is higher pp iff positive type}. 
\end{proof}

In the following, we rephrase \Cref{Theo: Divisibility conditions} in terms of dual lattices. For this, we need to fix terminology for dual lattices. 

\begin{Rem}
    The lattice $L_0$ is identified with column vectors $\mathbb{Z}^n \subseteq \mathbb{R}^n$. We therefore view the dual lattice $L_0^\vee$ as row vectors $\mathbb{Z}^{ 1 \times n} \subseteq \mathbb{R}^{1 \times n}$. We have fixed a basis $\{ \alpha_1, \ldots, \alpha_n\}$ of $L_0$, which gives rise to the dual basis $\{ \alpha_1^\vee, \ldots, \alpha_n^\vee\}$ of $L_0^\vee$. Therefore, we express vectors in $L_1^\vee$ with respect to this dual basis. In particular, since $L_1$ is generated by the columns of $B'$ as in \Cref{Con: B' and B}, the lattice $L_1^\vee$ is generated by the rows of $(B')^{-1}$. 
\end{Rem}

\begin{Theo} \label{Theo: Types are dual lattice points}
    Let $L_1 \xhookrightarrow[]{} L_0$ be a cofinite embedding of lattices with $|L_0/L_1| = m$. Then a vector $\gamma \in \mathbb{N}^{1 \times (n+1)}$ is a type of a cut if and only if $\langle \gamma, \mathbf{1}\rangle_{n+1} = m $ and 
    \[ \frac{1}{m}(\gamma_1, \ldots, \gamma_n) \in L_1^\vee. \]
    Furthermore, $\gamma$ is a type of a higher preprojective cut if and only if in addition, $\gamma_i >0$ for all $i \in \{1, \ldots, n+1 \}$.
\end{Theo}

\subsection{Simplices of types}\label{SSec: Hollow lattice polytopes}
Now, we want to investigate the problem of when a higher preprojective cut exists for a given $L_1 \xhookrightarrow{B'} L_0$ as in \Cref{Con: B' and B}. To do this, it suffices to show that a type of a higher preprojective cut exists. We first note, as alluded to in \Cref{Pro: Types are on a cone slice}, that the types of all cuts are the lattice points of a certain lattice simplex. 

\begin{Con}
    Recall from \Cref{Con: B' and B} the definition of the matrix $B$, which contains the matrix $B'$ as a submatrix. We interpret $B$ as embedding a lattice of rank $n+1$ into $\mathbb{Z}^{n+1}$ by taking its $\mathbb{Z}$-column space, just as $L_1$ is the $\mathbb{Z}$-column space of $B'$. Furthermore, this lattice has a dual lattice in $\mathbb{R}^{1 \times (n+1)}$, which we denote by $N_G$. More precisely, we write 
    \[ N_G =(\operatorname{Col}_\mathbb{Z}(B))^\vee \subseteq \mathbb{R}^{1 \times (n+1)}. \]
\end{Con}

\begin{Rem}
    Let $\gamma$ be the type of a cut. Since $\gamma_{n+1}$ is determined by $(\gamma_1, \ldots, \gamma_n)$, we can reinterpret \Cref{Theo: Types are dual lattice points} as saying that $\frac{1}{m} \gamma$ lies in $N_G$. Combining this with the condition that $\gamma_i \geq 0$ for all $i \in \{ 1, \ldots, n+1\}$ and $\langle \gamma, \mathbf{1} \rangle_{n+1} = m$, we see that $\frac{1}{m}\gamma$ lies in the convex hull $P = \operatorname{Conv}(\{ \frac{1}{m} \theta(C) \mid C \text{ is a trivial cut} \}) \subseteq \mathbb{R}^{1 \times (n+1)}$. Recall that a trivial cut $C$ is given by placing all arrows of a fixed type $i$ in degree $1$, and that its type therefore is given by $\theta(C)_i = m$. We fix this polytope $P$, and since the vertices of $P$ are in $N_G$ we consider it as an $N_G$-lattice polytope. Furthermore, since $P$ has $n+1$ vertices, it is a lattice simplex.     
\end{Rem}

Rephrasing \Cref{Theo: Types are dual lattice points} then yields the following characterisation.

\begin{Cor}\label{Cor: Cut exists iff internal point exists}
    The skew-group algebra $R \ast G$ admits a higher preprojective cut if and only if the polytope $P$ has an internal $N_G$-lattice point.
\end{Cor}

\begin{proof}
    To see this, simply note that an $N_G$-lattice point in the interior of $P$ has strictly positive entries, and conversely a point on the boundary has at least one entry $0$. 
\end{proof}

We now consider the case when $P$ has no internal $N$-points. 

\begin{Def}
    A lattice polytope $P$ is called \emph{hollow} if there exist no interior lattice points in $P$. A lattice polytope is called \emph{empty} if the only lattice points in $P$ are the vertices. 
\end{Def}

Given a lattice polytope, it is of course possible to check whether it contains an interior lattice point. However, given just the group $G$, we are not aware of an easy criterion for the existence of a higher preprojective cut that avoids computing $P$. Let us comment on this problem for cyclic groups.

\begin{Rem}
    Consider a cyclic group $G = \langle \frac{1}{m}(e_1, \ldots, e_{n+1}) \rangle \leq \SL_{n+1}(k)$, and fix some $B$ for it. Consider any $\gamma \in \mathbb{Z}^{1 \times (n+1)}$ such that 
    \[ \gamma B \in m \mathbb{Z}^{ 1 \times(n+1)} \]
    and $0 \leq \gamma_i < m $ for all $i$. Then $\frac{1}{m}(\gamma_1, \ldots, \gamma_{n+1})$ defines an element in $G$, and conversely every element $\frac{1}{m}(f_1, \ldots, f_{n+1})$ in $G$ satisfies 
    \[  (f_1, \ldots, f_{n+1}) B \in m \mathbb{Z}^{ 1 \times(n+1)}.  \] 
    Hence, to find types of cuts, we need to find elements $\frac{1}{m}(f_1, \ldots, f_{n+1})$ in $G$ such that $\sum_{i=1}^{n+1} f_i = m$. These are simply the junior elements (see \Cref{Sec: Toric geometry}), so we need to compute all elements in $G$ to decide whether $G$ has a junior element $\frac{1}{m}(f_1, \ldots, f_{n+1})$ with all $f_i \neq 0$, and hence whether a higher preprojective cut exists. We are not aware of a general procedure that avoids computing all elements of $G$ to check whether such an element, respectively such an internal point in $P$, exists. In \cite[Example 6.5]{DramburgGasanova}, an example for the group of order $11$ was given, further illustrating this point. 
\end{Rem}

The upshot of this perspective is that a generic lattice polytope is not hollow, as the following theorem by Nill and Ziegler shows, which we phrase for the standard lattice $\mathbb{Z}^n \subseteq \mathbb{R}^n$. Recall that a \emph{lattice projection} is a surjective affine map 
\[f \colon  \mathbb{R}^n \twoheadrightarrow \mathbb{R}^m, x \mapsto Ax + b \]
which preserves the lattices, that means which satisfies that $f(\mathbb{Z}^n) = \mathbb{Z}^m$. In this case, the lattice projection of the lattice polytope $P \subseteq \mathbb{R}^n$ is $f(P)$, which is a lattice polytope in $\mathbb{R}^m$. 

\begin{Theo}\cite[Theorem 1.1]{NillZiegler} \label{Theo: NillZiegler}
    For every $n \in \mathbb{N}$, there are only finitely many hollow $n$-dimensional lattice polytopes that do not project onto a lower dimensional hollow lattice polytope.
\end{Theo}

The finitely many exceptions in the above theorem are called exceptional polytopes. 

\begin{Cor}\label{Cor: No higher preprojective grading if exceptional or projecting}
    Let $G \leq \SL_{n+1}(k)$ be finite abelian such that the skew-group algebra $R \ast G$ does not afford a higher preprojective cut. Let $P$ be the simplex associated to $R \ast G$. 
    \begin{enumerate}
        \item If $P$ is an exceptional polytope, the order of $G$ is bounded. 
        \item For all but finitely many $G$, the polytope $P$ admits a lower dimensional hollow lattice projection.
    \end{enumerate}
\end{Cor}

\begin{proof}
    To prove the first statement, simply note that the proof of \cite[Theorem 1.1]{NillZiegler} shows that the volume of an exceptional hollow polytope is bounded and hence the order of an abelian group $G \leq \SL_{n+1}(k)$ giving rise to such an exceptional polytope is bounded as well. The second statement follows directly from \Cref{Cor: Cut exists iff internal point exists} and \Cref{Theo: NillZiegler}. 
\end{proof}

\begin{Rem}
    The \emph{width} (see e.g.\ \cite[Section 1.4]{NillZiegler}) of a lattice polytope does not decrease under lattice projection. It therefore follows from \Cref{Theo: NillZiegler} that the width of a hollow lattice polytope of a fixed dimension is bounded. However, the width of the simplex $P$ of types is not bounded, so in this sense a generic simplex we consider has an internal lattice point, which means that generically the algebra $kQ/I$ we consider admits a higher preprojective cut. 
\end{Rem}

The hollow lattice polytopes in dimension $2$ are known, and lead to the classification of $3$-preprojective algebras of type $\Tilde{A}$ as given in \cite{DramburgGasanova}. In higher dimensions, we are not aware of a complete classification of the exceptional hollow polytopes. Let us conclude by comparing the infinite families from \Cref{Theo: NillZiegler} and from Thibault's criterion \Cref{Theo: Thibaults Criterion}. 

\begin{Rem}
    Let $G \leq \SL_{n+1}(k)$ be abelian such that the embedding of $G$ factors through some embedding of $\SL_{n_1}(k) \times \SL_{n_2}(k)$ with $n_1 + n_2 = n+1$ and $n_1, n_2 >0$. Then it follows from \Cref{Theo: Thibaults Criterion} that the simplex $P$ for $G$ is hollow, and one can show that $P$ projects onto a lower dimensional hollow lattice polytope. However, this is not an equivalence: The group $G \simeq C_2 \times C_2 \times C_2$ embeds into $\SL_4(k)$ in a unique way so that the embedding does factor through a product $\SL_{n_1}(k) \times \SL_{n_2}(k)$. The resulting simplex $P$ of types of cuts projects onto the simplex $P'$ of types of cuts for the group $C_2 \times C_2$ embedded in $\SL_3(k)$. It is easy to compute that $P'$ is the unique exceptional hollow lattice triangle. It follows from \Cref{Theo: NillZiegler} that $P$ is hollow, but the embedding of $G$ does not factor.   
\end{Rem}

We therefore note that it would be interesting to find a more refined version of \Cref{Theo: Thibaults Criterion} that covers more of the infinite families of hollow polytopes. Ideally, this would give a purely algebraic description of the polytopes that appear in the second part of \Cref{Cor: No higher preprojective grading if exceptional or projecting}. In particular, the orders of the lattice generators $\alpha_1, \ldots, \alpha_{n+1}$ modulo $L_1$ may be used to describe the width of the lattice simplex. Since the width of a hollow lattice polytope is bounded, this might be used to characterise these infinite families of hollow simplices.

\section{Toric geometry}\label{Sec: Toric geometry}
In this section, we want to highlight the connections between higher preprojective cuts and toric geometry. When $n=1$, these observations are part of the classical McKay correspondence. When $n=2$, dimer models can be used, and these observations were made by Nakajima in \cite{Nakajima}, based on the work of Ishii and Ueda \cite{IUDimerSpecial, IUDimerCrepant, IUModuliDimer}. We show that the same correspondence holds for any $n$: The lattice simplex $P$ of types of cuts is naturally isomorphic to the lattice simplex of junior elements in $G$, which in turn is the slice of the cone of the toric variety $X = \spec(R^G)$ at height $1$. For a general resource on the parts of toric geometry we use, we refer the reader to \cite[Chapter 1]{CLS}, and we will use the standard notation established in this book. In particular, that means that when we consider an affine toric variety $X$, we denote by $N$ the lattice of $1$-parameter subgroups and by $\sigma \subseteq N_\mathbb{R} = N \otimes_{\mathbb{Z}} \mathbb{R}$ the cone defining $X$. After choosing an appropriate basis for $N_\mathbb{R}$, this lattice corresponds to the lattice $N_G$ from \Cref{SSec: Hollow lattice polytopes}. As before, we assume that $G \leq \SL_{n+1}(k)$ is abelian and acts on the polynomial ring $R = k[x_1, \ldots, x_{n+1}]$. We assume without loss of generality that the action of $G$ is diagonalised, so that the invariant ring $R^G$ is generated by monomials. We begin with the following observation. 

\begin{Rem}
    The invariant ring $R^G$ embeds into $R \ast G$ via $r \to r \otimes 1$, and in this way we obtain precisely the center of the skew-group algebra $Z(R \ast G) = R^G$, which is true in more generality for any commutative domain $R$ and finite group $G$ \cite[Proposition 4.8]{MarcosMartinezVMartins}. Since $R^G$ is generated by monomials, it is the coordinate ring of an affine toric variety. The torus action on $X = \spec(R^G)$ is diagonal in the coordinates induced from the invariant monomials. We fix these coordinates. 
\end{Rem}

Next, we note that the center is automatically a graded subalgebra. 

\begin{Rem}
    Let $\Gamma = \bigoplus_{i \in \mathbb{Z}} \Gamma_i$ be a $\mathbb{Z}$-graded algebra, and let $z \in Z(\Gamma)$ be a central element. This is equivalent to $az=za$ for all $a \in \Gamma$, which is equivalent to $az=za$ for all homogeneous $a \in \Gamma_j$. Decompose $ z = \sum_{i \in \mathbb{Z}} z_i$ into its homogeneous components, and note that 
    \[ \sum_{i \in \mathbb{Z}} az_i = az = za = \sum_{i \in \mathbb{Z}} z_i a  \]
has $a z_i$ as its homogeneous component of degree $i+j$. But this component is also $z_i a$, and hence $z_i$ is central. Therefore, the center is generated by homogeneous elements, and hence is a graded subalgebra. When we refer to a grading of the center of a graded algebra, we are referring to this grading if not specified otherwise. 
\end{Rem}

To alleviate notation, we write monomials as $\mathbf{x}^e = \prod_{i = 1}^{n+1} x_i^{e_i}  \in R$, where $e \in \mathbb{N}^{n+1}$. 

\begin{Pro}\label{Pro: Degrees of central elements}
     Let $C$ be a periodic cut of type $\gamma$, and $\mathbf{x}^e \in R^G$ be an invariant monomial. Then the degree of $\mathbf{x}^e$ with respect to the grading induced by $C$ is
     \[ |\mathbf{x}^e| = \frac{\langle e, \gamma \rangle_{n+1}}{m}. \]
     In particular, two periodic cuts have the same type if and only if they induce the same grading on $R^G$. 
\end{Pro}

\begin{proof}
We note that it is enough to prove the claim for the pure powers $\mathbf{x}^e=x_i^m$ for $i \in \{ 1,\ldots, n+1\}$, and the rest follows by linearity.
For every vertex $v \in Q_0$, denote by $c_v$ the cycle starting at $v$ and consisting of $m$ many arrows of type $i$. Then the element $x_i^m$ corresponds via $R \ast G \simeq kQ/I$ to $\sum_{v \in Q_0} c_v$. We recall from \Cref{Pro: Height function values on L1} that each such cycle $c_v$ has the same degree, so it suffices to compute the degree for one such cycle. We fix without loss of generality the cycle $c = c_0$ which passes through $0$. 
Then we view $c$ instead on $\hat{Q}$ as a path $p$ of length $m$ from $0$ to $m \alpha_i$. We note that $h_C(m \alpha_i)$ can be computed as 
    \[ h_C(m \alpha_i) = m - (n+1) \cdot |c|, \]
    since $|c|$ is the number of cut arrows in $p$.
    Using \Cref{Pro: Height function values on L1}, for any $i=1,\ldots, n$ we get
    \[ m - (n+1) \cdot |c|= h_C(m \alpha_i) = \left\langle m \alpha_i,\mathbf{1}-\frac{n+1}{m}\gamma\right\rangle_{n} =m-(n+1)\gamma_i,  \]
    hence $|c|=\gamma_i$, as desired. For $i=n+1$ we have $m \alpha_{n+1}=-m \alpha_1-\ldots-m \alpha_n$. Using the above equations and equivariance of $h_C$, for $i=n+1$ we get
\begin{align*}
    m - (n+1) \cdot |c|&= h_C(m \alpha_{n+1})  
    =-\sum_{i=1}^{n}{h_C(m \alpha_i)}=-\sum_{i=1}^{n}{(m-(n+1)\gamma_i)}\\&=-(mn-(n+1)(\gamma_1+\ldots+\gamma_n))\\&=(n+1)(m-\gamma_{n+1})-mn=m-(n+1)\gamma_{n+1},
\end{align*}
hence $|c|=\gamma_{n+1}$, as desired.
This proves our claim. Since the formula for the degree of $\mathbf{x}^e$ only depends on $e$ and the type of the cut, and since $\gamma$ can be reconstructed from the degrees of the pure powers in $R^G$, we conclude that two cuts define the same grading if and only if they have the same type. \qedhere

\end{proof}

\begin{Rem}
    Note that the above proposition and its proof can be extended to the following statement. The polynomial ring $R \supset R^G$ becomes a graded ring when declaring the variables $x_i$ to have fractional degrees $\frac{\gamma_i}{m}$, and then $R^G$ becomes a graded subring that is supported only in integer degrees. 
\end{Rem}

The above proposition means that we can identify the types of cuts with certain gradings on $R^G$. Since we are in the toric case, this has an especially nice interpretation. 

\begin{Rem}
    Note that $\mathbb{Z}$-gradings of a coordinate ring $k[X]$ correspond to $k^\ast$-actions on the affine variety $X = \spec(k[X])$ by \cite[Proposition 4.7.3]{SGA3}. In detail, a $\mathbb{Z}$-grading on $k[X]$ is equivalent to a coaction morphism $k[X] \to k[X] \otimes_k k[t, t^{-1}]$ for the ring of Laurent polynomials. Taking the spectrum on both sides then gives an action of $k^\ast = \spec(k[t, t^{-1}])$ on $X$. Since our variety $X = \spec(R^G)$ is toric, with a fixed torus $\mathbb{T}$, and because the generating monomials of $R^G$ are homogeneous for the gradings we consider, we obtain that these gradings on $R^G$ correspond to $1$-parameter subgroups of $(X, \mathbb{T})$. Furthermore, $X$ is defined by the cone $\sigma \subseteq N_\mathbb{R}$, where $N = \Hom(k^\ast, \mathbb{T})$ is the lattice of $1$-parameter subgroups. We note that this lattice is isomorphic to the lattice we considered in \Cref{SSec: Hollow lattice polytopes}. 
\end{Rem}

\begin{Pro} \label{Pro: Polytopes agree}
    Let $P$ be the lattice simplex from \Cref{SSec: Hollow lattice polytopes}, whose lattice points correspond to the types of cuts on $R\ast G \simeq kQ/I$, and $N$ the lattice of $1$-parameter subgroups of the toric variety $X = \spec(R^G)$. Then the lattice simplex in $N$ given by intersecting the defining cone $\sigma \subseteq N_\mathbb{R}$ with the hyperplane corresponding to the degree condition $|\mathbf{x}^\mathbf{1}| = |\prod_{i = 1}^{n+1} x_i | = 1$, is isomorphic to $P$.  
\end{Pro}

\begin{proof}
    Using \Cref{Pro: Degrees of central elements}, we identify types with the induced gradings, and with lattice points in $N$. Under this identification, we have for any type $\gamma$ that $|\mathbf{x}^\mathbf{1}| = \frac{\langle \mathbf{1}, \gamma \rangle_{n+1} }{m} = 1$, so the types lie on the hyperplane specified by this degree condition. To see that the lattice point $\lambda$ corresponding to a type lies in the cone $\sigma$, it suffices to check that $\lambda$, seen as a $1$-parameter subgroup, converges. But this is immediate since the grading assigns non-negative degrees $d_i \geq 0$ to any invariant monomial, so in appropriately chosen coordinates, $\lambda$ acts by multiplication by $t^{d_i}$ for $t \in k^\ast$. 
\end{proof}

With these gradings, the Gorenstein ring $R^G$ becomes graded so that its Gorenstein parameter is $1$. We remind the reader that the $a$-invariant is the negative of the Gorenstein parameter, so equivalently we have $a= -1$. This corresponds precisely to the condition in \Cref{Theo: HPG is f.d. GP1} that $R \ast G$ has Gorenstein parameter $1$. 

\begin{Pro}
    Let $R^G$ be as before, graded by some cut. Then the graded canonical module of $R^G$ is the ideal generated by $(\Pi_{i= 1}^{n+1} x_i)$. In particular, as a graded Gorenstein ring, $R^G$ has Gorenstein parameter $1$.
\end{Pro}

\begin{proof}
    The fact that $R^G$ is Gorenstein is well known, since $G \leq \SL_{n+1}(k)$ \cite{Watanabe}. In particular, the graded canonical module $\omega$ is generated by a single element. By \cite[Theorem 6.3.5]{BrunsHerzog}, $\omega$ is generated by the monomials in the interior of $C^\vee$, and in our case it is clear that this is generated by $(\Pi_{i= 1}^{n+1} x_i)$. By \Cref{Pro: Degrees of central elements}, we have that $|\Pi_{i= 1}^{n+1} x_i| = 1$ for any cut, and therefore $\omega(-1) \simeq R$. Hence the Gorenstein paramater is $1$. 
\end{proof} 

For completeness, we give another interpretation more directly in line with the McKay correspondence: The types of (nontrivial) cuts correspond precisely to junior elements in $G$. This brings us back to the original construction from \cite{AIR}, since we made explicit use of the juniors in \Cref{Pro: Decreasing arrows cut} and hence in \Cref{Pro: Cut constr} to construct our cuts. 

\begin{Rem}\label{Rem: Ito-Reid interpretation}
    In \cite[Section 2.2]{ItoReid}, a different interpretation of the lattice of $1$-parameter subgroups is chosen. We recall that we have diagonalised the action of $G$, so that every element can be written as $g = \frac{1}{|g|}(e_1, \ldots, e_{n+1})$. Of course, this notation depends on a choice of root of unity, but if we fix the diagonalisation and the root of unity, we can, after choosing a basis, identify $N$ with the lattice $N_G$ from \Cref{SSec: Hollow lattice polytopes}, which is generated by $\mathbb{Z}^{n+1} + \{ \frac{1}{|g|}(e_1, \ldots, e_{n+1})^t \mid g \in G \} $ in $\mathbb{R}^{n+1}$, where we now view $\frac{1}{|g|}(e_1, \ldots, e_{n+1})^t$ as a column vector with rational entries. An element $g = \frac{1}{|g|}(e_1, \ldots, e_{n+1})$ is called \emph{junior} if $\frac{1}{|g|}(e_1 +  \ldots + e_{n+1}) = 1$, and our simplex $P$ of types of cuts is then the $N_G$-lattice simplex of junior elements in. Note that this means that the $N_G$-points of $P$, \emph{excluding the vertices}, correspond to the junior elements of $G$.  
\end{Rem}

We conclude by pointing out that by general toric geometry, the lattice vectors in $P$ respectively the vectors of junior elements in $N_G$ correspond to crepant divisors. 

\begin{Rem}
    Let $N_G$ be as above, and denote $\sigma_G \subseteq (N_G)_\mathbb{R}$ the cone defining $X=\spec(R^G)$. By \cite{ItoReid}, the lattice points in $\sigma_G$ corresponding to junior elements in $G$ are in bijection with crepant exceptional prime divisors of any resolution of $X$. Furthermore, it was shown in \cite[Section 3.2, Section 6]{YamagishiIAbelian} that the rays generated by these lattice vectors in $C_G$ correspond to exceptional crepant divisors in a relative minimal model.
\end{Rem}
   
\section{Mutation lattices}\label{Sec: Mutation lattices}
In this section, we recall mutation of cuts, show that cuts of a given type can be turned into a finite distributive lattice whose cover relations coincide with (nonzero) mutations, and hence that any two cuts of the same type are related by a sequence of mutations. We also give an explicit construction of the maximal element in the lattice. As in \Cref{Con: B' and B}, we fix a cofinite $L_1 \xhookrightarrow{B'} L_0$ and $m = |L_0/L_1|$. 

\subsection{Mutation}
In this subsection, we provide a summary of cut mutation and its relevance to the $n$-representation infinite algebras we consider. This mutation is a higher version of the reflection functors of Berstein, Gel'fand and Ponomarev \cite{BGP}. In modern language, these functors are rephrased via Auslander-Platzeck-Reiten tilting modules \cite{APR}, and we refer to $n$-APR tilting as introduced in \cite{IyamaOppermann}. We fix a cut $C$ on $Q$, and recall that $Q_C = (Q_0, Q_1 - C)$.

\begin{Con}
    We call a vertex $x$ in $Q_C$ \emph{mutable} if it is a source or a sink in $Q_C$. Note that if $x$ is a source, i.e.\ if all arrows $a \in Q_1$ with $t(a) = x$ are in $C$, then all arrows $a \in Q_1$ with $s(a) = x$ are not in $C$, since otherwise we could complete such an $a \in C$ into an elementary cycle, whose last arrow is also in $C$ by assumption. Therefore, we can construct a new cut  
    \[ \mu_x(C) = \{ a \in C \mid  t(a) \neq x  \} \cup \{ a \in Q \mid s(a) = x \}  \]
    which we call the \emph{mutation of $C$ at $x$}. We further call this a \emph{nonzero} mutation if $x \neq 0$. We also use the same terminology for the infinite cut quiver $\hat{Q}_C$. Note that two sources in $\hat{Q}_C$ can not differ by an element $\alpha_i$ for any $1 \leq i \leq n+1$. Therefore, when mutating at a source $x$ in $\hat{Q}_C$, we can and do mutate at all points $x + L_1$ simultaneously to preserve periodicity. Mutation at a sink is defined dually. 
\end{Con}

Clearly, mutation at a source $x$ in $Q_C$ turns it into a sink in $Q_{\mu_x(C)}$ and vice versa. 

\begin{Rem} We make several remarks on cut mutation. 
    \begin{enumerate}
        \item Cut mutation as defined above is interesting because it corresponds to $n$-APR tilting of the $n$-representation infinite algebra $kQ_C/I_C$, as defined in \cite{IyamaOppermann}. In particular, cut mutation from $C$ to $\mu_x(C)$ gives rise to a derived equivalence between the $n$-representation infinite algebras $kQ_C /I_C$ and $kQ_{\mu_x(C)} /I_{\mu_x(C)}$.
        \item It was pointed out to us by Ryu Tomonaga that there seems to be no published proof of the statement that $n$-APR tilting corresponds to cut mutation. Tomonaga gave a geometric proof in the special case of higher type $\Tilde{A}$ in \cite{TomonagaToricStacks}. We provide a proof here that we believe is well-known to experts. To keep the numbering consistent with the published version, we state the following proposition without a number. Note that while the proof is adapted to type $\Tilde{A}$, it holds more generally for $n$-representation infinite algebras coming from a cut $C$ on the quiver $Q$ assuming that for every sink $x \in Q_C$, every incoming arrow $y \to x \in Q$ is in $Q_C$.  
    \end{enumerate}
\end{Rem}

\begin{Pro*}
    Let $\Lambda = (\Bbbk Q/I)_C$ be the cut algebra for a higher preprojective cut $C$, and denote the indecomposable projective $\Lambda$ modules associated to the vertices of $Q$ by $\{P_y\}_{y \in Q_0}$. Let $x \in (Q_C)_0$ be a sink, and fix $P = P_x$ and $K = \bigoplus_{y \neq x} P_x$. Consider the $n$-APR tilted algebra $\Lambda' = \mu_x(\Lambda) = \operatorname{End}_\Lambda(K \oplus \tau_n^-(P))$. Then $\mu_x(\Lambda) \simeq (\Bbbk Q/I)_{\mu_x(C)} $ is given by the mutated cut $\mu_x(C)$ on $Q$.
\end{Pro*} 

\begin{proof} 
    To see that $\Lambda'$ is indeed $n$-representation infinite again and derived equivalent to $\Lambda$, see \cite{IyamaOppermann, MizunoYamaura}. 
    We interpret the $n$-APR tilt as a change of grading on $kQ/I = \Pi_{n+1}(\Lambda) = \Pi_{n+1}(\Lambda')$. 
    Recall that $\Pi_{n+1}(\Lambda)$ is generated in degrees $0$ and $1$. These homogeneous components are given by $\Hom_\Lambda(\Lambda, \Lambda)$ and $\Hom_\Lambda(\Lambda, \tau_n^-(\Lambda))$ respectively. For the purposes of choosing arrows in these spaces, it suffices to view this in the higher transjective component $\mathcal{U} \simeq \operatorname{grproj}(\Pi_{n+1}(\Lambda)$. Since $\Lambda$ and $\Lambda' = \mu_i^-(\Lambda)$ are derived equivalent, their higher transjective components agree, so in the following we work with $\Hom_\mathcal{U}$ and simply drop the subscript from $\Hom$. We decompose the homogeneous components further according to $\Lambda = K \oplus P$, and study the effect of replacing $P$ by $\tau_n^-(P)$. 
    We have  
    \begin{align*}
        \Hom(\Lambda, \Lambda) = &\Hom(P, P) \oplus \Hom (P, K) \\ &\oplus \Hom (K, P) \oplus \Hom (K, K) \\
        \Hom (\Lambda, \tau_n^- (\Lambda)) = &\Hom (P, \tau_n^-(P)) \oplus \Hom (P, \tau_n^-(K)) \\ &\oplus  \Hom (K, \tau_n^-(P)) \oplus \Hom (K, \tau_n^-(K)). 
    \end{align*}
    We write down the same decomposition of the two homogeneous components for the grading giving $\Lambda' $ and obtain
    \begin{align*}
        \Hom(\Lambda', \Lambda') = &\Hom(\tau_n^-(P), \tau_n^-(P)) \oplus \Hom (\tau_n^-(P), K) \\ &\oplus \Hom (K, \tau_n^-(P)) \oplus \Hom (K, K) \\
        \Hom (\Lambda', \tau_n^- (\Lambda')) = &\Hom (\tau_n^-(P), \tau_n^{-2}(P)) \oplus \Hom (\tau_n^-(P), \tau_n^-(K)) \\ &\oplus  \Hom (K, \tau_n^{-2}(P)) \oplus \Hom (K, \tau_n^-(K)). 
    \end{align*}
    We now compare these two decompositions. We first list the summands that obviously do not contribute any arrows. The summands $\Hom(P, \tau_n^-(P)) $ and $ \Hom(\tau_n^-(P), \tau_n^{-2}(P))$ contribute no arrows because the quiver $Q$ has no loops. The summand $\Hom(K,P) $ actually vanishes because $P$ is simple projective, and $\Hom(\tau_n^-(P), K)$ vanishes because the higher preprojective component is directed. Next we list the summands that agree and do not change degree. For both $\Lambda$ and $\Lambda'$, the summand $\Hom(K,K)$ appears in degree $0$, and the summand $\Hom(K, \tau_n^-(K))$ appears in degree $1$. These contribute the parts of $Q$ not involving the vertex $x$. Next, the summands $\Hom(P,P) = \Bbbk$ and $\Hom(\tau_n^-(P), \tau_n^-(P)) = \Bbbk$ stay unchanged and in degree $0$ agree as well, corresponding to the vertex $x$. The summands that differ are the following. The summand $\Hom(P, K)$ in degree $0$ for $\Lambda$ gives rise to the arrows coming into vertex $x$ in $Q$. This summand is isomorphic to $\Hom(\tau_n^-(P), \tau_n^-(K))$, which appears for $\Lambda'$ in degree $1$, meaning that after mutation, the incoming arrows have moved from degree $0$ to degree $1$. Similarly, the summand $\Hom(K, \tau_n^-(P))$ that appears in degree $1$ for $\Lambda$ appears in degree $0$ for $\Lambda'$, showing that the outgoing arrows have moved from degree $1$ to degree $0$. Two summands are not yet accounted for. The summand $\Hom(P, \tau_n^-(K))$ that sits in degree $1$ for $\Lambda$ does not contribute any arrows, since we know that for the sink $x$ in the cut quiver all arrows coming into $x$ necessarily have degree $0$. Finally, the summand $\Hom(K, \tau_n^{-2}(P))$ does not contribute arrows because it factors through $\tau_n^-(\Lambda)$, and this corresponds to the fact that after mutation, the vertex $x$ is a source in the quiver for $\Lambda'$ and that all outgoing arrows from $x$ are now in degree $0$. 
\end{proof}

Next we note that nonzero mutations have very nice effect on height functions. 

\begin{Lem}
\label{Rem: mutations vs heights}
    Let $C$ be a periodic cut on $\hat{Q}$, and $h_C$ the associated height function. If $0 \neq s \in \hat{Q}_C$ is a source then 
    \[ h_{\mu_s(C)} (x) = \begin{cases}
        h_C(x) + n+1, & x - s \in L_1 \\
        h_C(x), & \text{else}.
    \end{cases} \]
\end{Lem}

\begin{proof}
    If $x - s \not \in L_1$, an arbitrary path $p \colon 0 \to \cdots \to x$ leads to the same height increment for both $h_C$ and $h_{\mu_s(C)}$: All arrows which lie in both $\hat{Q}_C$ and $\hat{Q}_{\mu_s(C)}$ and all arrows which lie in both $C$ and $\mu_s(C)$ contribute the same height increment. Those which are in $C$ and not in $\mu_s(C)$ are followed by an arrow that necessarily lies in $\mu_s(C)$, and vice versa, so also those subpaths of length $2$ contribute the same height. 
    
    If instead $x -s \in L_1$, we conclude by the previous argument that $h_C$ and $h_{\mu(C)}$ agree on a vertex $x'$ which has an arrow $a \colon x' \to x$. Since $x$ is a source in $\hat{Q}_C$, it follows that $a \in C$ and $a \not \in \mu_s(C)$, so the functions $h_C$ and $h_{\mu_s(C)}$ differ at $x$ by exactly $n+1$.  
\end{proof}

\subsection{Lattices}
The goal of this section is to show that for any type $\gamma$, the set $\{ C\mid C \text{ is } L_1 \text{-periodic and } \theta(C)=\gamma\}$ is a finite distributive lattice whose cover relations are given by nonzero mutations. This insight comes from the heuristic that we normalised height functions to $h_C(0) = 0$, so we ``should not'' mutate at $0$. As a corollary, we show that any two cuts of the same type are related by a sequence of mutations. 

For the reader's convenience we briefly recall the relevant notions. A \emph{poset} is a pair $(V,\le)$ where $V$ is a set and $\le$ is a partial order on $V$, that is a reflexive, antisymmetric, and transitive binary relation on $V$. A poset $(V,\le)$ is called a \emph{lattice} if each subset of two elements $\{x,y\}\subseteq V$ has a least upper bound called a \emph{join}, denoted by $x\vee y$, and dually a greatest lower bound called a \emph{meet}, denoted by $x\wedge y$. A lattice $(V,\vee,\wedge)$ is called \emph{distributive} if for any $\{x,y,z\}\subseteq V$ one has $x\vee(y\wedge z)=(x\vee y)\wedge (x\vee z)$. Further background on poset and lattice theory can be found in \cite[Chapter 3]{Stanley}, and we use the standard notation established there. The following is well known. 

\begin{Pro}
 $(\mathbb{Z}^m,\le)$ is a distributive lattice with $x\vee y=\max(x,y)$ and $x\wedge y=\min(x,y)$, where $\leq$ and $\max$ and $\min$ are taken componentwise.   
\end{Pro}

We return to our situation of interest. As in \Cref{Con: B' and B}, we fix a cofinite $L_1 \xhookrightarrow{B'} L_0$ and a type $\gamma$. Furthermore, we enumerate the elements of $L_0/L_1$ by $\{1, \ldots, m\} $ to fix an isomorphism $\mathbb{Z}^{|L_0/L_1|} \simeq \mathbb{Z}^m$. We construct a function which generalises the relative height function of two perfect matchings in a dimer model. 

\begin{Pro}
 Let $C_1$ and $C_2$ be two $L_1$-periodic cuts of type $\gamma$ and let $h_{C_1}$ and $h_{C_2}$ be the corresponding height functions. Then $\frac{h_{C_1}-h_{C_2}}{n+1}$ is an integer valued $L_1$-periodic function on $L_0$.   
\end{Pro}

\begin{proof}
 It is clear that $(n+1) \mid (h_{C_1}(x)-h_{C_2}(x))$ for any $x\in L_0$ and any pair of cuts regardless of their types. Hence $\frac{h_{C_1}-h_{C_2}}{n+1}$ is integer valued. Let $x\in L_0$ and $y\in L_1$. Since $\theta(C_1)=\theta(C_2)=\gamma$, we know from \Cref{Pro: Height function values on L1} that $h_{C_1}(y)=h_{C_2}(y)=h_\gamma(y)$. Therefore: 
 \begin{align*}
  (h_{C_1}-h_{C_2})(x+y)&=h_{C_1}(x+y)-h_{C_2}(x+y)\\&=
 h_{C_1}(x)+h_{C_1}(y)-h_{C_2}(x)-h_{C_2}(y)\\&=h_{C_1}(x)+h_{\gamma}(y)-h_{C_2}(x)-h_{\gamma}(y)\\&=h_{C_1}(x)-h_{C_2}(x)=(h_{C_1}-h_{C_2})(x),   \end{align*}
 which shows $L_1$-periodicity of $\frac{h_{C_1}-h_{C_2}}{n+1}$.
\end{proof}

\begin{Con}
    Let $C_0$ be an $L_1$-periodic cut of type $\gamma$. We will fix it and call it the reference cut. Then for any $L_1$-periodic cut $C$ of type $\gamma$, the function 
    \[ \overline{h}_C = \overline{h}_{C, C_0} \colon L_0/L_1 \to \mathbb{Z}, x + L_1 \mapsto \frac{h_{C}(x)-h_{C_0}(x)}{n+1} \]
    is well-defined and called the \emph{relative height function of $C$ with respect to $C_0$}. We suppress the dependence on $C_0$ whenever possible. 
    Therefore, any $C$ can be identified with the corresponding vector $v_C = (\overline{h}_{C} (x+L_1))_{x+L_1 \in L_0/L_1}  \in \mathbb{Z}^m$. By \Cref{Cor: Bijection height functs and cuts}, the assignment $C \mapsto v_C \in \mathbb{Z}^m$ is injective and this way the set of all $L_1$-periodic cuts of type $\gamma$ becomes a subset of $\mathbb{Z}^m$.
\end{Con}

We will now show that the set of vectors $v_C$ is closed under taking $\min$ and $\max$, in other words that this subset is a sublattice of $\mathbb{Z}^m$.    

\begin{Pro}\label{Pro: Min and max of cuts}
Let $C_1$ and $C_2$ be two $L_1$-periodic cuts of type $\gamma$ and let $h_{C_1}$ and $h_{C_2}$ be the corresponding height functions. Then the functions $h_{\min}=\min(h_{C_1},h_{C_2})$ and $h_{\max}=\max(h_{C_1},h_{C_2})$, are $L_1$-equivariant height functions. Moreover, $C_{h_{\min}}$ and $C_{h_{\max}}$ have type $\gamma$.
\end{Pro}

\begin{proof}
    We will only prove the claim for $h_{\min}$ since the proof for $h_{\max}$ is analogous. It is obvious that $h_{\min}(0)=0$. Secondly, we will show that $h_{\min}(x+\alpha_i)\in\{h_{\min}(x)+1, h_{\min}(x)-n\}$ for any $x\in L_0$ and any $i\in \{1,\ldots, n+1\}$. Fix such an $x$ and such an $i$. We will assume without loss of generality that $h_{C_1}(x)\le h_{C_2}(x)$. If  $h_{C_1}(x)=h_{C_2}(x)$, we are done since $h_{\min}$, restricted to points $x$ and $x+\alpha_i$, equals either $h_{C_1}$ or $h_{C_2}$, restricted to the same pair of points. Therefore we will further assume $h_{C_1}(x)<h_{C_2}(x)$. If $h_{C_1}(x+\alpha_i)\le h_{C_2}(x+\alpha_i)$, we are done again by the same argument as before. Therefore, the only interesting case is $h_{C_1}(x)<h_{C_2}(x)$ and  $h_{C_1}(x+\alpha_i)> h_{C_2}(x+\alpha_i)$. These two conditions imply $h_{C_1}(x+\alpha_i)=h_{C_1}(x)+1$ and $h_{C_2}(x+\alpha_i)=h_{C_2}(x)-n$. But then $h_{C_1}(x)<h_{C_2}(x)=h_{C_2}(x+\alpha_i)+n<h_{C_1}(x+\alpha_i)+n=h_{C_1}(x)+1+n$, which in particular gives $h_{C_1}(x)<h_{C_2}(x)<h_{C_1}(x)+1+n$, in other words, $0<h_{C_2}(x)-h_{C_1}(x)<n+1$. This is a contradiction since for any pair of cuts $C_1$ and $C_2$ and for any point $x\in L_0$ we have $(n+1)|(h_{C_2}(x)-h_{C_1}(x))$.
 
    Next, we need to show equivariance of $h_{\min}$. Let $x\in L_0$ and $y\in L_1$. Since $\theta(C_1)=\theta(C_2)=\gamma$, we know from \Cref{Pro: Height function values on L1} that $h_{C_1}(y)=h_{C_2}(y)=h_\gamma(y)$. Therefore, \begin{align*}
        h_{\min}(x+y)&=\min(h_{C_1}(x+y),h_{C_2}(x+y))\\&=\min(h_{C_1}(x)+h_\gamma(y), h_{C_2}(x)+h_\gamma(y))\\&=\min(h_{C_1}(x),h_{C_2}(x))+h_\gamma(y)\\&=h_{\min}(x)+h_{\min}(y).
    \end{align*}
    Finally, the height function $h_{\min}$ corresponds to a cut of type $\gamma$, which follows from \Cref{Cor: type from L1 height}. 
\end{proof}

The above clearly translates to relative height functions and hence to the vectors $v_C$. Note that there are finitely many $L_1$-periodic cuts of a given type, hence we obtain the following.

\begin{Cor} \label{Cor: Cuts are lattice}
    The set $\{ v_C \mid C \text{ is an } L_1\text{-periodic cut, } \theta(C) = \gamma \}$ is a finite sublattice of $(\mathbb{Z}^m, \leq)$. Therefore, it is a finite distributive lattice.  
\end{Cor}

We transfer notation along the bijection $C \mapsto v_C$.

\begin{Rem}
    It follows that the set $\{ C \mid C \text{ is an } L_1\text{-periodic cut, } \theta(C) = \gamma \}$ is a finite distributive lattice, where we write $C_1 \preceq C_2$ iff $v_{C_1} \leq v_{C_2}$. The join and meet of two cuts $C_1$ and $C_2$ is given by the cuts $C_1 \vee C_2$ and $C_1 \wedge C_2$ defined by the height functions $h_{\max}$ and $h_{\min}$ from \Cref{Pro: Min and max of cuts}. These satisfy $v_{C_1\wedge C_2}=\min (v_{C_1},v_{C_2})$ and $v_{C_1\vee C_2}=\max (v_{C_1},v_{C_2})$. We therefore refer to the lattice of cuts of type $\gamma$ from now on. 
\end{Rem}

Next, we will explore the cover relations of the lattice in question. Recall that a cover relation is a relation of the form $C_1 \prec C_2$ such that there is no $C_3$ with $C_1 \prec C_3 \prec C_2$. A cover relation is denoted by $C_1 \precdot C_2$. From now on we will mostly be interested in the case when $\gamma$ is positive, in other words, $\gamma_i>0$ for all $i\in \{1,\ldots, n+1\}$. In this case cover relations have a particularly nice interpretation. We will first briefly note which effect a mutation of a cut $C$ has on the vector $v_C$. This is a reformulation of \Cref{Rem: mutations vs heights}.

\begin{Lem}
\label{Rem: mutation effect on v}
Let $C_1$ and $C_2$ be two $L_1$-periodic cuts of a positive type $\gamma$. Assume $C_1$ can be obtained from $C_2$ via mutation at a nonzero sink $s$. Then $v_{C_1}=v_{C_2}-e_i$, where $e_i$ is the $i$th standard vector, with $i$ being the position indexing the point $s+L_1$.
\end{Lem}

\begin{Theo} \label{Theo: mutations are covers}
    Let $C_1$ and $C_2$ be two $L_1$-periodic cuts of a positive type $\gamma$. Then $C_1 \precdot C_2$ if and only if $C_1$ can be obtained from $C_2$ via a nonzero sink mutation (or, equivalently, $C_2$ can be obtained from $C_1$ via a nonzero source mutation).
\end{Theo}
 
\begin{proof}
    First, we assume that $C_1$ can be obtained from $C_2$ by a mutation at a nonzero sink $s$. Then by \Cref{Rem: mutation effect on v} we have $v_{C_1}=v_{C_2}-e_i$ for some $i$. Since there are no integer vectors strictly between $v_{C_1}$ and $v_{C_2}$,
    it follows that $C_1 \precdot C_2$. 
    Conversely, assume that $C_1 \precdot C_2 $, then in particular it follows that $h_{C_1}(x) \le h_{C_2}(x)$ for all $x \in L_0$. If we have $h_{C_1}(x) = h_{C_2}(x)$ for all $x \in L_0$, then $C_1 = C_2$, which is a contradiction, hence we fix some $x \in L_0$ such that $h_{C_1}(x) < h_{C_2}(x)$. Next, recall that it follows from \Cref{Lem: Positive type is acyclic} that there are no infinite paths in the infinite cut quiver $\hat{Q}_{C_2}$. We follow any path starting at $x$ in $\hat{Q}_{C_2}$, and this path must be finite, which means it ends at a sink $s \in L_0$. Since we followed a path in $\hat{Q}_{C_2}$, the height function $h_{C_2}$ increased along each arrow, so we have $h_{C_1}(s) < h_{C_2}(s)$. Note that $s \not \in L_1$, since otherwise \Cref{Pro: Height function values on L1} gives $h_{C_1}(s) = h_\gamma(s) = h_{C_2}(s)$. That means we can perform a nonzero sink mutation of $C_2$ at $s$ and obtain a new cut $C_3$. Note that $h_{C_1}(x)\le h_{C_2}(x)$ for all $x\in L_0$, combined with $h_{C_1}(s) < h_{C_2}(s)$, implies $v_{C_1}+e_i\le v_{C_2}$. Here $e_i$ is the $i$th standard vector with $i$ being the position indexing point $s+L_1$. Together with that, \Cref{Rem: mutation effect on v} gives $v_{C_3}=v_{C_2}-e_i$ with the same index $i$, hence we get $v_{C_1}\le v_{C_3}<v_{C_2}$. Since we assumed $C_1 \precdot C_2$, we cannot have $v_{C_1}<v_{C_3}<v_{C_2}$, hence $v_{C_1}=v_{C_3}$. 
    It follows that $C_1=C_3$, which by construction was obtained from $C_2$ by a single nonzero sink mutation.
\end{proof}

\begin{Cor}\label{Cor: Mutation is transitive}
Any pair of $L_1$-periodic cuts $C_1$ and $C_2$ of a positive type $\gamma$ can be obtained one from another via a sequence of (nonzero) mutations.    
\end{Cor}

\begin{proof}
Since $C_1\wedge C_2\preceq C_1$ and $C_1\wedge C_2\preceq C_2$, there are chains of covering relations $C_1\succdot \ldots \succdot C_1\wedge C_2\precdot\ldots\precdot C_2$. Each of these relations corresponds to a nonzero mutation and thus the statement holds.    
\end{proof}

Let us note an interesting consequence of the above. 

\begin{Rem}\label{Rem: Answering AIR}
    In \cite{AIR}, Amiot, Iyama and Reiten construct a cut for the skew-group algebra $R \ast G$ where $G$ is a cyclic group generated by a junior element. This is the same construction we used in \Cref{Pro: Decreasing arrows cut}. They furthermore note that it would be interesting to generalise their result to non-cyclic groups. Our results show that for abelian groups, there are no other cuts up to $n$-APR tilting and skewing. More precisely, it follows from \Cref{Pro: Cut constr} that the cuts we first construct arise by decomposing the abelian group $G \simeq H \times K$ so that $H$ is cyclic. Then $R \ast H$ is given a cut via Amiot-Iyama-Reiten's procedure, and this is extended to $R \ast G$ via the isomorphism $(R \ast G) \simeq (R \ast H ) \ast K$, see \cite[Section 5]{DramburgGasanova} for details. Then \Cref{Cor: Mutation is transitive} shows that every cut is mutation equivalent to one induced from Amiot-Iyama-Reiten's procedure, and hence any $n$-representation infinite algebra $\Lambda$ is $n$-APR tilting equivalent to some $\Lambda' \ast K$ where $K$ is an abelian group acting on the $n$-representation infinite algebra $\Lambda'$ obtained from Amiot-Iyama-Reiten's construction. 
\end{Rem}

\subsection{Constructing the minimal and the maximal element of the mutation lattice for a positive type}\label{SSec: Constructing min and max}
For a fixed type $\gamma$, the mutation lattice is finite and distributive. Therefore, it has a unique maximal element, the construction of which can be done explicitly if $\gamma>0$ and is the goal of this subsection. In a dual procedure (see \Cref{Cor: minimal element}), one can construct the unique minimal element. We fix a type $\gamma$ (which does not yet need to be positive), and use the following notation throughout this subsection: for $x\in \mathbb{Z}^n$ and $z\in\mathbb{Z}$ we will write $(x,0)=(x_1,\ldots, x_n,0)$ and $u_z(x)=(x,0)-z\textbf{1}$. Further, we consider the linear map $\psi(y)= (y, 0) - \frac{\langle y, \gamma \rangle_n}{m} \mathbf{1}$ and write $L=\psi(L_1)\subset \mathbb{Z}^{n+1}$. Note that $\langle l,\gamma\rangle_{n+1}=0$ for any $l\in L$. For every $x\in L_0$, let $S_x=\{z\in\mathbb{Z}\mid \exists l\in L: u_z(x)\ge l\}$. 

\begin{Con}
\label{Con: p}
    For every $x \in L_0$, the set $S_x$ is nonempty and bounded above. Therefore, we can construct the function 
    \[ p \colon L_0 \to \mathbb{Z}, x \to \max S_x.\]  
\end{Con}

\begin{proof}
    To see that $S_x$ is nonempty, simply note that by taking $z$ small enough, one can always reach a situation when $u_z(x)\ge (0,\ldots,0)\in L$. Now we show that $S_x$ is bounded above. Indeed, there exists $z_0$ such that for any $z\ge z_0$ all the coordinates of $u_z(x)$ are negative. Fix any $z_1 \ge z_0$. We claim that $z_1 \not \in S_x$. Assume on the contrary that $u_{z_1}(x)\ge l$ for some $l\in L$. Then $u_{z_1}(x)=l+\varepsilon$, where $\varepsilon$ is some vector with nonnegative coordinates. Then $0=\langle l, \gamma \rangle_{n+1}=\langle u_{z_1}(x)-\varepsilon, \gamma \rangle_{n+1}$, which is a contradiction since all the coordinates of $u_{z_1}(x)-\varepsilon$ are negative, and $\gamma$ is a nonzero nonnegative vector. Therefore, $S_x$ is a nonempty bounded above set of integers. Therefore, it has a maximum and hence the function $p$ is well defined.
\end{proof}

Next, we need to show the following properties of $p$. 

\begin{Lem}
\label{Lem: properties of p}
    For $y \in L_0$, we have $(1)\Rightarrow (2)\Rightarrow (3)$, where 
    \begin{enumerate}
    \item $u_{p(y)}(y)\in L$,
    \item $y\in L_1$,
    \item $p(y)=\frac{\langle y,\gamma\rangle_n}{m}$.
        
    \end{enumerate}
\end{Lem}

\begin{proof} 
    $(1) \Rightarrow (2)$: Let $u_{p(y)}(y)\in L$. Then $u_{p(y)}(y)=\psi(y')$ for some $y'\in L_1$, which gives
    \[
    (y,0)-p(y)\textbf{1}=(y',0)-\frac{\langle y',\gamma \rangle_n}{m}\textbf{1},
    \]
    which, if one compares the last coordinates, gives $p(y)=\frac{\langle y',\gamma \rangle_n}{m}$ and hence $(y,0)=(y',0)$ and thus $y=y'\in L_1$. \medskip \newline 
    $(2) \Rightarrow (3)$: Let $y\in L_1$. Then
    \begin{align*} \hspace{0.8cm}
    p(y) &= \max\{z\in\mathbb{Z}\mid \exists l\in L: (y,0)-z\textbf{1}\ge l \}\\&= \max\left\{z\in\mathbb{Z}\mid \exists l\in L: (y,0)-\left(z+\frac{\langle y,\gamma\rangle_n}{m}\right)\textbf{1}\ge l \right\}+\frac{\langle y,\gamma\rangle_n}{m}\\&=\max\left\{z\in\mathbb{Z}\mid \exists l\in L: \psi(y)-z\textbf{1}\ge l \right\}+\frac{\langle y,\gamma\rangle_n}{m}\\&=\max\left\{z\in\mathbb{Z}\mid \exists l'\in L-\psi(y): -z\textbf{1}\ge l' \right\}+\frac{\langle y,\gamma\rangle_n}{m}\\&=\max\left\{z\in\mathbb{Z}\mid \exists l\in L: -z\textbf{1}\ge l \right\}+\frac{\langle y,\gamma\rangle_n}{m}.
\end{align*}
Now, we claim that $\max\left\{z\in\mathbb{Z}\mid \exists l\in L: -z\textbf{1}\ge l \right\}=0$. Indeed, if $z_0\in \mathbb{Z}$ belongs to this set, we get $-z_0\textbf{1}=l+\varepsilon$, where $\varepsilon$ is a nonnegative vector, and then
\[ -z_0m=\langle-z_0\textbf{1},\gamma\rangle_{n+1}=\langle l,\gamma\rangle_{n+1}+\langle \varepsilon,\gamma\rangle_{n+1}=\langle\varepsilon,\gamma\rangle_{n+1}\ge 0, \]
which gives $z_0\le 0$. Clearly, $0$ belongs to the above set, and thus we get $p(y)=\frac{\langle y,\gamma\rangle_n}{m}$.
\end{proof}

\begin{Rem}
 In fact, it is possible to show that $(1)\Leftrightarrow (2)\Rightarrow (3)$ in \Cref{Lem: properties of p}, and that the remaining implication $(3)\Rightarrow (1)$ holds if additionally $\gamma$ is assumed to be positive. However, we only ever need the implications $(1)\Rightarrow (2)\Rightarrow (3)$.   
\end{Rem}
Now we are ready to define a candidate height function for our maximal element.  

\begin{Theo}
\label{Theo: constructing max}
 Let $\gamma$ be a (not necessarily positive) type and let 
 \[ h \colon L_0 \to \mathbb{Z}, x \mapsto \langle u_{p(x)}(x), \mathbf{1} \rangle_{n+1}.\] 
 Then $h$ is an $L_1$-equivariant height function and $C_h$ is of type $\gamma$. 
\end{Theo}

\begin{proof} We first show that $h$ is a height function. Since $0\in L_1$, the implication $(2)\Rightarrow (3)$ of \Cref{Lem: properties of p} tells us that $p(0)=0$ and thus $h(0)=0$. Now we want to compare $h(x+\alpha_i)$ with $h(x)$ for all $i \in \{1,\ldots, n+1 \}$. First we will consider the case $i \in \{1,\ldots, n\}$, where we can without loss of generality assume $i=1$. We will first show that $p(x+\alpha_1)\in \{p(x),p(x)+1\}$. Indeed,
\begin{align*}
p(x)&=\max\{z\in\mathbb{Z}\mid \exists l\in L: (x,0)-z\textbf{1}\ge l \} \\
    &\le \max\{z\in\mathbb{Z}\mid \exists l\in L: (x,0)-z\textbf{1}+(1,0,\ldots,0)\ge l \}\\
    &\le \max\{z\in\mathbb{Z}\mid \exists l\in L: (x,0)-z\textbf{1}+\textbf{1}\ge l \}\\
    &=\max\{z\in\mathbb{Z}\mid \exists l\in L: (x,0)-(z-1)\textbf{1}\ge l \}\\
    &=\max\{z\in\mathbb{Z}\mid \exists l\in L: (x,0)-z\textbf{1}\ge l \}+1 = p(x)+1. 
\end{align*}
Since $p(x+\alpha_1)=\max\{z\in\mathbb{Z}\mid \exists l\in L: (x,0)-z\textbf{1}+(1,0,\ldots,0)\ge l \}$, we conclude that $p(x+\alpha_1)\in \{p(x),p(x)+1\}$. Recall that 
\[h(x+\alpha_1)-h(x)=\langle u_{p(x+\alpha_1)}(x+\alpha_1)-u_{p(x)}(x),\textbf{1}\rangle_{n+1}.\]
If $p(x+\alpha_1)=p(x)$, we get 
\begin{align*}
    u_{p(x+\alpha_1)}(x+\alpha_1)-u_{p(x)}(x)&=u_{p(x)}(x+\alpha_1)-u_{p(x)}(x)\\&=u_{p(x)}(x)+(1,0,\ldots,0)-u_{p(x)}(x)=(1,0,\ldots,0)
\end{align*}
and hence $h(x+\alpha_1)-h(x)=1$. If $p(x+\alpha_1)=p(x)+1$, we get
\begin{align*}
   u_{p(x+\alpha_1)}(x+\alpha_1)-u_{p(x)}(x)&=u_{p(x)+1}(x+\alpha_1)-u_{p(x)}(x)\\&=u_{p(x)+1}(x)+(1,0,\ldots,0)-u_{p(x)}(x)\\&=u_{p(x)}(x)-\textbf{1}+(1,0,\ldots,0)-u_{p(x)}(x)=(0,-1,\ldots,-1)
\end{align*}
and hence $h(x+\alpha_1)-h(x)=-n$.

The case $i=n+1$ can be treated analogously, showing that 
\[ p\left(x- \left(\sum_{i= 1}^n \alpha_i \right)\right) = p(x + \alpha_{n+1}) \in \{p(x)-1,p(x)\},\] 
and that $h(x+\alpha_{n+1})-h(x)=1$ when $p(x+\alpha_{n+1})=p(x)-1$ and $h(x+\alpha_{n+1})-h(x)= -n$ when $p(x+\alpha_{n+1})=p(x)$. Therefore, we conclude that $h$ is a height function. 

Now we want to show that $h$ is $L_1$-equivariant. In order to show this, it is enough to show $L_1$-equivariance of $p(x)$. Indeed, $p(x+y)=p(x)+p(y)$ implies 
\begin{align*}
    h(x+y)&=\langle u_{p(x+y)}(x+y),\textbf{1}\rangle_{n+1}\\&=\langle u_{p(x)+p(y)}(x+y),\textbf{1}\rangle_{n+1}=\langle u_{p(x)}(x)+u_{p(y)}(y),\textbf{1}\rangle_{n+1}\\&=\langle u_{p(x)}(x),\textbf{1}\rangle_{n+1}+\langle u_{p(y)}(y),\textbf{1}\rangle_{n+1}=h(x)+h(y).
\end{align*}

Now, the implication $(2)\Rightarrow (3)$ of \Cref{Lem: properties of p} tells us that for any $y\in L_1$ we have $p(y)=\frac{\langle y,\gamma\rangle_n}{m}$ and hence $u_{p(y)}(y)=\psi(y)$. Therefore, 
\begin{align*}
 p(x+y)&=\max\{z\in\mathbb{Z}\mid \exists l\in L: (x+y,0)-z\textbf{1}\ge l\}\\&=\max\{z\in\mathbb{Z}\mid \exists l\in L: (x,0)-(z-p(y))\textbf{1}+(y,0)-p(y)\textbf{1}\ge l\}\\&=\max\{z\in\mathbb{Z}\mid \exists l\in L: (x,0)-(z-p(y))\textbf{1}+\psi(y)\ge l\}\\&=\max\{z\in\mathbb{Z}\mid \exists l'\in L-\psi(y): (x,0)-(z-p(y))\textbf{1}\ge l'\}\\&=\max\{z\in\mathbb{Z}\mid \exists l\in L: (x,0)-(z-p(y))\textbf{1}\ge l\}\\&=\max\{z\in\mathbb{Z}\mid \exists l\in L: (x,0)-z\textbf{1}\ge l\}+p(y)=p(x)+p(y).   
\end{align*}

We conclude that $h(x)$ is an equivariant height function. Therefore, it defines an $L_1$-periodic cut $C_h$. As before, \Cref{Lem: properties of p} gives us $u_{p(y)}(y)=\psi(y)$ for every $y \in L_1$, and we compute
\begin{align*}
 h(y)&=\langle u_{p(y)}(y),\textbf{1}\rangle_{n+1}=\left\langle(y,0)-\frac{\langle y,\gamma\rangle_{n}}{m}\textbf{1},\textbf{1}\right\rangle_{n+1}\\&=\langle(y,0),\textbf{1}\rangle_{n+1}-\frac{\langle y,\gamma\rangle_{n}}{m}\langle\textbf{1},\textbf{1}\rangle_{n+1}\\&=\langle y,\textbf{1}\rangle_n-\frac{n+1}{m}\langle y,\gamma\rangle_{n}=\left\langle y,\textbf{1}-\frac{n+1}{m}\gamma\right\rangle_n,   
\end{align*}
hence $C_h$ has type $\gamma$ by \Cref{Cor: type from L1 height}. 
\end{proof}

\begin{Pro}\label{Pro: sources in maximal element are 0}
    Let $h$ be as in \Cref{Theo: constructing max}. If $x$ is a source of $\hat{Q}_{C_h}$, then $x \in L_1$. 
\end{Pro}

\begin{proof}
Assume $x=(x_1,\ldots,x_n)$ is a source. Then all arrows $x-\alpha_i\to x$ for $i \in \{ 1,\ldots,n+1 \}$ belong to $\hat{C}_h$. In other words, $h(x)-h(x-\alpha_i)=-n$ for all $i \in \{1,\ldots,n+1\}$. From the proof of \Cref{Theo: constructing max} we know that this happens if and only if $p(x)=p(x-\alpha_i)+1$ for $i \in \{1,\ldots,n\}$ and $p(x)=p(x-\alpha_{n+1})$. By definition of $p(x)$ we know that there is some $l=(l_1,\ldots,l_{n+1})\in L$ such that $(x,0)-p(x)\textbf{1}\ge l$. Since $p(x-\alpha_1)=p(x)-1$, we get that $(x-\alpha_1,0)-p(x)\textbf{1}$ can not be greater or equal to any element in $L$, in particular, we get
\begin{align*}
  l\not\le(x-\alpha_1,0)-p(x)\textbf{1}=(x,0)-p(x)\textbf{1}-(1,0,\ldots,0).  
\end{align*}
Combining this with $l\le (x,0)-p(x)\textbf{1}$ gives $l_1=x_1-p(x)$. The same holds for all $i \in \{1,\ldots, n\}$. Finally, we have $p(x)=p(x-\alpha_{n+1})$, which means $(x-\alpha_{n+1},0)-(p(x)+1)\textbf{1}$ can not be greater or equal to any element in $L$, in particular, we get
\begin{align*}
  l\not\le(x+\alpha_1
+\ldots+\alpha_n,0)-(p(x)+1)\textbf{1}=(x,0)-p(x)\textbf{1}-(0,\ldots,0,1).  
\end{align*}
Combining this with $l\le (x,0)-p(x)\textbf{1}$ gives $l_{n+1}=-p(x)$. Therefore, we conclude that $u_{p(x)}(x)=(x,0)-p(x)\textbf{1}=l\in L$, and then the implication $(1)\Rightarrow (2)$ of \Cref{Lem: properties of p} implies that $x\in L_1$.
\end{proof}

To show that $C_h$ is maximal, we need to assume that $\gamma$ is positive, since then mutations and cover relations coincide by \Cref{Theo: mutations are covers}. 

\begin{Cor}
If $\gamma$ is positive, then $C_h$ constructed in \Cref{Theo: constructing max} is the unique maximal element in its mutation lattice. Furthermore, $Q_{C_h}$ has a unique source, located at $0$. 
\end{Cor}

\begin{proof}
Assume that $C_h$ is not maximal, then there exists $D\succ C_h$ and a chain of cover relations $C_h\precdot\ldots\precdot D$. Then by \Cref{Theo: mutations are covers}, $Q_{C_h}$ has to have a nonzero source, which is a contradiction with \Cref{Pro: sources in maximal element are 0}. Finally, note that $Q_{C_h}$ is acyclic by \Cref{Lem: Positive type is acyclic}, and since $Q_{C_h}$ is in addition finite, it must have a source, which must be located at $0$.  
\end{proof}

Next, we note that the lattice of cuts is self-dual. This follows essentially from the fact that $x \mapsto -x$ is an automorphism of $L_0$ which stabilizes $L_1$, and gives rise to automorphisms of $\hat{Q}$ and $Q$.   

\begin{Con}
Let $C$ be an $L_1$-periodic cut of type $\gamma$. We define the \emph{dual} cut $C^*$ as follows: for each $x\in L_0$ and for each $i\in\{1,\ldots, n+1\}$ we have $x\to x + \alpha_i\in C^*$ if and only if $-x- \alpha_i\to -x\in C$. One can easily see that $C^*$ is indeed a cut, which in addition is $L_1$-periodic and has type $\gamma$. Furthermore, one can observe that $h_{C}(-x)=-h_{C^*}(x)$ for all $x\in L_0$. In particular, a vertex $x$ is a source (respectively, sink) in $Q_C$ if and only if $-x$ is a sink (respectively, source) in $Q_{C^*}$, and $(\mu_{x}(C))^*=\mu_{-x}(C^*)$.  
\end{Con}

As an immediate consequence, we obtain the following. 
\begin{Cor}\label{Cor: minimal element}    
If $\gamma$ is a (not necessarily positive) type vector, the mutation lattice of cuts of type $\gamma$ is self-dual.
\end{Cor}

In particular, the minimal element can be obtained by taking the dual of the maximal element. For completeness, we also give the direct construction of the minimal element in case $\gamma$ is positive. All relevant proofs can be transferred from the construction of the maximal element. In particular, we recall the notation $u_z(x)=(x,0)-z\textbf{1}$. The reader will notice the duality in the construction as well. 

\begin{Con}
In order to construct the minimal element, we proceed as for the maximal element. Define the set $S'_x=\{z\in \mathbb{Z}\mid\exists l \in L: u_z(x)\le l\}$. One can prove that $\min S'_x$ exists, and hence we define $p'(x)=\min S'_x$ for every $x \in L_0$. The height function of the minimal element is then $h'(x)=\langle u_{p'(x)}(x), \mathbf{1}\rangle_{n+1}$. 
\end{Con}

If $\gamma$ is not positive, the corresponding finite distributive lattice is harder to understand. Cover relations are not given by mutations, since none of the cut quivers have sources or sinks. The construction from \Cref{Theo: constructing max} still produces a cut, but we are aware neither of a proof that it is the maximal element, nor of a counterexample.
\begin{Ques}
Let $\gamma$ be a nonpositive type and let $C_h$ be the cut constructed in \Cref{Theo: constructing max}. Is it true that $C_h$ is the maximal element in the corresponding mutation lattice?
\end{Ques}

\subsection{An example}
We conclude by going through one example in detail. For this subsection, we let $G \simeq C_2 \times C_6 = \langle g \rangle \times \langle h \rangle$ be non-cyclic of order $12$, where the generators $g$ and $h$ have order $2$ and $6$ respectively. We embed $G$ into $\SL_3(\Bbbk)$ by sending $g$ to the matrix $\frac{1}{2}(1,1,0)$ and $h$ to $\frac{1}{6}(1,4,1)$, which we will from now on also address by $g$ and $h$ respectively. As we saw in \Cref{Exp: C2 x C6}, this subgroup gives rise to the matrix 
\[ B' =  \begin{pmatrix}
   6 & 4 \\ 0& 2 
\end{pmatrix}.\]
We will later use the fact that the same lattice is generated by the column space of
\[  \begin{pmatrix}
   0 & -2 \\ 6 & 2 
\end{pmatrix}.\]
This subgroup exhibits some extra symmetry, which we will use to simplify many of the computations. 

\begin{Rem}
    Let $G = \langle \frac{1}{2}(1,1,0), \frac{1}{6}(1,4,1) \rangle \leq \SL_3(\Bbbk)$ as above. We have the standard action of the symmetric group $S_3$ on $R = \Bbbk[x_1, x_2, x_3]$ by permuting the variables. This gives rise to the standard conjugation action of $S_3$ on $\SL_3(\Bbbk)$. Furthermore, this induces an action on $L_0$ by permuting $\{ \alpha_1, \alpha_2, \alpha_3\}$, which extends to an action on $\hat{Q}$.   
    The relevant observation is now that $G$ is stabilised by this action. Equivalently, $L_1 = \operatorname{Col}_\mathbb{Z}(B')$ is stabilised. Therefore, the action descends to one on $Q_G = \hat{Q}/L_1$.  

    In particular, we see that if $C$ is a cut of type $(\gamma_1, \gamma_2, \gamma_3)$ and $\sigma \in S_3$, then the set $\sigma(C)$ is a cut of type $(\gamma_{\sigma(1)}, \gamma_{\sigma(2)}, \gamma_{\sigma(3)})$, and the $2$-representation infinite algebras arising from $C$ and $\sigma(C)$ are isomorphic. 
\end{Rem}

We first compute all possible type vectors and draw the simplex of types. This has already been done in \cite{DramburgGasanova}.

\begin{Rem}
    With the above setup, any type vector $(\gamma_1, \gamma_2, \gamma_3)$ for $G \leq \SL_3(\Bbbk)$ satisfies $\gamma_1 + \gamma_2 + \gamma_3 = 12$. Furthermore, by \Cref{Theo: Divisibility conditions}, we need that the entries of  
    \[ (\gamma_1, \gamma_2) \cdot B' = (6 \gamma_1, 4 \gamma_1 + 2 \gamma_2) \]
    are divisible by $12$. 
    We therefore see that $\gamma_1$ needs to be even. Since it is also at most $12$, we have $\gamma_1 \in \{0, 2, 4, 6, 8, 10, 12 \}$. Next, we note that the condition $12 \mid 4 \gamma_1 + 2 \gamma_2$ simplifies to $6 \mid 2 \gamma_1 + \gamma_2$. Keeping in mind that $\gamma_1 + \gamma_2 + \gamma_3 = 12$, we can then list all types as follows.  
    \begin{align*}
        &(0, 0, 12) & &(0, 6, 6) & &(0, 12, 0) \\
        &(2, 2, 8) & &(2, 8, 2) & & \\
        &(4, 4, 4)  & &  & &\\
        &(6, 0, 6)  & &(6, 6, 0) & & \\
        &(8, 2, 2) & & & & \\
        &(12, 0, 0) & &  & & 
    \end{align*}
    We notice immediately that the symmetric group acts on the type vectors by permuting entries, as expected. Arranging the vectors in a simplex as in \Cref{Theo: Types are dual lattice points} then gives the following picture. 
    \[
\begin{tikzpicture}[-,scale=0.4]

\begin{scope}[rotate=-15.4,inner sep=1.5mm]

\begin{scriptsize}

\coordinate (O) at (0, 0, 0);
\coordinate (X) at (0.25,0,-0.25);
\coordinate (Y) at (0, -0.25, 0.25);
\coordinate (Z) at (-0.25, 0.25, 0);

\draw (O) -> ($12 *(X)$) -> ($12 *(X) + 12 *(Z)$) -> (O);

\node[anchor= north east] at (O) {$(0, 0, 12)$};
\node[anchor= north west] at ($ 12 *(X)$) {$(0, 12, 0)$};
\node[anchor= south] at ($-12 *(Y)$) {$(12,0,0)$};

\node[anchor= north] at ($6 *(X)$) {$(0, 6, 6)$};
\node[anchor= east] at ($-6 *(Y)$) {$(6, 0, 6)$};
\node[anchor= west] at ($6 *(X) - 6 *(Y)$) {$(6, 6, 0)$};

\node at ($2 *(X) - 8 *(Y)$) {$ \bullet$};
\node at ($4 *(X) - 4 *(Y)$) {$ \bullet$};
\node at ($2 *(X) - 2 *(Y)$) {$ \bullet$};
\node at ($8 *(X) - 2 *(Y)$) {$ \bullet$};

\node at ($6 *(X)$) {$\bullet$};
\node at ($-6 *(Y)$) {$\bullet$};
\node at ($6 *(X) - 6 *(Y)$) {$\bullet$};

\node at (O) {$\bullet$};
\node at ($-12 *(Y)$) {$\bullet$};
\node at ($12 *(X)$) {$\bullet$};
\end{scriptsize}
\end{scope}
\end{tikzpicture}
\]
In this picture, the aforementioned $S_3$ symmetry is visible as the dihedral symmetries of this lattice triangle.      
\end{Rem}

\subsubsection{The mutation lattice for $(2,2,8)$}
Now, we describe the mutation lattices in detail. As mentioned before, the $S_3$-symmetry takes cuts of type $(2,2,8)$ to those of type $(2,8,2)$ and $(8,2,2)$. The corresponding $2$-representation infinite algebras are isomorphic, and furthermore the mutation lattices are isomorphic, so it suffices to compute just one of them. We choose $(2,2,8)$ here. 

Before we compute the mutation lattices, we introduce conventions on how we draw the quivers. 
\begin{Rem}
    We work in the case when $n=2$, that means the lattice $L_0$ has rank $2$ and every vertex in the quiver $\hat{Q}$ has $3$ outgoing and $3$ incoming vertices, and the elementary cycles have length $3$. We therefore draw $\hat{Q}$ as an infinite triangular grid. We visualize $Q_G$ as the quotient of $\hat{Q}$ by the $L_1$-action, so we draw below the infinite grid with the vertex $0$ labelled explicitly as such. The $L_1$-equivariance implies the following. The hexagon that we draw gives rise to the torus by identifying opposing sides, and the quiver $Q_G$ is the image of the drawn hexagon under this identification. For the reader's convenience, we draw three red points which are identified, and three blue points which are identified. We follow the convention that the horizontal arrows are type $1$, the arrows pointing down are type $2$, and the arrows pointing up are type $3$.
    \[
\begin{tikzpicture}[-stealth,scale=0.75, every node/.style={scale=0.6}]
\begin{scope}[scale= 0.75, rotate=0,inner sep=1.5mm]

\begin{scriptsize}


\coordinate (O) at (0, 0);
\coordinate (X) at (1,0);
\coordinate (Y) at (cos 240, sin 240);
\coordinate (Z) at (cos 120, sin 120); 


\WLozenge{($(Z)-(X)$)}
\WLozenge{($(Z)$)}
\WLozenge{($-2*(X)$)}
\ULozenge{($-1*(X)$)}
\VLozenge{($(O)$)}
\VLozenge{($1*(X)$)}
\ULozenge{($-1*(X)+1*(Y)$)}
\WLozenge{($1*(Y)$)}
\WLozenge{($1*(X)+1*(Y)$)}
\WLozenge{($2*(Y)$)}
\WLozenge{($(X)+2*(Y)$)}
\ULozenge{(O)}
\ULozenge{(X)}
\ULozenge{(Y)}
\ULozenge{($(X) + (Y)$)}
\ULozenge{($(O) - (Y)$)}
\ULozenge{(Z)}
\ULozenge{($(Z) - (X)$)}
\VLozenge{($-2 *(X) - (Z)$)}
\VLozenge{($-1 *(X)$)}

\node at (O) {$0$};
\node at ($-2*(X)$) { \color{red} $\bullet$};
\node at ($-2*(Z)$) { \color{red} $\bullet$};
\node at ($-2*(Y)$) { \color{red} $\bullet$};

\node at ($2*(X)$) { \color{blue} $\bullet$};
\node at ($2*(Z)$) { \color{blue} $\bullet$};
\node at ($2*(Y)$) { \color{blue} $\bullet$};

\end{scriptsize}
\end{scope}
\end{tikzpicture}
\]
\end{Rem}

\begin{Rem}
    Let $\gamma = (2,2,8)$. The maximal element in the mutation lattice and the corresponding height function are drawn below. 
    \[
\begin{tikzpicture}[-stealth,scale=0.75, every node/.style={scale=0.6}]
\begin{scope}[scale= 0.75, rotate=0,inner sep=1.5mm]

\begin{scriptsize}


\coordinate (O) at (0, 0);
\coordinate (X) at (1,0);
\coordinate (Y) at (cos 240, sin 240);
\coordinate (Z) at (cos 120, sin 120); 


\WLozenge{($(Z)-(X)$)}
\WLozenge{($(Z)$)}
\WLozenge{($-2*(X)$)}
\ULozenge{($-1*(X)$)}
\VLozenge{($(O)$)}
\VLozenge{($1*(X)$)}
\ULozenge{($-1*(X)+1*(Y)$)}
\WLozenge{($1*(Y)$)}
\WLozenge{($1*(X)+1*(Y)$)}
\WLozenge{($2*(Y)$)}
\WLozenge{($(X)+2*(Y)$)}

\node at ($2*(Z)$) { $ \text{-} 1$};
\node at ($(X)+2*(Z)$) { $0$};
\node at ($2*(X)+2*(Z)$) {$1$};

\node at ($(Y)+2*(Z)$) {$0$};
\node at ($(Z)$) {$1$};
\node at ($(X)+(Z)$) {$2$};
\node at ($2*(X)+(Z)$) {$3$};

\node at ($2*(Y)+2*(Z)$) {$1$};
\node at ($(Y)+(Z)$) {$2$};
\node at ($(O)$) {$0$};
\node at ($(X)$) {$1$};
\node at ($2*(X)$) {$2$};

\node at ($2*(Y)+2*(Z)-1*(Z)$) {$3$};
\node at ($(Y)$) {$1$};
\node at ($-1*(Z)$) {$2$};
\node at ($(X)-1*(Z)$) {$3$};

\node at ($2*(Y)$) {$2$};
\node at ($(Y)-1*(Z)$) {$3$};
\node at ($(X)+(Y)-1*(Z)$) {$4$};  

\end{scriptsize}
\end{scope}
\end{tikzpicture} 
\]
We see that the cut quiver has a unique source, located at $0$. Clearly, this means that this is indeed the maximal element, since no further mutation at a source is possible. Furthermore, there is only one sink, located at vertex $2 \alpha_1 + \alpha_3$. Thus, this is the only mutable vertex, and hence there is a unique cover relation involving the maximal element. Continuing down the mutation lattice, we obtain the picture in \Cref{Fig:228 lattice} on the following page.   
\end{Rem}

\begin{figure}
    \centering
\[
\begin{tikzpicture}[-stealth,scale=0.75, every node/.style={scale=0.6}]
\begin{scope}[scale= 0.75, rotate=0,inner sep=1.5mm]

\begin{scriptsize}


\coordinate (O) at (0, 0);
\coordinate (X) at (1,0);
\coordinate (Y) at (cos 240, sin 240);
\coordinate (Z) at (cos 120, sin 120); 


\WLozenge{($(Z)-(X)$)}
\WLozenge{($(Z)$)}
\WLozenge{($-2*(X)$)}
\ULozenge{($-1*(X)$)}
\VLozenge{($(O)$)}
\VLozenge{($1*(X)$)}
\ULozenge{($-1*(X)+1*(Y)$)}
\WLozenge{($1*(Y)$)}
\WLozenge{($1*(X)+1*(Y)$)}
\WLozenge{($2*(Y)$)}
\WLozenge{($(X)+2*(Y)$)}

\node at ($2*(Z)$) { $ \text{-} 1$};
\node at ($(X)+2*(Z)$) { $0$};
\node at ($2*(X)+2*(Z)$) {$1$};

\node at ($(Y)+2*(Z)$) {$0$};
\node at ($(Z)$) {$1$};
\node at ($(X)+(Z)$) {$2$};
\node at ($2*(X)+(Z)$) {$3$};

\node at ($2*(Y)+2*(Z)$) {$1$};
\node at ($(Y)+(Z)$) {$2$};
\node at ($(O)$) {$0$};
\node at ($(X)$) {$1$};
\node at ($2*(X)$) {$2$};

\node at ($2*(Y)+2*(Z)-1*(Z)$) {$3$};
\node at ($(Y)$) {$1$};
\node at ($-1*(Z)$) {$2$};
\node at ($(X)-1*(Z)$) {$3$};

\node at ($2*(Y)$) {$2$};
\node at ($(Y)-1*(Z)$) {$3$};
\node at ($(X)+(Y)-1*(Z)$) {$4$};

\WLozenge{($(Z)-(X)+3*(X)+6*(Y)$)}
\WLozenge{($(Z)+3*(X)+6*(Y)$)}
\WLozenge{($-2*(X)+3*(X)+6*(Y)$)}
\ULozenge{($-1*(X)+3*(X)+6*(Y)$)}
\ULozenge{($-1*(Y)+3*(X)+6*(Y)$)}
\VLozenge{($3*(X)+6*(Y)$)}
\VLozenge{($-1*(X)+1*(Y)+3*(X)+6*(Y)$)}
\WLozenge{($1*(Y)+3*(X)+6*(Y)$)}
\WLozenge{($1*(X)+1*(Y)+3*(X)+6*(Y)$)}
\WLozenge{($2*(Y)+3*(X)+6*(Y)$)}
\WLozenge{($(X)+2*(Y)+3*(X)+6*(Y)$)}

\node at ($2*(Z)+3*(X)+6*(Y)$) { $ \text{-} 1$};
\node at ($(X)+2*(Z)+3*(X)+6*(Y)$) { $0$};
\node at ($2*(X)+2*(Z)+3*(X)+6*(Y)$) {$1$};

\node at ($(Y)+2*(Z)+3*(X)+6*(Y)$) {$0$};
\node at ($(Z)+3*(X)+6*(Y)$) {$1$};
\node at ($(X)+(Z)+3*(X)+6*(Y)$) {$2$};
\node at ($2*(X)+(Z)+3*(X)+6*(Y)$) {$0$};

\node at ($2*(Y)+2*(Z)+3*(X)+6*(Y)$) {$1$};
\node at ($(Y)+(Z)+3*(X)+6*(Y)$) {$2$};
\node at ($(O)+3*(X)+6*(Y)$) {$0$};
\node at ($(X)+3*(X)+6*(Y)$) {$1$};
\node at ($2*(X)+3*(X)+6*(Y)$) {$2$};

\node at ($2*(Y)+2*(Z)-1*(Z)+3*(X)+6*(Y)$) {$0$};
\node at ($(Y)+3*(X)+6*(Y)$) {$1$};
\node at ($-1*(Z)+3*(X)+6*(Y)$) {$2$};
\node at ($(X)-1*(Z)+3*(X)+6*(Y)$) {$3$};

\node at ($2*(Y)+3*(X)+6*(Y)$) {$2$};
\node at ($(Y)-1*(Z)+3*(X)+6*(Y)$) {$3$};
\node at ($(X)+(Y)-1*(Z)+3*(X)+6*(Y)$) {$4$};

\WLozenge{($(Z)+5*(X)+12*(Y)$)}
\WLozenge{($6*(X)+12*(Y)$)}
\ULozenge{($-1*(X)+5*(X)+12*(Y)$)}
\ULozenge{($-1*(Y)+5*(X)+12*(Y)$)}
\VLozenge{($5*(X)+12*(Y)$)}
\VLozenge{($-1*(X)+1*(Y)+5*(X)+12*(Y)+4*(X)+2*(Z)$)}
\WLozenge{($1*(Y)+5*(X)+12*(Y)$)}
\WLozenge{($1*(X)+1*(Y)+5*(X)+12*(Y)$)}
\WLozenge{($2*(X)+1*(Y)+5*(X)+12*(Y)$)}
\WLozenge{($1*(X)+2*(Y)+5*(X)+12*(Y)$)}
\WLozenge{($2*(X)+2*(Y)+5*(X)+12*(Y)$)}

\node at ($2*(Z)+6*(X)+12*(Y)$) { $ \text{-} 1$};
\node at ($(X)+2*(Z)+6*(X)+12*(Y)$) { $0$};
\node at ($2*(X)+2*(Z)+6*(X)+12*(Y)$) {$1$};

\node at ($(Y)+2*(Z)+6*(X)+12*(Y)$) {$0$};
\node at ($(Z)+6*(X)+12*(Y)$) {$1$};
\node at ($(X)+(Z)+6*(X)+12*(Y)$) {$\text{-} 1$};
\node at ($2*(X)+(Z)+6*(X)+12*(Y)$) {$0$};

\node at ($2*(Y)+2*(Z)+6*(X)+12*(Y)$) {$1$};
\node at ($(Y)+(Z)+6*(X)+12*(Y)$) {$\text{-} 1$};
\node at ($(O)+6*(X)+12*(Y)$) {$0$};
\node at ($(X)+6*(X)+12*(Y)$) {$1$};
\node at ($2*(X)+6*(X)+12*(Y)$) {$2$};

\node at ($2*(Y)+2*(Z)-1*(Z)+6*(X)+12*(Y)$) {$0$};
\node at ($(Y)+6*(X)+12*(Y)$) {$1$};
\node at ($-1*(Z)+6*(X)+12*(Y)$) {$2$};
\node at ($(X)-1*(Z)+6*(X)+12*(Y)$) {$3$};

\node at ($2*(Y)+6*(X)+12*(Y)$) {$2$};
\node at ($(Y)-1*(Z)+6*(X)+12*(Y)$) {$3$};
\node at ($(X)+(Y)-1*(Z)+6*(X)+12*(Y)$) {$4$};

\WLozenge{($9*(X)+18*(Y)$)}
\ULozenge{($-1*(Y)+7*(X)+18*(Y)$)}
\VLozenge{($-1*(X)+1*(Y)+7*(X)+18*(Y)+4*(X)+2*(Z)$)}
\WLozenge{($1*(Y)+8*(X)+18*(Y)$)}
\WLozenge{($1*(X)+1*(Y)+8*(X)+18*(Y)$)}
\WLozenge{($2*(X)+1*(Y)+8*(X)+18*(Y)$)}
\WLozenge{($1*(X)+2*(Y)+8*(X)+18*(Y)$)}
\WLozenge{($1*(Y)+8*(X)+18*(Y)$)}
\draw ($1*(Y)+10*(X)+18*(Y)-0.8*(Z)$)  -- ($1*(Y)+10*(X)+18*(Y)-0.2*(Z)$);
\draw ($7*(X)+18*(Y)+0.2*(X)$)  -- ($7*(X)+18*(Y)+0.8*(X)$);
\draw ($9*(X)+16*(Y)+0.2*(Y)$)  -- ($9*(X)+16*(Y)+0.8*(Y)$);

\node at ($2*(Z)+9*(X)+18*(Y)$) { $ \text{-} 1$};
\node at ($(X)+2*(Z)+9*(X)+18*(Y)$) { $0$};
\node at ($2*(X)+2*(Z)+9*(X)+18*(Y)$) {$\text{-} 2$};

\node at ($(Y)+2*(Z)+9*(X)+18*(Y)$) {$0$};
\node at ($(Z)+9*(X)+18*(Y)$) {$\text{-} 2$};
\node at ($(X)+(Z)+9*(X)+18*(Y)$) {$\text{-} 1$};
\node at ($2*(X)+(Z)+9*(X)+18*(Y)$) {$0$};

\node at ($2*(Y)+2*(Z)+9*(X)+18*(Y)$) {$\text{-} 2$};
\node at ($(Y)+(Z)+9*(X)+18*(Y)$) {$\text{-} 1$};
\node at ($(O)+9*(X)+18*(Y)$) {$0$};
\node at ($(X)+9*(X)+18*(Y)$) {$1$};
\node at ($2*(X)+9*(X)+18*(Y)$) {$2$};

\node at ($2*(Y)+2*(Z)-1*(Z)+9*(X)+18*(Y)$) {$0$};
\node at ($(Y)+9*(X)+18*(Y)$) {$1$};
\node at ($-1*(Z)+9*(X)+18*(Y)$) {$2$};
\node at ($(X)-1*(Z)+9*(X)+18*(Y)$) {$3$};

\node at ($2*(Y)+9*(X)+18*(Y)$) {$2$};
\node at ($(Y)-1*(Z)+9*(X)+18*(Y)$) {$3$};
\node at ($(X)+(Y)-1*(Z)+9*(X)+18*(Y)$) {$1$};

\WLozenge{($12*(X)+24*(Y)$)}
\WLozenge{($11*(X)+24*(Y)$)}
\WLozenge{($11*(X)+23*(Y)$)}
\WLozenge{($10*(X)+24*(Y)$)}
\WLozenge{($1*(Y)+11*(X)+24*(Y)$)}
\WLozenge{($1*(X)+1*(Y)+11*(X)+24*(Y)$)}
\ULozenge{($2*(X)+1*(Y)+11*(X)+24*(Y)$)}
\VLozenge{($2*(X)+2*(Y)+11*(X)+24*(Y)$)}
\draw ($2.2*(X)+11*(X)+24*(Y)$)  -- ($2.8*(X)+11*(X)+24*(Y)$);
\draw ($2*(X)+10*(X)+24*(Y)+1.2*(Z)$)  -- ($2*(X)+10*(X)+24*(Y)+1.8*(Z)$);
\draw ($2*(X)+10*(X)+25.2*(Y)$)  -- ($2*(X)+10*(X)+25.8*(Y)$);

\node at ($2*(Z)+12*(X)+24*(Y)$) { $ \text{-} 1$};
\node at ($(X)+2*(Z)+12*(X)+24*(Y)$) { $ \text{-} 3$};
\node at ($2*(X)+2*(Z)+12*(X)+24*(Y)$) {$\text{-} 2$};

\node at ($(Y)+2*(Z)+12*(X)+24*(Y)$) {$ \text{-} 3$};
\node at ($(Z)+12*(X)+24*(Y)$) {$\text{-} 2$};
\node at ($(X)+(Z)+12*(X)+24*(Y)$) {$\text{-} 1$};
\node at ($2*(X)+(Z)+12*(X)+24*(Y)$) {$0$};

\node at ($2*(Y)+2*(Z)+12*(X)+24*(Y)$) {$\text{-} 2$};
\node at ($(Y)+(Z)+12*(X)+24*(Y)$) {$\text{-} 1$};
\node at ($(O)+12*(X)+24*(Y)$) {$0$};
\node at ($(X)+12*(X)+24*(Y)$) {$1$};
\node at ($2*(X)+12*(X)+24*(Y)$) {$2$};

\node at ($2*(Y)+2*(Z)-1*(Z)+12*(X)+24*(Y)$) {$0$};
\node at ($(Y)+12*(X)+24*(Y)$) {$1$};
\node at ($-1*(Z)+12*(X)+24*(Y)$) {$2$};
\node at ($(X)-1*(Z)+12*(X)+24*(Y)$) {$0$};

\node at ($2*(Y)+12*(X)+24*(Y)$) {$2$};
\node at ($(Y)-1*(Z)+12*(X)+24*(Y)$) {$0$};
\node at ($(X)+(Y)-1*(Z)+12*(X)+24*(Y)$) {$1$};

\WLozenge{($(Y)+2*(Z)+15*(X)+30*(Y)$)}
\WLozenge{($(Y)+2*(Z)+16*(X)+30*(Y)$)}
\WLozenge{($2*(Y)+2*(Z)+15*(X)+30*(Y)$)}
\WLozenge{($2*(Y)+2*(Z)+16*(X)+30*(Y)$)}
\WLozenge{($2*(Y)+2*(Z)+17*(X)+30*(Y)$)}
\WLozenge{($2*(Y)+2*(Z)+16*(X)+31*(Y)$)}
\WLozenge{($2*(Y)+2*(Z)+18*(X)+32*(Y)$)}
\ULozenge{($2*(Y)+3*(Z)+18*(X)+32*(Y)$)}
\VLozenge{($2*(Y)+3*(Z)+19*(X)+32*(Y)$)}
\ULozenge{($(Y)+3*(Z)+19*(X)+32*(Y)$)}
\VLozenge{($3*(Y)+3*(Z)+18*(X)+32*(Y)$)}

\node at ($2*(Z)+15*(X)+30*(Y)$) { $ \text{-} 4$};
\node at ($(X)+2*(Z)+15*(X)+30*(Y)$) { $ \text{-} 3$};
\node at ($2*(X)+2*(Z)+15*(X)+30*(Y)$) {$\text{-} 2$};

\node at ($(Y)+2*(Z)+15*(X)+30*(Y)$) {$ \text{-} 3$};
\node at ($(Z)+15*(X)+30*(Y)$) {$\text{-} 2$};
\node at ($(X)+(Z)+15*(X)+30*(Y)$) {$\text{-} 1$};
\node at ($2*(X)+(Z)+15*(X)+30*(Y)$) {$0$};

\node at ($2*(Y)+2*(Z)+15*(X)+30*(Y)$) {$\text{-} 2$};
\node at ($(Y)+(Z)+15*(X)+30*(Y)$) {$\text{-} 1$};
\node at ($(O)+15*(X)+30*(Y)$) {$0$};
\node at ($(X)+15*(X)+30*(Y)$) {$1$};
\node at ($2*(X)+15*(X)+30*(Y)$) {$\text{-} 1$};

\node at ($2*(Y)+2*(Z)-1*(Z)+15*(X)+30*(Y)$) {$0$};
\node at ($(Y)+15*(X)+30*(Y)$) {$1$};
\node at ($-1*(Z)+15*(X)+30*(Y)$) {$\text{-} 1$};
\node at ($(X)-1*(Z)+15*(X)+30*(Y)$) {$0$};

\node at ($2*(Y)+15*(X)+30*(Y)$) {$\text{-} 1$};
\node at ($(Y)-1*(Z)+15*(X)+30*(Y)$) {$0$};
\node at ($(X)+(Y)-1*(Z)+15*(X)+30*(Y)$) {$1$};

\WLozenge{($(Y)+2*(Z)+18*(X)+36*(Y)$)}
\WLozenge{($(Y)+2*(Z)+19*(X)+36*(Y)$)}
\WLozenge{($(Y)+2*(Z)+18*(X)+37*(Y)$)}
\WLozenge{($(Y)+2*(Z)+19*(X)+37*(Y)$)}
\WLozenge{($-1*(Z)+18*(X)+36*(Y)$)}
\WLozenge{($-1*(Z)+17*(X)+37*(Y)$)}
\WLozenge{($-1*(Z)+18*(X)+37*(Y)$)}
\ULozenge{($17*(X)+37*(Y)$)}
\VLozenge{($18*(X)+37*(Y)$)}
\ULozenge{($18*(X)+36*(Y)$)}
\VLozenge{($19*(X)+36*(Y)$)}

\node at ($2*(Z)+18*(X)+36*(Y)$) { $ \text{-} 4$};
\node at ($(X)+2*(Z)+18*(X)+36*(Y)$) { $ \text{-} 3$};
\node at ($2*(X)+2*(Z)+18*(X)+36*(Y)$) {$\text{-} 2$};

\node at ($(Y)+2*(Z)+18*(X)+36*(Y)$) {$ \text{-} 3$};
\node at ($(Z)+18*(X)+36*(Y)$) {$\text{-} 2$};
\node at ($(X)+(Z)+18*(X)+36*(Y)$) {$\text{-} 1$};
\node at ($2*(X)+(Z)+18*(X)+36*(Y)$) {$0$};

\node at ($2*(Y)+2*(Z)+18*(X)+36*(Y)$) {$\text{-} 2$};
\node at ($(Y)+(Z)+18*(X)+36*(Y)$) {$\text{-} 1$};
\node at ($(O)+18*(X)+36*(Y)$) {$0$};
\node at ($(X)+18*(X)+36*(Y)$) {$\text{-} 2$};
\node at ($2*(X)+18*(X)+36*(Y)$) {$\text{-} 1$};

\node at ($2*(Y)+2*(Z)-1*(Z)+18*(X)+36*(Y)$) {$0$};
\node at ($(Y)+18*(X)+36*(Y)$) {$\text{-} 2$};
\node at ($-1*(Z)+18*(X)+36*(Y)$) {$\text{-} 1$};
\node at ($(X)-1*(Z)+18*(X)+36*(Y)$) {$0$};

\node at ($2*(Y)+18*(X)+36*(Y)$) {$\text{-} 1$};
\node at ($(Y)-1*(Z)+18*(X)+36*(Y)$) {$0$};
\node at ($(X)+(Y)-1*(Z)+18*(X)+36*(Y)$) {$1$};

\WLozenge{($(Y)+2*(Z)+21*(X)+42*(Y)$)}
\WLozenge{($(Y)+2*(Z)+22*(X)+42*(Y)$)}
\WLozenge{($(Y)+2*(Z)+21*(X)+43*(Y)$)}
\WLozenge{($(Y)+2*(Z)+22*(X)+43*(Y)$)}
\WLozenge{($(Y)+2*(Z)+24*(X)+44*(Y)$)}
\WLozenge{($(Y)+2*(Z)+24*(X)+45*(Y)$)}
\WLozenge{($(Y)+2*(Z)+23*(X)+45*(Y)$)}
\VLozenge{($(Y)+3*(Z)+23*(X)+45*(Y)$)}
\VLozenge{($(Y)+3*(Z)+24*(X)+45*(Y)$)}
\ULozenge{($3*(Z)+24*(X)+45*(Y)$)}
\ULozenge{($3*(Z)+24*(X)+44*(Y)$)}

\node at ($2*(Z)+21*(X)+42*(Y)$) { $ \text{-} 4$};
\node at ($(X)+2*(Z)+21*(X)+42*(Y)$) { $ \text{-} 3$};
\node at ($2*(X)+2*(Z)+21*(X)+42*(Y)$) {$\text{-} 2$};

\node at ($(Y)+2*(Z)+21*(X)+42*(Y)$) {$ \text{-} 3$};
\node at ($(Z)+21*(X)+42*(Y)$) {$\text{-} 2$};
\node at ($(X)+(Z)+21*(X)+42*(Y)$) {$\text{-} 1$};
\node at ($2*(X)+(Z)+21*(X)+42*(Y)$) {$\text{-} 3$};

\node at ($2*(Y)+2*(Z)+21*(X)+42*(Y)$) {$\text{-} 2$};
\node at ($(Y)+(Z)+21*(X)+42*(Y)$) {$\text{-} 1$};
\node at ($(O)+21*(X)+42*(Y)$) {$0$};
\node at ($(X)+21*(X)+42*(Y)$) {$\text{-} 2$};
\node at ($2*(X)+21*(X)+42*(Y)$) {$\text{-} 1$};

\node at ($2*(Y)+2*(Z)-1*(Z)+21*(X)+42*(Y)$) {$\text{-} 3$};
\node at ($(Y)+21*(X)+42*(Y)$) {$\text{-} 2$};
\node at ($-1*(Z)+21*(X)+42*(Y)$) {$\text{-} 1$};
\node at ($(X)-1*(Z)+21*(X)+42*(Y)$) {$0$};

\node at ($2*(Y)+21*(X)+42*(Y)$) {$\text{-} 1$};
\node at ($(Y)-1*(Z)+21*(X)+42*(Y)$) {$0$};
\node at ($(X)+(Y)-1*(Z)+21*(X)+42*(Y)$) {$1$};


\WLozenge{($(Y)+2*(Z)+9*(Y)$)}
\WLozenge{($(Y)+2*(Z)+9*(Y)+(X)$)}
\WLozenge{($(Y)+2*(Z)+11*(Y)+(X)$)}
\WLozenge{($(Y)+2*(Z)+11*(Y)+2*(X)$)}
\WLozenge{($(Y)+2*(Z)+11*(Y)+3*(X)$)}
\WLozenge{($(Y)+2*(Z)+12*(Y)+2*(X)$)}
\WLozenge{($(Y)+2*(Z)+12*(Y)+3*(X)$)}
\ULozenge{($2*(Y)+2*(Z)+9*(Y)$)}
\VLozenge{($2*(Y)+2*(Z)+9*(Y)+(X)$)}
\VLozenge{($2*(Y)+2*(Z)+9*(Y)+2*(X)$)}
\ULozenge{($2*(Y)+2*(Z)+8*(Y)+2*(X)$)}

\node at ($2*(Z)+9*(Y)$) { $ \text{-} 1$};
\node at ($(X)+2*(Z)+9*(Y)$) { $0$};
\node at ($2*(X)+2*(Z)+9*(Y)$) {$1$};

\node at ($(Y)+2*(Z)+9*(Y)$) {$0$};
\node at ($(Z)+9*(Y)$) {$1$};
\node at ($(X)+(Z)+9*(Y)$) {$2$};
\node at ($2*(X)+(Z)+9*(Y)$) {$0$};

\node at ($2*(Y)+2*(Z)+9*(Y)$) {$1$};
\node at ($(Y)+(Z)+9*(Y)$) {$\text{-}1$};
\node at ($(O)+9*(Y)$) {$0$};
\node at ($(X)+9*(Y)$) {$1$};
\node at ($2*(X)+9*(Y)$) {$2$};

\node at ($2*(Y)+2*(Z)-1*(Z)+9*(Y)$) {$0$};
\node at ($(Y)+9*(Y)$) {$1$};
\node at ($-1*(Z)+9*(Y)$) {$2$};
\node at ($(X)-1*(Z)+9*(Y)$) {$3$};

\node at ($2*(Y)+9*(Y)$) {$2$};
\node at ($(Y)-1*(Z)+9*(Y)$) {$3$};
\node at ($(X)+(Y)-1*(Z)+9*(Y)$) {$4$};

\WLozenge{($(Y)+(Z)+9*(Y)+3*(X)+6*(Y)$)}
\WLozenge{($(Y)+(Z)+9*(Y)+4*(X)+6*(Y)$)}
\WLozenge{($(Y)+(Z)+9*(Y)+3*(X)+7*(Y)$)}
\WLozenge{($(Y)+(Z)+9*(Y)+4*(X)+7*(Y)$)}
\WLozenge{($(Y)+(Z)+9*(Y)+5*(X)+7*(Y)$)}
\WLozenge{($(Y)+(Z)+9*(Y)+4*(X)+8*(Y)$)}
\WLozenge{($(Y)+(Z)+9*(Y)+5*(X)+8*(Y)$)}
\ULozenge{($(Y)+(Z)+9*(Y)+2*(X)+6*(Y)$)}
\ULozenge{($(Y)+(Z)+9*(Y)+2*(X)+5*(Y)$)}
\VLozenge{($(Y)+(Z)+9*(Y)+3*(X)+5*(Y)$)}
\VLozenge{($(Y)+(Z)+9*(Y)+4*(X)+5*(Y)$)}

\node at ($2*(Z)+9*(Y)+3*(X)+6*(Y)$) { $ \text{-} 1$};
\node at ($(X)+2*(Z)+9*(Y)+3*(X)+6*(Y)$) { $0$};
\node at ($2*(X)+2*(Z)+9*(Y)+3*(X)+6*(Y)$) {$1$};

\node at ($(Y)+2*(Z)+9*(Y)+3*(X)+6*(Y)$) {$0$};
\node at ($(Z)+9*(Y)+3*(X)+6*(Y)$) {$\text{-}2$};
\node at ($(X)+(Z)+9*(Y)+3*(X)+6*(Y)$) {$\text{-}1$};
\node at ($2*(X)+(Z)+9*(Y)+3*(X)+6*(Y)$) {$0$};

\node at ($2*(Y)+2*(Z)+9*(Y)+3*(X)+6*(Y)$) {$1$};
\node at ($(Y)+(Z)+9*(Y)+3*(X)+6*(Y)$) {$\text{-}1$};
\node at ($(O)+9*(Y)+3*(X)+6*(Y)$) {$0$};
\node at ($(X)+9*(Y)+3*(X)+6*(Y)$) {$1$};
\node at ($2*(X)+9*(Y)+3*(X)+6*(Y)$) {$2$};

\node at ($2*(Y)+2*(Z)-1*(Z)+9*(Y)+3*(X)+6*(Y)$) {$0$};
\node at ($(Y)+9*(Y)+3*(X)+6*(Y)$) {$1$};
\node at ($-1*(Z)+9*(Y)+3*(X)+6*(Y)$) {$2$};
\node at ($(X)-1*(Z)+9*(Y)+3*(X)+6*(Y)$) {$3$};

\node at ($2*(Y)+9*(Y)+3*(X)+6*(Y)$) {$2$};
\node at ($(Y)-1*(Z)+9*(Y)+3*(X)+6*(Y)$) {$3$};
\node at ($(X)+(Y)-1*(Z)+9*(Y)+3*(X)+6*(Y)$) {$4$};

\WLozenge{($(Z)+9*(Y)+6*(X)+12*(Y)$)}
\WLozenge{($(Z)+10*(Y)+6*(X)+12*(Y)$)}
\WLozenge{($(Z)+10*(Y)+7*(X)+12*(Y)$)}
\WLozenge{($(Z)+10*(Y)+6*(X)+13*(Y)$)}
\WLozenge{($(Z)+10*(Y)+7*(X)+13*(Y)$)}
\WLozenge{($(Z)+10*(Y)+8*(X)+13*(Y)$)}
\ULozenge{($(Z)+9*(Y)+5*(X)+12*(Y)$)}
\VLozenge{($(Z)+11*(Y)+8*(X)+13*(Y)$)}
\draw ($2*(Z)+11*(Y)+8*(X)+13*(Y)+0.2*(Y)$)--($2*(Z)+11*(Y)+8*(X)+13*(Y)+0.8*(Y)$);
\draw ($3*(Z)+11*(Y)+7.2*(X)+13*(Y)$)--($3*(Z)+11*(Y)+7.8*(X)+13*(Y)$);

\node at ($2*(Z)+9*(Y)+6*(X)+12*(Y)$) { $ \text{-} 1$};
\node at ($(X)+2*(Z)+9*(Y)+6*(X)+12*(Y)$) { $\text{-}3$};
\node at ($2*(X)+2*(Z)+9*(Y)+6*(X)+12*(Y)$) {$\text{-}2$};

\node at ($(Y)+2*(Z)+9*(Y)+6*(X)+12*(Y)$) {$0$};
\node at ($(Z)+9*(Y)+6*(X)+12*(Y)$) {$\text{-}2$};
\node at ($(X)+(Z)+9*(Y)+6*(X)+12*(Y)$) {$\text{-}1$};
\node at ($2*(X)+(Z)+9*(Y)+6*(X)+12*(Y)$) {$0$};

\node at ($2*(Y)+2*(Z)+9*(Y)+6*(X)+12*(Y)$) {$\text{-}2$};
\node at ($(Y)+(Z)+9*(Y)+6*(X)+12*(Y)$) {$\text{-}1$};
\node at ($(O)+9*(Y)+6*(X)+12*(Y)$) {$0$};
\node at ($(X)+9*(Y)+6*(X)+12*(Y)$) {$1$};
\node at ($2*(X)+9*(Y)+6*(X)+12*(Y)$) {$2$};

\node at ($2*(Y)+2*(Z)-1*(Z)+9*(Y)+6*(X)+12*(Y)$) {$0$};
\node at ($(Y)+9*(Y)+6*(X)+12*(Y)$) {$1$};
\node at ($-1*(Z)+9*(Y)+6*(X)+12*(Y)$) {$2$};
\node at ($(X)-1*(Z)+9*(Y)+6*(X)+12*(Y)$) {$3$};

\node at ($2*(Y)+9*(Y)+6*(X)+12*(Y)$) {$2$};
\node at ($(Y)-1*(Z)+9*(Y)+6*(X)+12*(Y)$) {$0$};
\node at ($(X)+(Y)-1*(Z)+9*(Y)+6*(X)+12*(Y)$) {$1$};

\WLozenge{($(Z)+9*(Y)+9*(X)+18*(Y)$)}
\WLozenge{($(Z)+9*(Y)+8*(X)+19*(Y)$)}
\WLozenge{($(Z)+9*(Y)+9*(X)+19*(Y)$)}
\WLozenge{($(Z)+9*(Y)+10*(X)+19*(Y)$)}
\WLozenge{($(Z)+9*(Y)+9*(X)+20*(Y)$)}
\WLozenge{($(Z)+9*(Y)+11*(X)+21*(Y)$)}
\ULozenge{($2*(Z)+9*(Y)+11*(X)+21*(Y)$)}
\VLozenge{($2*(Z)+9*(Y)+12*(X)+21*(Y)$)}
\draw ($2*(Z)+8*(Y)+12*(X)+21*(Y)+0.2*(X)$)--($2*(Z)+8*(Y)+12*(X)+21*(Y)+0.8*(X)$);
\draw ($2*(Z)+9*(Y)+11*(X)+21*(Y)+0.2*(Y)$)--($2*(Z)+9*(Y)+11*(X)+21*(Y)+0.8*(Y)$);
\draw ($(Z)+9*(Y)+9*(X)+18*(Y)+0.2*(Z)$)--($(Z)+9*(Y)+9*(X)+18*(Y)+0.8*(Z)$);

\node at ($2*(Z)+9*(Y)+9*(X)+18*(Y)$) { $ \text{-} 1$};
\node at ($(X)+2*(Z)+9*(Y)+9*(X)+18*(Y)$) { $\text{-}3$};
\node at ($2*(X)+2*(Z)+9*(Y)+9*(X)+18*(Y)$) {$\text{-}2$};

\node at ($(Y)+2*(Z)+9*(Y)+9*(X)+18*(Y)$) {$\text{-}3$};
\node at ($(Z)+9*(Y)+9*(X)+18*(Y)$) {$\text{-}2$};
\node at ($(X)+(Z)+9*(Y)+9*(X)+18*(Y)$) {$\text{-}1$};
\node at ($2*(X)+(Z)+9*(Y)+9*(X)+18*(Y)$) {$0$};

\node at ($2*(Y)+2*(Z)+9*(Y)+9*(X)+18*(Y)$) {$\text{-}2$};
\node at ($(Y)+(Z)+9*(Y)+9*(X)+18*(Y)$) {$\text{-}1$};
\node at ($(O)+9*(Y)+9*(X)+18*(Y)$) {$0$};
\node at ($(X)+9*(Y)+9*(X)+18*(Y)$) {$1$};
\node at ($2*(X)+9*(Y)+9*(X)+18*(Y)$) {$2$};

\node at ($2*(Y)+2*(Z)-1*(Z)+9*(Y)+9*(X)+18*(Y)$) {$0$};
\node at ($(Y)+9*(Y)+9*(X)+18*(Y)$) {$1$};
\node at ($-1*(Z)+9*(Y)+9*(X)+18*(Y)$) {$\text{-}1$};
\node at ($(X)-1*(Z)+9*(Y)+9*(X)+18*(Y)$) {$0$};

\node at ($2*(Y)+9*(Y)+9*(X)+18*(Y)$) {$2$};
\node at ($(Y)-1*(Z)+9*(Y)+9*(X)+18*(Y)$) {$0$};
\node at ($(X)+(Y)-1*(Z)+9*(Y)+9*(X)+18*(Y)$) {$1$};

\WLozenge{($(Y)+2*(Z)+9*(Y)+12*(X)+24*(Y)$)}
\WLozenge{($(Y)+2*(Z)+9*(Y)+13*(X)+24*(Y)$)}
\WLozenge{($(Y)+2*(Z)+10*(Y)+12*(X)+24*(Y)$)}
\WLozenge{($(Y)+2*(Z)+10*(Y)+13*(X)+24*(Y)$)}
\WLozenge{($(Y)+2*(Z)+11*(Y)+13*(X)+24*(Y)$)}
\WLozenge{($(Y)+2*(Z)+11*(Y)+15*(X)+24*(Y)$)}
\WLozenge{($(Y)+2*(Z)+12*(Y)+15*(X)+24*(Y)$)}
\VLozenge{($(Y)+2*(Z)+12*(Y)+14*(X)+24*(Y)$)}
\ULozenge{($(Y)+2*(Z)+11*(Y)+14*(X)+24*(Y)$)}
\ULozenge{($(Y)+2*(Z)+10*(Y)+14*(X)+24*(Y)$)}
\VLozenge{($(Y)+2*(Z)+10*(Y)+15*(X)+24*(Y)$)}

\node at ($2*(Z)+9*(Y)+12*(X)+24*(Y)$) { $ \text{-} 4$};
\node at ($(X)+2*(Z)+9*(Y)+12*(X)+24*(Y)$) { $\text{-}3$};
\node at ($2*(X)+2*(Z)+9*(Y)+12*(X)+24*(Y)$) {$\text{-}2$};

\node at ($(Y)+2*(Z)+9*(Y)+12*(X)+24*(Y)$) {$\text{-}3$};
\node at ($(Z)+9*(Y)+12*(X)+24*(Y)$) {$\text{-}2$};
\node at ($(X)+(Z)+9*(Y)+12*(X)+24*(Y)$) {$\text{-}1$};
\node at ($2*(X)+(Z)+9*(Y)+12*(X)+24*(Y)$) {$0$};

\node at ($2*(Y)+2*(Z)+9*(Y)+12*(X)+24*(Y)$) {$\text{-}2$};
\node at ($(Y)+(Z)+9*(Y)+12*(X)+24*(Y)$) {$\text{-}1$};
\node at ($(O)+9*(Y)+12*(X)+24*(Y)$) {$0$};
\node at ($(X)+9*(Y)+12*(X)+24*(Y)$) {$\text{-}2$};
\node at ($2*(X)+9*(Y)+12*(X)+24*(Y)$) {$\text{-}1$};

\node at ($2*(Y)+2*(Z)-1*(Z)+9*(Y)+12*(X)+24*(Y)$) {$0$};
\node at ($(Y)+9*(Y)+12*(X)+24*(Y)$) {$1$};
\node at ($-1*(Z)+9*(Y)+12*(X)+24*(Y)$) {$\text{-}1$};
\node at ($(X)-1*(Z)+9*(Y)+12*(X)+24*(Y)$) {$0$};

\node at ($2*(Y)+9*(Y)+12*(X)+24*(Y)$) {$\text{-}1$};
\node at ($(Y)-1*(Z)+9*(Y)+12*(X)+24*(Y)$) {$0$};
\node at ($(X)+(Y)-1*(Z)+9*(Y)+12*(X)+24*(Y)$) {$1$};


\WLozenge{($(Y)+2*(Z)+9*(Y)+9*(X)$)}
\WLozenge{($(Y)+2*(Z)+10*(Y)+9*(X)$)}
\WLozenge{($(Y)+2*(Z)+10*(Y)+11*(X)$)}
\WLozenge{($(Y)+2*(Z)+11*(Y)+11*(X)$)}
\WLozenge{($(Y)+2*(Z)+11*(Y)+12*(X)$)}
\WLozenge{($(Y)+2*(Z)+12*(Y)+11*(X)$)}
\WLozenge{($(Y)+2*(Z)+12*(Y)+12*(X)$)}
\VLozenge{($2*(Y)+2*(Z)-1*(Z)+9*(Y)+9*(X)$)}
\ULozenge{($2*(Y)+2*(Z)-1*(Z)+8*(Y)+9*(X)$)}
\ULozenge{($2*(Y)+2*(Z)-1*(Z)+7*(Y)+9*(X)$)}
\VLozenge{($2*(Y)+2*(Z)-1*(Z)+7*(Y)+10*(X)$)}

\node at ($2*(Z)+9*(Y)+9*(X)$) { $ \text{-} 1$};
\node at ($(X)+2*(Z)+9*(Y)+9*(X)$) { $0$};
\node at ($2*(X)+2*(Z)+9*(Y)+9*(X)$) {$1$};

\node at ($(Y)+2*(Z)+9*(Y)+9*(X)$) {$0$};
\node at ($(Z)+9*(Y)+9*(X)$) {$1$};
\node at ($(X)+(Z)+9*(Y)+9*(X)$) {$\text{-} 1$};
\node at ($2*(X)+(Z)+9*(Y)+9*(X)$) {$0$};

\node at ($2*(Y)+2*(Z)+9*(Y)+9*(X)$) {$1$};
\node at ($(Y)+(Z)+9*(Y)+9*(X)$) {$2$};
\node at ($(O)+9*(Y)+9*(X)$) {$0$};
\node at ($(X)+9*(Y)+9*(X)$) {$1$};
\node at ($2*(X)+9*(Y)+9*(X)$) {$2$};

\node at ($2*(Y)+2*(Z)-1*(Z)+9*(Y)+9*(X)$) {$0$};
\node at ($(Y)+9*(Y)+9*(X)$) {$1$};
\node at ($-1*(Z)+9*(Y)+9*(X)$) {$2$};
\node at ($(X)-1*(Z)+9*(Y)+9*(X)$) {$3$};

\node at ($2*(Y)+9*(Y)+9*(X)$) {$2$};
\node at ($(Y)-1*(Z)+9*(Y)+9*(X)$) {$3$};
\node at ($(X)+(Y)-1*(Z)+9*(Y)+9*(X)$) {$4$};

\WLozenge{($(Y)+2*(Z)+9*(Y)+9*(X)+3*(X)+6*(Y)$)}
\WLozenge{($(O)+9*(Y)+9*(X)+3*(X)+6*(Y)$)}
\WLozenge{($(O)+9*(Y)+8*(X)+3*(X)+7*(Y)$)}
\WLozenge{($(O)+9*(Y)+8*(X)+4*(X)+7*(Y)$)}
\WLozenge{($(O)+9*(Y)+8*(X)+5*(X)+7*(Y)$)}
\WLozenge{($(O)+9*(Y)+8*(X)+4*(X)+8*(Y)$)}
\VLozenge{($(Y)+(Z)+9*(Y)+9*(X)+3*(X)+6*(Y)$)}
\ULozenge{($(Y)+(Z)+9*(Y)+9*(X)+3*(X)+5*(Y)$)}
\draw ($2*(X)+2*(Z)+9*(Y)+9*(X)+3*(X)+6.2*(Y)$)  -- ($2*(X)+2*(Z)+9*(Y)+9*(X)+3*(X)+6.8*(Y)$);
\draw ($2*(Y)+2*(Z)+9*(Y)+9*(X)+3.2*(X)+6*(Y)$) -- ($2*(Y)+2*(Z)+9*(Y)+9*(X)+3.8*(X)+6*(Y)$);
\draw ($(X)+(Y)-0.8*(Z)+9*(Y)+9*(X)+3*(X)+6*(Y)$) -- ($(X)+(Y)-0.2*(Z)+9*(Y)+9*(X)+3*(X)+6*(Y)$);

\node at ($2*(Z)+9*(Y)+9*(X)+3*(X)+6*(Y)$) { $ \text{-} 1$};
\node at ($(X)+2*(Z)+9*(Y)+9*(X)+3*(X)+6*(Y)$) { $0$};
\node at ($2*(X)+2*(Z)+9*(Y)+9*(X)+3*(X)+6*(Y)$) {$\text{-}2$};

\node at ($(Y)+2*(Z)+9*(Y)+9*(X)+3*(X)+6*(Y)$) {$0$};
\node at ($(Z)+9*(Y)+9*(X)+3*(X)+6*(Y)$) {$1$};
\node at ($(X)+(Z)+9*(Y)+9*(X)+3*(X)+6*(Y)$) {$\text{-} 1$};
\node at ($2*(X)+(Z)+9*(Y)+9*(X)+3*(X)+6*(Y)$) {$0$};

\node at ($2*(Y)+2*(Z)+9*(Y)+9*(X)+3*(X)+6*(Y)$) {$\text{-}2$};
\node at ($(Y)+(Z)+9*(Y)+9*(X)+3*(X)+6*(Y)$) {$\text{-}1$};
\node at ($(O)+9*(Y)+9*(X)+3*(X)+6*(Y)$) {$0$};
\node at ($(X)+9*(Y)+9*(X)+3*(X)+6*(Y)$) {$1$};
\node at ($2*(X)+9*(Y)+9*(X)+3*(X)+6*(Y)$) {$2$};

\node at ($2*(Y)+2*(Z)-1*(Z)+9*(Y)+9*(X)+3*(X)+6*(Y)$) {$0$};
\node at ($(Y)+9*(Y)+9*(X)+3*(X)+6*(Y)$) {$1$};
\node at ($-1*(Z)+9*(Y)+9*(X)+3*(X)+6*(Y)$) {$2$};
\node at ($(X)-1*(Z)+9*(Y)+9*(X)+3*(X)+6*(Y)$) {$3$};

\node at ($2*(Y)+9*(Y)+9*(X)+3*(X)+6*(Y)$) {$2$};
\node at ($(Y)-1*(Z)+9*(Y)+9*(X)+3*(X)+6*(Y)$) {$3$};
\node at ($(X)+(Y)-1*(Z)+9*(Y)+9*(X)+3*(X)+6*(Y)$) {$1$};

\WLozenge{($2*(Y)+2*(Z)+9*(Y)+9*(X)+6*(X)+12*(Y)$)}
\WLozenge{($2*(Y)+2*(Z)+9*(Y)+9*(X)+7*(X)+12*(Y)$)}
\WLozenge{($2*(Y)+2*(Z)+9*(Y)+9*(X)+8*(X)+12*(Y)$)}
\WLozenge{($2*(Y)+2*(Z)+9*(Y)+9*(X)+7*(X)+13*(Y)$)}
\WLozenge{($2*(Y)+2*(Z)+9*(Y)+9*(X)+8*(X)+13*(Y)$)}
\WLozenge{($2*(Y)+2*(Z)+9*(Y)+9*(X)+8*(X)+14*(Y)$)}
\VLozenge{($2*(Y)+2*(Z)+8*(Y)+9*(X)+7*(X)+12*(Y)$)}
\ULozenge{($2*(Y)+2*(Z)+9*(Y)+9*(X)+9*(X)+13*(Y)$)}
\draw ($2*(Y)+2*(Z)+8*(Y)+9*(X)+9*(X)+13*(Y)+0.2*(X)$)--
($2*(Y)+2*(Z)+8*(Y)+9*(X)+9*(X)+13*(Y)+0.8*(X)$);
\draw ($2*(Y)+2*(Z)+7.2*(Y)+10*(X)+7*(X)+12*(Y)$)--($2*(Y)+2*(Z)+7.8*(Y)+10*(X)+7*(X)+12*(Y)$);

\node at ($2*(Z)+9*(Y)+9*(X)+6*(X)+12*(Y)$) { $ \text{-} 1$};
\node at ($(X)+2*(Z)+9*(Y)+9*(X)+6*(X)+12*(Y)$) { $0$};
\node at ($2*(X)+2*(Z)+9*(Y)+9*(X)+6*(X)+12*(Y)$) {$\text{-}2$};

\node at ($(Y)+2*(Z)+9*(Y)+9*(X)+6*(X)+12*(Y)$) {$ \text{-} 3$};
\node at ($(Z)+9*(Y)+9*(X)+6*(X)+12*(Y)$) {$ \text{-} 2$};
\node at ($(X)+(Z)+9*(Y)+9*(X)+6*(X)+12*(Y)$) {$\text{-} 1$};
\node at ($2*(X)+(Z)+9*(Y)+9*(X)+6*(X)+12*(Y)$) {$0$};

\node at ($2*(Y)+2*(Z)+9*(Y)+9*(X)+6*(X)+12*(Y)$) {$\text{-}2$};
\node at ($(Y)+(Z)+9*(Y)+9*(X)+6*(X)+12*(Y)$) {$\text{-}1$};
\node at ($(O)+9*(Y)+9*(X)+6*(X)+12*(Y)$) {$0$};
\node at ($(X)+9*(Y)+9*(X)+6*(X)+12*(Y)$) {$1$};
\node at ($2*(X)+9*(Y)+9*(X)+6*(X)+12*(Y)$) {$2$};

\node at ($2*(Y)+2*(Z)-1*(Z)+9*(Y)+9*(X)+6*(X)+12*(Y)$) {$0$};
\node at ($(Y)+9*(Y)+9*(X)+6*(X)+12*(Y)$) {$1$};
\node at ($-1*(Z)+9*(Y)+9*(X)+6*(X)+12*(Y)$) {$2$};
\node at ($(X)-1*(Z)+9*(Y)+9*(X)+6*(X)+12*(Y)$) {$0$};

\node at ($2*(Y)+9*(Y)+9*(X)+6*(X)+12*(Y)$) {$2$};
\node at ($(Y)-1*(Z)+9*(Y)+9*(X)+6*(X)+12*(Y)$) {$3$};
\node at ($(X)+(Y)-1*(Z)+9*(Y)+9*(X)+6*(X)+12*(Y)$) {$1$};  

\WLozenge{($(Y)+2*(Z)+9*(Y)+9*(X)+9*(X)+18*(Y)$)}
\WLozenge{($(Y)+2*(Z)+9*(Y)+10*(X)+9*(X)+18*(Y)$)}
\WLozenge{($(Y)+2*(Z)+9*(Y)+9*(X)+9*(X)+19*(Y)$)}
\WLozenge{($(Y)+2*(Z)+9*(Y)+10*(X)+9*(X)+19*(Y)$)}
\WLozenge{($(Y)+2*(Z)+9*(Y)+11*(X)+9*(X)+19*(Y)$)}
\WLozenge{($(Y)+2*(Z)+10*(Y)+10*(X)+9*(X)+19*(Y)$)}
\WLozenge{($(Y)+2*(Z)+10*(Y)+11*(X)+9*(X)+19*(Y)$)}
\VLozenge{($(Y)+2*(Z)+11*(Y)+11*(X)+9*(X)+19*(Y)$)}
\VLozenge{($(Y)+2*(Z)+11*(Y)+12*(X)+9*(X)+19*(Y)$)}
\ULozenge{($(Y)+2*(Z)+10*(Y)+12*(X)+9*(X)+19*(Y)$)}
\ULozenge{($(Y)+2*(Z)+9*(Y)+12*(X)+9*(X)+19*(Y)$)}

\node at ($2*(Z)+9*(Y)+9*(X)+9*(X)+18*(Y)$) { $ \text{-} 4$};
\node at ($(X)+2*(Z)+9*(Y)+9*(X)+9*(X)+18*(Y)$) { $\text{-}3$};
\node at ($2*(X)+2*(Z)+9*(Y)+9*(X)+9*(X)+18*(Y)$) {$\text{-}2$};

\node at ($(Y)+2*(Z)+9*(Y)+9*(X)+9*(X)+18*(Y)$) {$ \text{-} 3$};
\node at ($(Z)+9*(Y)+9*(X)+9*(X)+18*(Y)$) {$ \text{-} 2$};
\node at ($(X)+(Z)+9*(Y)+9*(X)+9*(X)+18*(Y)$) {$\text{-} 1$};
\node at ($2*(X)+(Z)+9*(Y)+9*(X)+9*(X)+18*(Y)$) {$0$};

\node at ($2*(Y)+2*(Z)+9*(Y)+9*(X)+9*(X)+18*(Y)$) {$\text{-}2$};
\node at ($(Y)+(Z)+9*(Y)+9*(X)+9*(X)+18*(Y)$) {$\text{-}1$};
\node at ($(O)+9*(Y)+9*(X)+9*(X)+18*(Y)$) {$0$};
\node at ($(X)+9*(Y)+9*(X)+9*(X)+18*(Y)$) {$1$};
\node at ($2*(X)+9*(Y)+9*(X)+9*(X)+18*(Y)$) {$\text{-}1$};

\node at ($2*(Y)+2*(Z)-1*(Z)+9*(Y)+9*(X)+9*(X)+18*(Y)$) {$0$};
\node at ($(Y)+9*(Y)+9*(X)+9*(X)+18*(Y)$) {$1$};
\node at ($-1*(Z)+9*(Y)+9*(X)+9*(X)+18*(Y)$) {$2$};
\node at ($(X)-1*(Z)+9*(Y)+9*(X)+9*(X)+18*(Y)$) {$0$};

\node at ($2*(Y)+9*(Y)+9*(X)+9*(X)+18*(Y)$) {$\text{-}1$};
\node at ($(Y)-1*(Z)+9*(Y)+9*(X)+9*(X)+18*(Y)$) {$0$};
\node at ($(X)+(Y)-1*(Z)+9*(Y)+9*(X)+9*(X)+18*(Y)$) {$1$};

\WLozenge{($(Y)+2*(Z)+9*(Y)+9*(X)+12*(X)+24*(Y)$)}
\WLozenge{($(Y)+2*(Z)+9*(Y)+10*(X)+12*(X)+24*(Y)$)}
\WLozenge{($(Y)+2*(Z)+10*(Y)+9*(X)+12*(X)+24*(Y)$)}
\WLozenge{($(Y)+2*(Z)+10*(Y)+10*(X)+12*(X)+24*(Y)$)}
\WLozenge{($(Y)+2*(Z)+10*(Y)+11*(X)+12*(X)+24*(Y)$)}
\WLozenge{($(Y)+2*(Z)+12*(Y)+11*(X)+12*(X)+24*(Y)$)}
\WLozenge{($(Y)+2*(Z)+12*(Y)+12*(X)+12*(X)+24*(Y)$)}
\ULozenge{($(Y)+3*(Z)+12*(Y)+11*(X)+12*(X)+24*(Y)$)}
\VLozenge{($(Y)+3*(Z)+12*(Y)+12*(X)+12*(X)+24*(Y)$)}
\VLozenge{($(Y)+3*(Z)+12*(Y)+13*(X)+12*(X)+24*(Y)$)}
\ULozenge{($(Y)+3*(Z)+11*(Y)+13*(X)+12*(X)+24*(Y)$)}

\node at ($2*(Z)+9*(Y)+9*(X)+12*(X)+24*(Y)$) { $ \text{-} 4$};
\node at ($(X)+2*(Z)+9*(Y)+9*(X)+12*(X)+24*(Y)$) { $\text{-}3$};
\node at ($2*(X)+2*(Z)+9*(Y)+9*(X)+12*(X)+24*(Y)$) {$\text{-}2$};

\node at ($(Y)+2*(Z)+9*(Y)+9*(X)+12*(X)+24*(Y)$) {$ \text{-} 3$};
\node at ($(Z)+9*(Y)+9*(X)+12*(X)+24*(Y)$) {$ \text{-} 2$};
\node at ($(X)+(Z)+9*(Y)+9*(X)+12*(X)+24*(Y)$) {$\text{-} 1$};
\node at ($2*(X)+(Z)+9*(Y)+9*(X)+12*(X)+24*(Y)$) {$0$};

\node at ($2*(Y)+2*(Z)+9*(Y)+9*(X)+12*(X)+24*(Y)$) {$\text{-}2$};
\node at ($(Y)+(Z)+9*(Y)+9*(X)+12*(X)+24*(Y)$) {$\text{-}1$};
\node at ($(O)+9*(Y)+9*(X)+12*(X)+24*(Y)$) {$0$};
\node at ($(X)+9*(Y)+9*(X)+12*(X)+24*(Y)$) {$1$};
\node at ($2*(X)+9*(Y)+9*(X)+12*(X)+24*(Y)$) {$\text{-}1$};

\node at ($2*(Y)+2*(Z)-1*(Z)+9*(Y)+9*(X)+12*(X)+24*(Y)$) {$0$};
\node at ($(Y)+9*(Y)+9*(X)+12*(X)+24*(Y)$) {$\text{-}2$};
\node at ($-1*(Z)+9*(Y)+9*(X)+12*(X)+24*(Y)$) {$\text{-}1$};
\node at ($(X)-1*(Z)+9*(Y)+9*(X)+12*(X)+24*(Y)$) {$0$};

\node at ($2*(Y)+9*(Y)+9*(X)+12*(X)+24*(Y)$) {$\text{-}1$};
\node at ($(Y)-1*(Z)+9*(Y)+9*(X)+12*(X)+24*(Y)$) {$0$};
\node at ($(X)+(Y)-1*(Z)+9*(Y)+9*(X)+12*(X)+24*(Y)$) {$1$};


\draw ($1.7*(X)+3.4*(Y)$)--($1.3*(X)+2.6*(Y)$);
\draw ($19.7*(X)+39.4*(Y)$)--($19.3*(X)+38.6*(Y)$);

\draw ($1.3*(X)+7.7*(Y)$)--($1.7*(X)+7.3*(Y)$);
\draw ($19.3*(X)+34.7*(Y)$)--($19.7*(X)+34.3*(Y)$);
\draw ($4.3*(X)+13.7*(Y)$)--($4.7*(X)+13.3*(Y)$);
\draw ($16.3*(X)+28.7*(Y)$)--($16.7*(X)+28.3*(Y)$);
\draw ($7.3*(X)+19.7*(Y)$)--($7.7*(X)+19.3*(Y)$);
\draw ($13.3*(X)+22.7*(Y)$)--($13.7*(X)+22.3*(Y)$);
\draw ($10.3*(X)+25.7*(Y)$)--($10.7*(X)+25.3*(Y)$);
\draw ($10.3*(X)+16.7*(Y)$)--($10.7*(X)+16.3*(Y)$);
\draw ($13.3*(X)+31.7*(Y)$)--($13.7*(X)+31.3*(Y)$);
\draw ($7.3*(X)+10.7*(Y)$)--($7.7*(X)+10.3*(Y)$);

\draw ($6.4*(X)+7.7*(Y)$)--($5.6*(X)+7.3*(Y)$);
\draw ($15.4*(X)+34.7*(Y)$)--($14.6*(X)+34.3*(Y)$);
\draw ($9.4*(X)+13.7*(Y)$)--($8.6*(X)+13.3*(Y)$);
\draw ($12.4*(X)+28.7*(Y)$)--($11.6*(X)+28.3*(Y)$);
\draw ($12.4*(X)+19.7*(Y)$)--($11.6*(X)+19.3*(Y)$);
\draw ($9.4*(X)+22.7*(Y)$)--($8.6*(X)+22.3*(Y)$);
\draw ($15.4*(X)+25.7*(Y)$)--($14.6*(X)+25.3*(Y)$);
\draw ($6.4*(X)+16.7*(Y)$)--($5.6*(X)+16.3*(Y)$);
\draw ($18.4*(X)+31.7*(Y)$)--($17.6*(X)+31.3*(Y)$);
\draw ($3.4*(X)+10.7*(Y)$)--($2.6*(X)+10.3*(Y)$);


\draw [red] ($(O)$) circle (63pt);
\draw [red] ($21*(X)+42*(Y)$) circle (63pt);
\draw [blue] ($3*(X)+6*(Y)$) circle (63pt);
\draw [blue] ($18*(X)+36*(Y)$) circle (63pt);
\draw [magenta] ($6*(X)+12*(Y)$) circle (63pt);
\draw [magenta] ($15*(X)+30*(Y)$) circle (63pt);
\draw [orange] ($9*(X)+18*(Y)$) circle (63pt);
\draw [orange] ($12*(X)+24*(Y)$) circle (63pt);

\draw [yellow] ($9*(Y)$) circle (63pt);
\draw [yellow] ($21*(X)+33*(Y)$) circle (63pt);
\draw [cyan] ($3*(X)+15*(Y)$) circle (63pt);
\draw [cyan] ($18*(X)+27*(Y)$) circle (63pt);
\draw [violet] ($6*(X)+21*(Y)$) circle (63pt);
\draw [violet] ($15*(X)+21*(Y)$) circle (63pt);
\draw [green] ($9*(X)+27*(Y)$) circle (63pt);
\draw [green] ($12*(X)+15*(Y)$) circle (63pt);
\draw [brown]   ($12*(X)+33*(Y)$) circle (63pt);
\draw [brown]   ($9*(X)+9*(Y)$) circle (63pt);




\end{scriptsize}
\end{scope}
\end{tikzpicture}
\]
    \caption{}
    \label{Fig:228 lattice}
\end{figure}

Next, we make some comments on the duality. 

\begin{Rem}
    In addition to the height functions, cut quivers and cover relations, we also draw coloured circles around the cut quivers. Two cut quivers are circled in the same colour if and only if the cuts are dual to one another. Our convention of drawing the quiver in a hexagonal shape has the effect that the duality may be visualised by rotating a cut quiver around $0$ by $180$ degrees, and then reversing the direction of the arrows. We note that no cut in this lattice is self-dual. 
\end{Rem}

Finally, we explain that this lattice contains significantly more information than just the isomorphism classes of certain $2$-representation infinite algebras. 

\begin{Rem}
    The reader may notice that many of the $2$-representation infinite algebras arising from different cuts are isomorphic as algebras. Indeed, consider for example the minimal and maximal elements in the lattice. To see the isomorphism, translate the infinite cut quiver by $2 \alpha_1 + \alpha_3$, and then quotient again by $L_1$ to obtain the cut quiver for the minimal element. The resulting algebras are isomorphic. 
    Furthermore, we note that all cut quivers have either $1$ sink and $1$ source, or $2$ sinks and $2$ sources. A similar translation argument shows that all the algebras resulting from a cut quiver with $1$ sink and $1$ source are isomorphic, as are those arising from a cut quiver with $2$ sinks and $2$ sources. We will later see that the number of sinks and sources does not determine isomorphism type. However, it is a common feature of the mutation lattice that it contains many cuts which give rise to isomorphic algebras, which nonetheless embed into $R \ast G$ as different algebras. 
\end{Rem}

We also interpret the above observation in terms of $n$-APR tilting. 

\begin{Rem}
    As noted above, different cuts in the mutation lattice can give rise to isomorphic $n$-representation infinite algebras, which are different when viewed as subalgebras of their common higher preprojective algebra. From the perspective of $n$-APR tilting, this is to be expected. The mutation sequence from one cut to the other gives rise to two different iterated $n$-APR tilting modules. However, it is not surprising that different iterated tilting modules have the same endomorphism ring, meaning that the cut algebras are isomorphic. 
\end{Rem}

\subsubsection{The mutation lattice for $(4,4,4)$}
Now, we continue with the remaining type vector, which is $(4,4,4)$. The resulting mutation lattice is significantly larger. We first draw its Hasse diagram, and then remark on some properties of the lattice. Finally, we list the values of the height functions on a fundamental domain. 

\begin{Rem}
    One can compute the mutation lattice for the type vector $(4,4,4)$. This lattice has $57$ elements, so we only draw its Hasse diagram in \Cref{Fig: 444 lattice}.
    The reader may notice that the lattice has rank $16$, meaning that we have drawn $17$ rows. To avoid confusion with conventions from combinatorics, we henceforth only refer to rows $1$ to $17$ from the top to the bottom. The lattice duality can be visualized as reflection along row $9$. Every element in this row is self-dual.
\end{Rem}

\begin{figure}
    \[
    \begin{tikzpicture}[]
  \matrix (m) [matrix of nodes, row sep=7mm, column sep=9mm, nodes={circle,draw, fill= black, inner sep=1pt}] 
  {
     & & . & & & \\   
     & & . & & & \\   
     & .& .& .& & \\   
     & .& .& .& & \\   
     & .& .& .& .& \\   
     & .& .& .& .& \\   
     .& .& .& .& .& .\\   
     & .& .& .& .& \\   
     & .& .& .& .& .\\   
     & .& .& .& .& \\   
     .& .& .& .& .& .\\   
     & .& .& .& .& \\   
     & .& .& .& .& \\   
     & .& .& .& & \\   
     & .& .& .& & \\   
     & & .& & & \\   
     & & .& & &  \\    
  };

  \draw[thick] (m-1-3) -- (m-2-3);
  \draw[thick] (m-2-3) -- (m-3-3); \draw[thick] (m-2-3) -- (m-3-2);
  \draw[thick] (m-2-3) -- (m-3-4);
  \draw[thick] (m-3-2) -- (m-4-2);
  \draw[thick] (m-3-2) -- (m-4-3);
  \draw[thick] (m-3-3) -- (m-4-2);
  \draw[thick] (m-3-3) -- (m-4-4);
  \draw[thick] (m-3-4) -- (m-4-3);
  \draw[thick] (m-3-4) -- (m-4-4);
  \draw[thick] (m-4-2) -- (m-5-2);
  \draw[thick] (m-4-3) -- (m-5-3);
  \draw[thick] (m-4-4) -- (m-5-4);
  \draw[thick] (m-4-2) -- (m-5-5);
  \draw[thick] (m-4-3) -- (m-5-5);
  \draw[thick] (m-4-4) -- (m-5-5);
  \draw[thick] (m-5-5) -- (m-6-5);
  \draw[thick] (m-5-5) -- (m-6-4);
  \draw[thick] (m-5-5) -- (m-6-3);
  \draw[thick] (m-5-5) -- (m-6-2);
  \draw[thick] (m-5-4) -- (m-6-4);
  \draw[thick] (m-5-3) -- (m-6-3);
  \draw[thick] (m-5-2) -- (m-6-2);
  \draw[thick] (m-6-2) -- (m-7-1);
  \draw[thick] (m-6-2) -- (m-7-2);
  \draw[thick] (m-6-2) -- (m-7-4);
  \draw[thick] (m-6-3) -- (m-7-1);
  \draw[thick] (m-6-3) -- (m-7-3);
  \draw[thick] (m-6-3) -- (m-7-5);
  \draw[thick] (m-6-4) -- (m-7-2);
  \draw[thick] (m-6-4) -- (m-7-3);
  \draw[thick] (m-6-4) -- (m-7-6);
  \draw[thick] (m-6-5) -- (m-7-4);
  \draw[thick] (m-6-5) -- (m-7-5);
  \draw[thick] (m-6-5) -- (m-7-6);
  \draw[thick] (m-7-1) -- (m-8-2);
  \draw[thick] (m-7-1) -- (m-8-3);
  \draw[thick] (m-7-2) -- (m-8-2);
  \draw[thick] (m-7-2) -- (m-8-4);
  \draw[thick] (m-7-3) -- (m-8-2);
  \draw[thick] (m-7-3) -- (m-8-5);
  \draw[thick] (m-7-4) -- (m-8-3);
  \draw[thick] (m-7-4) -- (m-8-4);
  \draw[thick] (m-7-5) -- (m-8-3);
  \draw[thick] (m-7-5) -- (m-8-5);
  \draw[thick] (m-7-6) -- (m-8-4);
  \draw[thick] (m-7-6) -- (m-8-5);
  \draw[thick] (m-8-2) -- (m-9-2);
  \draw[thick] (m-8-3) -- (m-9-3);
  \draw[thick] (m-8-4) -- (m-9-4);
  \draw[thick] (m-8-5) -- (m-9-5);
  \draw[thick] (m-8-2) -- (m-9-6);
  \draw[thick] (m-8-3) -- (m-9-6);
  \draw[thick] (m-8-4) -- (m-9-6);
  \draw[thick] (m-8-5) -- (m-9-6);
  \draw[thick] (m-9-2) -- (m-10-2);
  \draw[thick] (m-9-3) -- (m-10-3);
  \draw[thick] (m-9-4) -- (m-10-4);
  \draw[thick] (m-9-5) -- (m-10-5);
  \draw[thick] (m-9-6) -- (m-10-2);
  \draw[thick] (m-9-6) -- (m-10-3);
  \draw[thick] (m-9-6) -- (m-10-4);
  \draw[thick] (m-9-6) -- (m-10-5);
  \draw[thick] (m-10-2) -- (m-11-1);
  \draw[thick] (m-10-2) -- (m-11-2);
  \draw[thick] (m-10-2) -- (m-11-3);
  \draw[thick] (m-10-3) -- (m-11-1);
  \draw[thick] (m-10-3) -- (m-11-4);
  \draw[thick] (m-10-3) -- (m-11-5);
  \draw[thick] (m-10-4) -- (m-11-2);
  \draw[thick] (m-10-4) -- (m-11-4);
  \draw[thick] (m-10-4) -- (m-11-6);
  \draw[thick] (m-10-5) -- (m-11-3);
  \draw[thick] (m-10-5) -- (m-11-5);
  \draw[thick] (m-10-5) -- (m-11-6);
  \draw[thick] (m-11-1) -- (m-12-2);
  \draw[thick] (m-11-1) -- (m-12-3);
  \draw[thick] (m-11-2) -- (m-12-2);
  \draw[thick] (m-11-2) -- (m-12-4);
  \draw[thick] (m-11-3) -- (m-12-3);
  \draw[thick] (m-11-3) -- (m-12-4);
  \draw[thick] (m-11-4) -- (m-12-2);
  \draw[thick] (m-11-4) -- (m-12-5);
  \draw[thick] (m-11-5) -- (m-12-3);
  \draw[thick] (m-11-5) -- (m-12-5);
  \draw[thick] (m-11-6) -- (m-12-4);
  \draw[thick] (m-11-6) -- (m-12-5);
  \draw[thick] (m-12-2) -- (m-13-2);
  \draw[thick] (m-12-2) -- (m-13-5);
  \draw[thick] (m-12-3) -- (m-13-3);
  \draw[thick] (m-12-3) -- (m-13-5);
  \draw[thick] (m-12-4) -- (m-13-4);
  \draw[thick] (m-12-4) -- (m-13-5);
  \draw[thick] (m-12-5) -- (m-13-5);
  \draw[thick] (m-13-2) -- (m-14-2);
  \draw[thick] (m-13-3) -- (m-14-3);
  \draw[thick] (m-13-4) -- (m-14-4);
  \draw[thick] (m-13-5) -- (m-14-2);
  \draw[thick] (m-13-5) -- (m-14-3);
  \draw[thick] (m-13-5) -- (m-14-4);
  \draw[thick] (m-14-2) -- (m-15-2);
  \draw[thick] (m-14-2) -- (m-15-3);
  \draw[thick] (m-14-3) -- (m-15-2);
  \draw[thick] (m-14-3) -- (m-15-4);
  \draw[thick] (m-14-4) -- (m-15-3);
  \draw[thick] (m-14-4) -- (m-15-4);
  \draw[thick] (m-15-2) -- (m-16-3);
  \draw[thick] (m-15-3) -- (m-16-3);
  \draw[thick] (m-15-4) -- (m-16-3);
  \draw[thick] (m-16-3) -- (m-17-3);
\end{tikzpicture}
    \]
    
    \caption{The Hasse diagram of the mutation lattice for type $(4,4,4)$.}
    \label{Fig: 444 lattice}
\end{figure}

We list in the above table the height functions of all elements of the mutation lattice for $\gamma=(4,4,4)$, according to the following convention. The quiver has $12$ vertices, and all vertices are uniquely determined as $j \alpha_1 + i \alpha_2$ for $0 \leq j \leq 5$ and $0 \leq i \leq 1$. We therefore write the height function $h_C$ of a cut $C$ as a $2 \times 6$ matrix with $v_{i+1,j+1} = h_C(j \alpha_1 + i \alpha_2)$. The other values of $h_C$ are determined by $L_1$-equivariance, combined with \Cref{Pro: Height function values on L1}. For the reader's convenience, we write the values of mutable \emph{sinks} in boldface. We list the height functions according to the rows, and the entries of a row, read left to right, correspond to the vertices of the Hasse diagram, read left to right. 
\begingroup
\renewcommand{\arraystretch}{1.5} 
\begin{table}[]
    \centering
    \begin{tabular}{|c|c|}
    \hline 
    \multicolumn{2}{||c||}{Elements of the mutation lattice for $(4,4,4)$} \\
    \hline 
        Row & Elements  \\
        \hline 
        1 & $ \left( \begin{smallmatrix}
            0 & 1 & 2 & 3 & \mathbf{4} & 2 \\ 1 & 2 & 3 & 1 & 2 & 3
        \end{smallmatrix} \right)$ \\
        \hline 
        2 & $ \left( \begin{smallmatrix}
            0 & 1 & 2 & \mathbf{3} & 1 & 2 \\ 1 & 2 & \mathbf{3} & 1 & 2 & \mathbf{3}
        \end{smallmatrix} \right) $ \\
        \hline 
        3 & $ \left( \begin{smallmatrix}
            0 & 1 & 2 & \mathbf{3} & 1 & 2 \\ 1 & 2 & 0 & 1 & 2 & \mathbf{3}
        \end{smallmatrix} \right) \; 
        \left( \begin{smallmatrix}
            0 & 1 & 2 & 0 & 1 & 2 \\ 1 & 2 & \mathbf{3} & 1 & 2 & \mathbf{3}
        \end{smallmatrix} \right) \;
        \left( \begin{smallmatrix}
            0 & 1 & 2 & \mathbf{3} & 1 & 2 \\ 1 & 2 & \mathbf{3} & 1 & 2 & 0
        \end{smallmatrix} \right) $ \\
        \hline 
        4 & $ \left( \begin{smallmatrix}
            0 & 1 & 2 & 0 & 1 & 2 \\ 1 & \mathbf{2} & 0 & 1 & 2 & \mathbf{3}
        \end{smallmatrix} \right) \; 
        \left( \begin{smallmatrix}
            0 & 1 & 2 & \mathbf{3} & 1 & \mathbf{2} \\ 1 & 2 & 0 & 1 & 2 & 0
        \end{smallmatrix} \right) \;
        \left( \begin{smallmatrix}
            0 & 1 & 2 & 0 & 1 & 2 \\ 1 & 2 & \mathbf{3} & 1 & \mathbf{2} & 0
        \end{smallmatrix} \right) $ \\
        \hline 
        5 & $ \left( \begin{smallmatrix}
            0 & 1 & 2 & 0 & 1 & 2 \\ 1 & -1 & 0 & 1 & 2 & \mathbf{3}
        \end{smallmatrix} \right) \; 
        \left( \begin{smallmatrix}
            0 & 1 & 2 & \mathbf{3} & 1 & -1 \\ 1 & 2 & 0 & 1 & 2 & 0
        \end{smallmatrix} \right) \;
        \left( \begin{smallmatrix}
            0 & 1 & 2 & 0 & 1 & 2 \\ 1 & 2 & \mathbf{3} & 1 & -1 & 0
        \end{smallmatrix} \right)  \;
        \left( \begin{smallmatrix}
            0 & 1 & \mathbf{2} & 0 & 1 & \mathbf{2} \\ 1 & \mathbf{2} & 0 & 1 & \mathbf{2} & 0
        \end{smallmatrix} \right) $ \\
        \hline 
        6 & $ \left( \begin{smallmatrix}
            0 & 1 & \mathbf{2} & 0 & 1 & \mathbf{2} \\ 1 & -1 & 0 & 1 & \mathbf{2} & 0
        \end{smallmatrix} \right) \; 
        \left( \begin{smallmatrix}
            0 & 1 & \mathbf{2} & 0 & 1 & -1 \\ 1 & \mathbf{2} & 0 & 1 & \mathbf{2} & 0
        \end{smallmatrix} \right) \;
        \left( \begin{smallmatrix}
            0 & 1 & \mathbf{2} & 0 & 1 & \mathbf{2} \\ 1 & \mathbf{2} & 0 & 1 & -1 & 0
        \end{smallmatrix} \right)  \;
        \left( \begin{smallmatrix}
            0 & 1 & -1 & 0 & 1 & \mathbf{2} \\ 1 & \mathbf{2} & 0 & 1 & \mathbf{2} & 0
        \end{smallmatrix} \right) $ \\
        \hline 
        7 & $ \left( \begin{smallmatrix}
            0 & 1 & \mathbf{2} & 0 & 1 & -1 \\ 1 & -1 & 0 & 1 & \mathbf{2} & 0
        \end{smallmatrix} \right) \; 
        \left( \begin{smallmatrix}
            0 & 1 & \mathbf{2} & 0 & 1 & \mathbf{2} \\ 1 & -1 & 0 & 1 & -1 & 0
        \end{smallmatrix} \right) \;
        \left( \begin{smallmatrix}
            0 & 1 & \mathbf{2} & 0 & 1 & -1 \\ 1 & \mathbf{2} & 0 & 1 & -1 & 0
        \end{smallmatrix} \right) \; $ \\ & $
        \left( \begin{smallmatrix}
            0 & 1 & -1 & 0 & 1 & \mathbf{2} \\ 1 & -1 & 0 & 1 & \mathbf{2} & 0
        \end{smallmatrix} \right) \;
        \left( \begin{smallmatrix}
            0 & 1 & -1 & 0 & 1 & -1 \\ 1 & \mathbf{2} & 0 & 1 & \mathbf{2} & 0
        \end{smallmatrix} \right) \;
        \left( \begin{smallmatrix}
            0 & 1 & -1 & 0 & 1 & \mathbf{2} \\ 1 & \mathbf{2} & 0 & 1 & -1 & 0
        \end{smallmatrix} \right)$ \\
        \hline 
        8 & $ \left( \begin{smallmatrix}
            0 & 1 & \mathbf{2} & 0 & \mathbf{1} & -1 \\ 1 & -1 & 0 & 1 & -1 & 0
        \end{smallmatrix} \right) \;
        \left( \begin{smallmatrix}
            0 & 1 & -1 & 0 & 1 & -1 \\ \mathbf{1} & -1 & 0 & 1 & \mathbf{2} & 0
        \end{smallmatrix} \right) \;
        \left( \begin{smallmatrix}
            0 & \mathbf{1} & -1 & 0 & 1 & \mathbf{2} \\ 1 & -1 & 0 & 1 & -1 & 0
        \end{smallmatrix} \right) \; 
        \left( \begin{smallmatrix}
            0 & 1 & -1 & 0 & 1 & -1 \\ 1 & \mathbf{2} & 0 & \mathbf{1} & -1 & 0
        \end{smallmatrix} \right)$ \\
        \hline
        9 & $ \left( \begin{smallmatrix}
            0 & 1 & \mathbf{2} & 0 & -2 & -1 \\ 1 & -1 & 0 & 1 & -1 & 0
        \end{smallmatrix} \right) \;
        \left( \begin{smallmatrix}
            0 & 1 & -1 & 0 & 1 & -1 \\ -2 & -1 & 0 & 1 & \mathbf{2} & 0
        \end{smallmatrix} \right) \;
        \left( \begin{smallmatrix}
            0 & -2 & -1 & 0 & 1 & \mathbf{2} \\ 1 & -1 & 0 & 1 & -1 & 0
        \end{smallmatrix} \right) \; 
        \left( \begin{smallmatrix}
            0 & 1 & -1 & 0 & 1 & -1 \\ 1 & \mathbf{2} & 0 & -2 & -1 & 0
        \end{smallmatrix} \right) $ \\ & $
        \left( \begin{smallmatrix}
            0 & \mathbf{1} & -1 & 0 & \mathbf{1} & -1 \\ \mathbf{1} & -1 & 0 & \mathbf{1} & -1 & 0
        \end{smallmatrix} \right)$ \\
        \hline
        10 & $ \left( \begin{smallmatrix}
            0 & \mathbf{1} & -1 & 0 & -2 & -1 \\ \mathbf{1} & -1 & 0 & \mathbf{1} & -1 & 0
        \end{smallmatrix} \right) \;
        \left( \begin{smallmatrix}
            0 & \mathbf{1} & -1 & 0 & \mathbf{1} & -1 \\ -2 & -1 & 0 & \mathbf{1} & -1 & 0
        \end{smallmatrix} \right) \;
        \left( \begin{smallmatrix}
            0 & -2 & -1 & 0 & \mathbf{1} & -1 \\ \mathbf{1} & -1 & 0 & \mathbf{1} & -1 & 0
        \end{smallmatrix} \right) \; 
        \left( \begin{smallmatrix}
            0 & \mathbf{1} & -1 & 0 & \mathbf{1} & -1 \\ \mathbf{1} & -1 & 0 & -2 & -1 & 0
        \end{smallmatrix} \right) $ \\
        \hline
        11 & $ \left( \begin{smallmatrix}
            0 & \mathbf{1} & -1 & 0 & -2 & -1 \\ -2 & -1 & 0 & \mathbf{1} & -1 & 0
        \end{smallmatrix} \right) \;
        \left( \begin{smallmatrix}
            0 & -2 & -1 & 0 & -2 & -1 \\ \mathbf{1} & -1 & 0 & \mathbf{1} & -1 & 0
        \end{smallmatrix} \right) \;
        \left( \begin{smallmatrix}
            0 & \mathbf{1} & -1 & 0 & -2 & -1 \\ \mathbf{1} & -1 & 0 & -2 & -1 & 0
        \end{smallmatrix} \right) $ \\ & $
        \left( \begin{smallmatrix}
            0 & -2 & -1 & 0 & \mathbf{1} & -1 \\ -2 & -1 & 0 & \mathbf{1} & -1 & 0
        \end{smallmatrix} \right) \;
        \left( \begin{smallmatrix}
            0 & \mathbf{1} & -1 & 0 & \mathbf{1} & -1 \\ -2 & -1 & 0 & -2 & -1 & 0
        \end{smallmatrix} \right) \; 
        \left( \begin{smallmatrix}
            0 & -2 & -1 & 0 & \mathbf{1} & -1 \\ \mathbf{1} & -1 & 0 & -2 & -1 & 0
        \end{smallmatrix} \right) $ \\
        \hline
        12 & $ \left( \begin{smallmatrix}
            0 & -2 & -1 & 0 & -2 & -1 \\ -2 & -1 & 0 & \mathbf{1} & -1 & \mathbf{0}
        \end{smallmatrix} \right) \;
        \left( \begin{smallmatrix}
            0 & \mathbf{1} & -1 & \mathbf{0} & -2 & -1 \\ -2 & -1 & 0 & -2 & -1 & 0
        \end{smallmatrix} \right) \; 
        \left( \begin{smallmatrix} 
            0 & -2 & -1 & 0 & -2 & -1 \\ \mathbf{1} & -1 & \mathbf{0} & -2 & -1 & 0
        \end{smallmatrix} \right)  \;
        \left( \begin{smallmatrix}
            0 & -2 & -1 & 0 & \mathbf{1} & -1 \\ -2 & -1 & 0 & -2 & -1 & 0
        \end{smallmatrix} \right) $ \\
        \hline
        13 & $ \left( \begin{smallmatrix}
            0 & -2 & -1 & 0 & -2 & -1 \\ -2 & -1 & 0 & \mathbf{1} & -1 & -3
        \end{smallmatrix} \right) \;
        \left( \begin{smallmatrix}
            0 & \mathbf{1} & -1 & -3 & -2 & -1 \\ -2 & -1 & 0 & -2 & -1 & 0
        \end{smallmatrix} \right) \; 
        \left( \begin{smallmatrix} 
            0 & -2 & -1 & 0 & -2 & -1 \\ \mathbf{1} & -1 & -3 & -2 & -1 & 0
        \end{smallmatrix} \right)  \;
        \left( \begin{smallmatrix}
            0 & -2 & -1 & \mathbf{0} & -2 & -1 \\ -2 & -1 & \mathbf{0} & -2 & -1 & \mathbf{0}
        \end{smallmatrix} \right) $ \\
        \hline
        14 & $ \left( \begin{smallmatrix}
            0 & -2 & -1 & \mathbf{0} & -2 & -1 \\ -2 & -1 & \mathbf{0} & -2 & -1 & -3
        \end{smallmatrix} \right) \;
        \left( \begin{smallmatrix}
            0 & -2 & -1 & -3 & -2 & -1 \\ -2 & -1 & \mathbf{0} & -2 & -1 & \mathbf{0}
        \end{smallmatrix} \right) \; 
        \left( \begin{smallmatrix} 
            0 & -2 & -1 & \mathbf{0} & -2 & -1 \\ -2 & -1 & -3 & -2 & -1 & \mathbf{0}
        \end{smallmatrix} \right)  $ \\
        \hline 
        15 & $ \left( \begin{smallmatrix}
            0 & -2 & -1 & -3 & -2 & -1 \\ -2 & -1 & \mathbf{0} & -2 & -1 & -3
        \end{smallmatrix} \right) \;
        \left( \begin{smallmatrix}
            0 & -2 & -1 & \mathbf{0} & -2 & -1 \\ -2 & -1 & -3 & -2 & -1 & -3
        \end{smallmatrix} \right)
        \left( \begin{smallmatrix} 
            0 & -2 & -1 & -3 & -2 & -1 \\ -2 & -1 & -3 & -2 & -1 & \mathbf{0}
        \end{smallmatrix} \right)  $ \\
        \hline
        16 & $ \left( \begin{smallmatrix}
            0 & -2 & \mathbf{-1} & -3 & -2 & -1 \\ -2 & -1 & -3 & -2 & -1 & -3
        \end{smallmatrix} \right) $ \\
        \hline 
        17 & $ \left( \begin{smallmatrix}
            0 & -2 & -4 & -3 & -2 & -1 \\ -2 & -1 & -3 & -2 & -1 & -3
        \end{smallmatrix} \right) $ \\
        \hline 
    \end{tabular}
    \label{Tab: Big lattice}
\end{table}
\endgroup

For the reader's convenience, we briefly explain how to read the duality from the heights we have provided. 

\begin{Rem}
    Let $C$ be a cut, and as before write the heights at the twelve chosen points as 
    \[ \left(    \begin{smallmatrix}
        h_C(0) & h_C(\alpha_1) & h_C( 2 \alpha_1) & h_C( 3 \alpha_1)& h_C( 4 \alpha_1)& h_C( 5 \alpha_1) \\
        h_C( \alpha_2) & h_C(\alpha_1+ \alpha_2) & h_C( 2 \alpha_1+ \alpha_2) & h_C( 3 \alpha_1+ \alpha_2)& h_C( 4 \alpha_1+ \alpha_2)& h_C( 5 \alpha_1 + \alpha_2)
    \end{smallmatrix} \right).
    \]
    Then the heights $h_{C^\ast}$ with respect to the dual cut $C^\ast$ on the same twelve points are given by 
    \[ (-1) \cdot  \left(    \begin{smallmatrix}
        h_C(0) & h_C(5 \alpha_1) & h_C( 4 \alpha_1) & h_C( 3 \alpha_1)& h_C( 2 \alpha_1)& h_C(  \alpha_1) \\
        h_C( 4 \alpha_1 + \alpha_2) & h_C( 3 \alpha_1+ \alpha_2) & h_C( 2 \alpha_1+ \alpha_2) & h_C( 1 \alpha_1+ \alpha_2)& h_C(  \alpha_2)& h_C( 5 \alpha_1 + \alpha_2)
    \end{smallmatrix} \right).
    \]
    Applying this for example to the height for the cut in row $9$ which has the highest outdegree in the Hasse Diagram, we see that 
    \[ \left( \begin{smallmatrix}
            0 & 1 & -1 & 0 & 1 & -1 \\ 1 & -1 & 0 & 1 & -1 & 0
        \end{smallmatrix} \right) = (-1) \cdot \left( \begin{smallmatrix}
            0 & -1 & 1 & 0 & -1 & 1 \\ -1 & 1 & 0 & -1 & 1 & 0
        \end{smallmatrix} \right), \]
        so indeed the element is self-dual. 
\end{Rem}

As for type $(2,2,8)$, many of the resulting $2$-representation infinite algebras are isomorphic, despite arising from different elements in the mutation lattice. 

\begin{Rem}
    Denote the cut corresponding to the maximal element in the mutation lattice by $C_1$, and the leftmost cut in row $5$ by $C_2$. It is not difficult to see that the cut algebras are isomorphic, i.e.\ we have $(kQ/I)_{C_1} \simeq (kQ/I)_{C_2}$. To visualise this, one may take the cut $C_1$ and translate it by $\alpha_1 + \alpha_2$ to obtain $C_2$. One can see similarly that the $3$ first elements in row $5$ all define isomorphic cut algebras. 
\end{Rem}

\section*{Acknowledgement}
We thank Martin Herschend and David Ploog for many helpful discussions. We thank Yukari Ito for her lecture series at the International Conference on Representations of Algebras 21 (ICRA 21) and subsequent discussions. We thank the anonymous referee for many suggestions that helped improve the exposition of the paper and that prompted us to discover self-duality of mutation lattices. The second author was supported by a fellowship from the Wenner-Gren Foundations (grant WGF2022-0052).

\printbibliography

\end{document}